\documentclass[hidelinks,onefignum,onetabnum]{siamart220329}
\usepackage{lipsum}
\usepackage{amsfonts}
\usepackage{graphicx}
\usepackage{epstopdf}
\usepackage{algorithm}
\usepackage{hyperref}
\usepackage{algpseudocode}
\usepackage[percent]{overpic}
\usepackage{enumitem}
\usepackage{pifont}
\newcommand{\cmark}{\textcolor[rgb]{0,0.58,0}{\ding{51}}}%
\newcommand{\xmark}{\textcolor[rgb]{1,0,0}{\ding{55}}}
\makeatletter 
\newcounter{algorithmicH}
\let\oldalgorithmic\algorithmic
\renewcommand{\algorithmic}{
  \stepcounter{algorithmicH}
  \oldalgorithmic}
\renewcommand{\theHALG@line}{ALG@line.\thealgorithmicH.\arabic{ALG@line}}
\makeatother
\usepackage{comment}
\ifpdf
  \DeclareGraphicsExtensions{.eps,.pdf,.png,.jpg}
\else
  \DeclareGraphicsExtensions{.eps}
\fi

\newsiamremark{remark}{Remark}
\newsiamremark{hypothesis}{Hypothesis}
\crefname{hypothesis}{Hypothesis}{Hypotheses}
\newsiamthm{claim}{Claim}

\headers{\lowercase{mp}EDMD for measure-preserving dynamical systems}{M. J. Colbrook}

\title{\!{}The \lowercase{mp}EDMD\! Algorithm for Data-Driven Computations of Measure-Preserving Dynamical Systems\thanks{Submitted to the editors\today.
\funding{This work was funded by a FSMP Fellowship at École Normale Supérieure.}}}

\author{Matthew J. Colbrook\thanks{DAMTP, University of Cambridge, (\email{m.colbrook@damtp.cam.ac.uk}).}}

\usepackage{amsopn}

\usepackage{amsfonts,amssymb,amsmath}

\ifpdf
\hypersetup{
  pdftitle={mpEDMDcolbrook},
  pdfauthor={M. J. Colbrook}
}
\fi

\begin{document}

\maketitle

\begin{abstract}
Koopman operators globally linearize nonlinear dynamical systems and their spectral information is a powerful tool for the analysis and decomposition of nonlinear dynamical systems. However, Koopman operators are infinite-dimensional, and computing their spectral information is a considerable challenge. We introduce \textit{measure-preserving extended dynamic mode decomposition} (\texttt{mpEDMD}), the first truncation method whose eigendecomposition converges to the spectral quantities of Koopman operators for general measure-preserving dynamical systems. \texttt{mpEDMD} is a data-driven algorithm based on an orthogonal Procrustes problem that enforces measure-preserving truncations of Koopman operators using a general dictionary of observables. It is flexible and easy to use with any pre-existing DMD-type method, and with different types of data. We prove convergence of \texttt{mpEDMD} for projection-valued and scalar-valued spectral measures, spectra, and Koopman mode decompositions. For the case of delay embedding (Krylov subspaces), our results include the first convergence rates of the approximation of spectral measures as the size of the dictionary increases. We demonstrate \texttt{mpEDMD} on a range of challenging examples, its increased robustness to noise compared with other DMD-type methods, and its ability to capture the energy conservation and cascade of experimental measurements of a turbulent boundary layer flow with Reynolds number $> 6\times 10^4$ and state-space dimension $>10^5$.
\end{abstract}\vspace{-2mm}

\vspace{-2mm}

\vspace{-2mm}

\section{Introduction}
We consider dynamical systems whose state $\pmb{x}$ evolves over a state-space $\Omega\subseteq\mathbb{R}^d$ in discrete time-steps according to a function $F:\Omega \rightarrow \Omega$, i.e.,
\begin{equation}\setlength\abovedisplayskip{4pt}\setlength\belowdisplayskip{4pt}
\pmb{x}_{n+1} = F(\pmb{x}_n), \qquad n\geq 0,
\label{eq:DynamicalSystem} 
\end{equation}
for an initial condition $\pmb{x}_0\in\Omega$. We assume that \eqref{eq:DynamicalSystem} is \textit{measure-preserving} with respect to a positive measure $\omega$ on $\Omega$. 
This covers many systems of interest such as Hamiltonian flows~\cite{arnold1989mathematical}, geodesic flows~\cite{dubrovin2012modern}, Bernoulli schemes~\cite{shields1973theory}, physical systems in equilibrium~\cite{hill1986introduction}, and ergodic systems~\cite{walters2000introduction}. Moreover, many dynamical systems admit invariant measures~\cite{kryloff1937theorie} or have measure-preserving post-transient behavior~\cite{mezic2005spectral}. With the arrival of big data and machine learning, the numerical analysis of dynamical systems is currently undergoing a data-driven renaissance \cite{kantz2004nonlinear,schmidt2009distilling,brunton2019data,brunton2016discovering,burov2021kernel,zhao2016analog,guo2019data}. In many applications, the system's dynamics may be too complicated to describe analytically, or we may only have access to incomplete knowledge of its evolution. Therefore, we do not assume explicit knowledge of the function $F$. Instead, we assume that we have access to discrete-time snapshots of this system, i.e.,
\begin{equation}\setlength\abovedisplayskip{4pt}\setlength\belowdisplayskip{4pt}
\label{snapshot_data}
\{\pmb{x}^{(m)},\pmb{y}^{(m)}\}_{m=1}^M\quad \text{such that}\quad \pmb{y}^{(m)}=F(\pmb{x}^{(m)}),\quad m=1,...,M.
\end{equation}
Suitable data could be collected from one long time trajectory, corresponding to $\pmb{x}^{(m)}=F^{m-1}(\pmb{x}_0)$ ($m-1$ applications of $F$), or from multiple shorter trajectories. We approximate spectral quantities of \eqref{eq:DynamicalSystem} using the data \eqref{snapshot_data}.  We do this via a new Galerkin discretization that allows us to prove convergence results and maintain the measure-preserving nature of the dynamical system. Convergence of spectral quantities is crucial for recovering the correct dynamical behavior of \eqref{eq:DynamicalSystem}, and preserving the measure is crucial for improved qualitative and long-time behavior (see~\cref{fig:turbulence2}).

A popular and powerful framework for the data-driven study of dynamical systems is provided by Koopman operators. First introduced by Koopman and von Neumann in the 1930's~\cite{koopman1931hamiltonian,koopman1932dynamical}, Koopman operators allow a global linearization of~\eqref{eq:DynamicalSystem} using the space of scalar functions on $\Omega$~\cite{mezicAMS}. Their increasing popularity, known as ``Koopmanism''~\cite{budivsic2012applied}, has led to thousands of articles over the last decade~\cite{brunton2021modern}. Popular applications include
epidemiology~\cite{proctor2015discovering},
finance~\cite{mann2016dynamic}, fluid dynamics~\cite{schmid2010dynamic,rowley2009spectral,mezic2013analysis}, neuroscience~\cite{brunton2016extracting},
molecular dynamics~\cite{klus2018data,schwantes2013improvements}, 
and robotics~\cite{berger2015estimation,bruder2019modeling}.

Since~\cref{eq:DynamicalSystem} is measure-preserving, its Koopman operator, $\mathcal{K}$, is defined by
\begin{equation}\setlength\abovedisplayskip{4pt}\setlength\belowdisplayskip{4pt}
[\mathcal{K}g](\pmb{x}) = (g\circ F)(\pmb{x}), \qquad \pmb{x}\in\Omega, \qquad g\in L^2(\Omega,\omega), 
\label{eq:KoopmanOperator} 
\end{equation}
and is an isometry on the space $L^2(\Omega,\omega)$ with inner product $\langle \cdot,\cdot \rangle$ and norm $\|\cdot\|$. 
The functions $g$ are also known as `observables' because they indirectly measure the state of the dynamical system. 
The Koopman operator transforms the nonlinear dynamics in the state variable $\pmb{x}$ into equivalent linear dynamics in the observables $g$. Hence, the behavior of the dynamical system \eqref{eq:DynamicalSystem} is determined by the spectral information of $\mathcal{K}$ (e.g., see~\cref{KMD_rem}). Obtaining linear representations for nonlinear systems has the potential to revolutionize our ability to predict and control these systems. However, there is price to pay for this linearisation -- $\mathcal{K}$ acts on an \textit{infinite-dimensional} space. Therefore, its spectral information can be far more complicated and more difficult to compute than that of a finite matrix~\cite{webb2021spectra,colbrook2020foundations,SCI_big}. Common challenges include: computing spectral measures and continuous spectra~\cite{brunton2017chaos,lusch2018deep}; spectral pollution~\cite{williams2015data}, where discretizations cause spurious eigenvalues (and hence spurious coherent structures) to appear~\cite{lewin2010spectral}; and, in the context of this paper, preserving the isometric nature of $\mathcal{K}$. This last issue is often crucial to ensuring that approximations retain the physical properties of the original system (e.g., energy conservation). Structure-preserving algorithms have a rich history in geometric integration~\cite{3-540-30663-3} and have recently come to the fold in data-driven problems~\cite{hesthaven2022reduced,
celledoni2021structure,karniadakis2021physics,greydanus2019hamiltonian,
hernandez2021structure}.

\begin{table}[t]
\renewcommand{\arraystretch}{1.4}
\centering\vspace{-2mm}
\begin{tabular}{|l||c|c|c|c|}
\cline{2-5}
\multicolumn{1}{c|}{}& \!{}\textbf{DMD}\!{} & \!{}\textbf{EDMD}\!{} & \!{}\textbf{piDMD}\!{} & \!{}\textbf{mpEDMD}\!{}\\ \cline{2-5}\noalign{\vskip\doublerulesep
         \vskip-\arrayrulewidth}\cline{1-5}
\!{}Aux. SVD matrices\!{}&n/a & n/a &\!$YX^*=V_1 S V_2^*$\! & \!$G^{-\frac{1}{2}}A^*G^{-\frac{1}{2}}=U_1\Sigma U_2^*$\!\\\hline
\!{}Koopman matrix\!{}& \!$(YX^{\dagger})^\top$\! & $G^{\dagger}A$ & $V_2 V_1^*$ & $G^{-\frac{1}{2}}U_2U_1^* G^{\frac{1}{2}}$\\\hline
\!{}Nonlinear dictionary\!{}&\xmark & \cmark & \xmark & \cmark\\\hline
\!{}Conv. spec. meas.\!{}& \xmark&\xmark &\xmark &\cmark\\\hline
\!{}Conv. spectra\!{}& \xmark&\xmark &\xmark &\cmark\\\hline
\!{}Conv. KMD\!{}&\xmark & \cmark & \xmark & \cmark\\\hline
\!{}Measure-preserving\!{}&\xmark &\xmark & \xmark/\cmark$^\ddagger$ &\cmark\\ \hline
\end{tabular}\renewcommand{\arraystretch}{0.98}
\label{comparison_table}
\caption{Comparisons of Galerkin discretizations discussed in this paper. $X=[\pmb{x}^{(1)} \,\,\cdots\,\, \pmb{x}^{(M)}],Y=[\pmb{y}^{(1)}\,\, \cdots\,\, \pmb{y}^{(M)} ]\in\mathbb{C}^{d\times {M}}$ are matrices of the snapshots (linear dictionary) and it is common to combine DMD with a truncated SVD (see~\cref{sec:EDMD_recap}). $G=\Psi_X^*W\Psi_X$ and $A=\Psi_X^*W\Psi_Y$, where $\Psi_X$, $\Psi_Y$ are given in~\eqref{psidef} and $W=\mathrm{diag}(w_1,\ldots,w_M)$ is a diagonal matrix of quadrature weights.\\$^\ddagger$Note: {piDMD is measure-preserving only if $XX^*$ and $W$ are multiples of the identity, $I$.}}\vspace{-8mm}
\end{table}

Most existing approaches to approximate $\mathcal{K}$ and its spectral properties are based on dynamic mode decomposition (DMD)~\cite{schmid2010dynamic,rowley2009spectral,tu2014dynamic,kutz2016dynamic} 
or its variants~\cite{chen2012variants,jovanovic2014sparsity,proctor2016dynamic,colbrook2021rigorous}. DMD approximates $\mathcal{K}$ via a best-fit linear model of~\cref{eq:DynamicalSystem} that advances spatial measurements from one time step to the next. However, DMD is based on linear observables, which are not rich enough for many nonlinear systems. To overcome this,~\cite{williams2015data} introduced extended DMD (EDMD), a Galerkin approximation of $\mathcal{K}$ acting on a dictionary of (nonlinear) observables (see~\cref{sec:EDMD_recap}). As the number of snapshots, $M$, increases, the eigenvalues computed by EDMD correspond to the so-called finite section method~\cite{bottcher1983finite}. Since the finite section method can suffer from spectral pollution, 
spectral pollution is also a major concern for EDMD~\cite{williams2015data}. Moreover, as the dictionary becomes richer, the spectral measures of EDMD do not typically converge weakly to that of $\mathcal{K}$ (see~\cref{sec:shift_example} for a generic example). Finally, although the Koopman mode decomposition (KMD) provided by EDMD converges in an appropriate sense (in contrast to DMD), it is not measure-preserving, and this is of serious concern in many applications (see \cref{sec:turbulent_flow} for a real-world example). 

Recently, Baddoo and co-authors~\cite{baddoo2021physics} introduced physics-informed DMD (piDMD), which enforces symmetry constraints on the DMD approximation. For conservative systems, piDMD enforces the DMD matrix to be orthogonal. However, piDMD uses linear observables and implicitly assumes that these are orthonormal in $L^2(\Omega,\omega)$, which does not hold (or may not even be possible after reweighting) for many dynamical systems. Moreover, no convergence results are known for piDMD.

Motivated by the relative advantages of EDMD and piDMD, we introduce a new approximation of $\mathcal{K}$ that is measure-preserving and that converges to the correct spectral information. Our method uses an orthogonal Procrustes problem using \textit{general dictionaries} and \textit{nonlinear measurements}, and we call our algorithm measure-preserving EDMD (\texttt{mpEDMD}).~\cref{comparison_table} compares DMD, EDMD, piDMD and \texttt{mpEDMD}, and highlights some of the benefits of \texttt{mpEDMD}. It is precisely the fact that \texttt{mpEDMD} corresponds to a normal truncation that allows convergence results (e.g., see proof of~\cref{thm_weak_conv}). Our contributions include:
\begin{itemize}[leftmargin=0.4cm]
	\item We introduce \texttt{mpEDMD} to deal with generic measure-preserving systems.\footnote{For example, we do not assume in this paper that the system is ergodic.} \texttt{mpEDMD} is simple and easy to use with any pre-existing DMD-type method, it is measure-preserving, and it can be used with a range of different data structures and acquisition methods (e.g., single trajectories or multiple trajectories).
	\item We prove convergence of \texttt{mpEDMD} for various spectral quantities of interest, summarized in~\cref{approx_table}. Our results include weak convergence\footnote{This means convergence after integrating against a Lipschitz continuous test function on the unit circle (where the measures are supported)~\cite[Ch.~1]{billingsley2013convergence}. The computation of spectral measures poses a serious numerical challenge~\cite{colbrook2020} and can only ever be done in this weak sense~\cite{colbrook2019computing}. For example, the spectral type of $\mathcal{K}$ is well-known to be sensitive to arbitrarily small perturbations.} of projection-valued and scalar-valued spectral measures, convergence of spectra (including spectral inclusion and ways to deal with spectral pollution) and convergence of KMDs in $L^2(\Omega,\omega)$.	\texttt{mpEDMD} is the first truncation method whose eigendecomposition converges to these spectral quantities for general measure-preserving dynamical systems. \cref{delay} is the first result in the literature on convergence rates of the approximation of spectral measures as the size of the dictionary increases.
	\item We demonstrate our convergence results and the use of \texttt{mpEDMD} on several examples, including numerically simulated data and experimental data. These examples also demonstrate the increased robustness of \texttt{mpEDMD} to noise compared with other DMD-type methods, and the ability to deal with difficult problems such as capturing the energy conservation and statistics of a turbulent boundary layer flow.
\end{itemize}
\vspace{1mm}

\begin{table}[t]
\renewcommand{\arraystretch}{1.4}
\centering\vspace{-2mm}
\begin{tabular}{|l|l|l|}
\hline
\multicolumn{1}{|c|}{\!\textbf{Spectral quantity}\!}& \multicolumn{1}{|c|}{\!\textbf{Approximation}\!}& \multicolumn{1}{|c|}{\!\textbf{Convergence results}\!} \\ \hline\hline
\! Spec. measure $\mathcal{E}$\!
& \!$\mathcal{E}_{N,M}\!=\!\sum_{j=1}^{N}v_jv_j^*G\delta_{\lambda_j}$\!
&  \! Thm. \ref{thm_weak_conv} and \ref{thm_weak_conv2}\!                            \\ \hline
\! Spec. measures $\mu_g$\!
& \!$\smash{\mu_{\pmb{g}}^{(N,M)}}\!\!=\!\sum_{j=1}^N|v_j^*G\pmb{g}|^2\delta_{\lambda_j}$\!\! &\! Thm. \ref{thm_weak_conv_scalar}, Cor. \ref{mug_conv} and \ref{delay}\!                              \\ \hline
\! Approx. pt. spec. \!$\sigma_{\mathrm{ap}}(\mathcal{K})$\!\!
& \!$\{\lambda_1,\ldots,\lambda_N\}$\!
& \! Thm.~\ref{lemma_spectra_conv} and Eq.~\cref{residual_bound} \!                             \\ \hline
\! Koop. mode decomp. \!&
\!{}Eq.~\cref{LIP_pred} for
$g(\pmb{x}_n).$\! &
\! Lemma \ref{lemma_SOT_conv} and Rem. \ref{KMD_rem}\!\\ \hline
\end{tabular}\renewcommand{\arraystretch}{0.98}
\label{approx_table}
\caption{Lookup table of the approximated spectral quantities using \texttt{mpEDMD} (\cref{alg:mp_EDMD}) and the relevant convergence results of this paper. The vectors $\smash{\{v_j\}_{j=1}^N}$ denote the eigenvectors of $\mathbb{K}$ with corresponding eigenvalues $\smash{\{\lambda_j\}_{j=1}^N}$, and $\delta_{\lambda_j}$ denotes a Dirac delta distribution centered at $\lambda_j$.}\vspace{-8mm}
\end{table}

The remainder of the paper is organized as follows. In~\cref{sec:math_prelims} we introduce various concepts and notation, and motivate the computation of spectral properties of $\mathcal{K}$.~\Cref{sec:EDMD_recap} recalls the basics of EDMD. In~\cref{sec:mpEDMDD} we introduce \texttt{mpEDMD} and prove its convergence properties in~\cref{sec:conv_theory}. A range of numerical examples are presented in~\cref{sec:numerical_examples} and we conclude in~\cref{sec:conclusion}. General purpose code for \texttt{mpEDMD} and the examples of this paper can be found at \textcolor[rgb]{0,0,1}{\url{https://github.com/MColbrook/Measure-preserving-Extended-Dynamic-Mode-Decomposition}}.

\section{Mathematical preliminaries}\label{sec:math_prelims}

Here we provide the background material on spectral measures and approximate point spectra needed to understand later sections.

\subsection{Spectral measures and unitary extensions of $\mathcal{K}$}\label{sec:spec_meas_def_jkhkhjk}
Spectral measures provide a way of diagonalizing normal operators. However, a Koopman operator that is an isometry does not necessarily commute with its adjoint - a famous example is the Koopman operator of the tent map. Despite this, a Koopman operator $\mathcal{K}:L^2(\Omega,\omega)\rightarrow L^2(\Omega,\omega)$ of a measure-preserving dynamical system has a unitary extension $\mathcal{K}'$ defined on an extended Hilbert space $\mathcal{H}'$ with $L^2(\Omega,\omega)\subset\mathcal{H}'$ \cite[Chapter I]{nagy2010harmonic}. Such an extension is not unique, but it still allows us to understand the spectral information of $\mathcal{K}$ by considering $\mathcal{K}'$, which is a normal operator. After projecting back onto $L^2(\Omega,\omega)$, the measure is independent of the extension (\cref{spec_meas_indep}). If $F$ is invertible and measure-preserving, $\mathcal{K}$ is unitary and we can simply take $\mathcal{K}'=\mathcal{K}$ and $\mathcal{H}'=L^2(\Omega,\omega)$.

The spectral theorem for a normal matrix $B\in \mathbb{C}^{n\times n}$, i.e., $B^*B = BB^*$, states that there exists an orthonormal basis of eigenvectors $v_1,\dots,v_n$ for $\mathbb{C}^n$ such that
\begin{equation}\setlength\abovedisplayskip{4pt}\setlength\belowdisplayskip{4pt}\label{eqn:disc_decomp}
v = \left(\sum_{k=1}^n v_kv_k^*\right)v, \quad v\in\mathbb{C}^n \qquad\text{and}\qquad Bv = \left(\sum_{k=1}^n\lambda_k v_kv_k^*\right)v, \quad v\in\mathbb{C}^n,
\end{equation}
where $\lambda_1,\ldots,\lambda_n$ are eigenvalues of $B$, i.e., $Bv_k = \lambda_kv_k$ for $1\leq k\leq n$. In other words, the projections $v_kv_k^*$ simultaneously decompose the space $\mathbb{C}^n$ and diagonalize the operator $B$. This intuition carries over to the infinite-dimensional setting of this paper, by replacing $v\in\mathbb{C}^n$ by $f\in\mathcal{H}'$, and $B$ by a normal operator $\mathcal{K}'$. However, if $\mathcal{K}'$ has non-empty continuous spectrum, then the eigenvectors of $\mathcal{K}'$ do not form a basis for $\mathcal{H}'$ or diagonalize $\mathcal{K}'$. Instead, the spectral theorem for normal operators states that the projections $v_kv_k^*$ in~\eqref{eqn:disc_decomp} can be replaced by a projection-valued measure $\mathcal{E}'$ supported on the spectrum of $\mathcal{K}'$~\cite[Thm.~VIII.6]{reed1972methods}. In our setting, $\mathcal{K}'$ is unitary and hence its spectrum is contained inside the unit circle $\mathbb{T}$. The measure $\mathcal{E}'$ assigns an orthogonal projector to each Borel measurable subset of $\mathbb{T}$ such that
$$\setlength\abovedisplayskip{4pt}\setlength\belowdisplayskip{4pt}
f=\left(\int_\mathbb{T} d\mathcal{E}'(\lambda)\right)f \qquad\text{and}\qquad \mathcal{K}'f=\left(\int_\mathbb{T} \lambda\,d\mathcal{E}'(\lambda)\right)f, \qquad f\in\mathcal{H}'.
$$
Analogous to~\eqref{eqn:disc_decomp}, $\mathcal{E}'$ decomposes $\mathcal{H}'$ and diagonalizes the operator $\mathcal{K}'$. For example, if $U\subset\mathbb{T}$ contains only discrete eigenvalues of $\mathcal{K}'$ and no other types of spectra, then $\mathcal{E}'(U)$ is simply the spectral projector onto the invariant subspace spanned by the corresponding eigenfunctions. More generally, $\mathcal{E}'$ decomposes elements of $\mathcal{H}'$ along the discrete and continuous spectrum of $\mathcal{K}'$~[Section 2]\cite{colbrook2020}.

\begin{proposition}
\label{spec_meas_indep}
Let $\mathcal{P}$ denote the orthogonal projection from $\mathcal{H}'$ to $L^2(\Omega,\omega)$ and define $\mathcal{E}=P\mathcal{E}'P^*$. Then $\mathcal{E}$ is independent of the choice of unitary extension.
\end{proposition}

\begin{proof}
For any $g,h\in L^2(\Omega,\omega)$ and Borel measurable set $U\subset\mathbb{T}$, $
\langle \mathcal{E}(U)g,h\rangle=\langle \mathcal{E}'(U)g,h\rangle_{\mathcal{H}'}.$ Hence, it is enough to show that the scalar-valued measures $\mu_{g,h}$, $\mu_{g,h}(U):=\langle\mathcal{E}'(U)g,h\rangle_{\mathcal{H}'}$, are independent of the choice of $\mathcal{K}'$. For $n\in\mathbb{Z}$,
$$\setlength\abovedisplayskip{4pt}\setlength\belowdisplayskip{4pt}
\int_{\mathbb{T}}\lambda^n \,d\mu_{g,h}(\lambda)=\langle (\mathcal{K}')^ng,h \rangle_{\mathcal{H}'}=\begin{cases}
\langle \mathcal{K}^ng,h \rangle,&\quad \text{if }n\geq 0,\\
\langle g,\mathcal{K}^{-n}h \rangle,&\quad \text{otherwise}.
\end{cases}
$$
Since $\mu_{g,h}$ is determined by these moments, the result follows.
\end{proof}

\cref{spec_meas_indep} shows that the choice of unitary extension is immaterial. Henceforth, we dispense with the extension $\mathcal{K}'$, and call $\mathcal{E}$ the spectral measure of $\mathcal{K}$. 
The approximation of $\mathcal{E}$ plays a critical role in many applications. For example, in model reduction, the approximate spectral projections provide a low order model \cite{mezic2004comparison,mezic2005spectral}. A related example is the KMD in~\cref{KMD_rem}. Furthermore, the decomposition of $\mathcal{E}$ into atomic and continuous parts often characterizes a dynamical system. For example, suppose $F$ is measure-preserving and bijective, and $\omega$ is a probability measure. Then, the dynamical system is: (1) ergodic if and only if $\lambda=1$ is a simple eigenvalue of $\mathcal{K}$, (2) weakly mixing if and only if $\lambda=1$ is a simple eigenvalue of $\mathcal{K}$ and there are no other eigenvalues, and (3) mixing if $\lambda=1$ is a simple eigenvalue of $\mathcal{K}$ and $\mathcal{K}$ has absolutely continuous spectrum on $\mathrm{span}\{1\}^\perp$~\cite{halmos2017lectures}. Different spectral types also have interpretations in various applications such as fluid mechanics~\cite{mezic2013analysis}, anomalous transport~\cite{zaslavsky2002chaos}, and the analysis of invariants/exponents of trajectories~\cite{kantz2004nonlinear}.

Given an observable $g\in L^2(\Omega,\omega)$ of interest that is normalized to have $\|g\| = 1$, the spectral measure of $\mathcal{K}$ with respect to $g$ is a probability measure defined as $\mu_{g}(U):=\langle\mathcal{E}(U)g,g\rangle$, where $U\subset\mathbb{T}$ is a Borel measurable set~\cite{reed1972methods}. The proof of~\cref{spec_meas_indep} shows that the moments of the measure $\mu_g$ are the correlations $\langle \mathcal{K}^n g,g\rangle$ and $\langle g,\mathcal{K}^n g\rangle$ for $n\in\mathbb{Z}_{\geq 0}$. For example, if our system corresponds to the dynamics on an attractor, these statistical properties allow comparison of complex dynamics~\cite{mezic2004comparison}. More generally, the spectral measure of $\mathcal{K}$ with respect to almost every $g\in L^2(\Omega,\omega)$ is a signature for the forward-time dynamics of~\eqref{eq:DynamicalSystem}. This is because $\mu_g$ completely determines $\mathcal{K}$ when $g$ is cyclic, i.e., when the closure of $\mathrm{span}\{g,\mathcal{K}g,\mathcal{K}^2g,\ldots\}$ is $L^2(\Omega,\omega)$, and almost every $g$ is cyclic. If $g$ is not cyclic, then $\mu_g$ only determines the action of $\mathcal{K}$ on the closure of $\mathrm{span}\{g,\mathcal{K}g,\mathcal{K}^2g,\dots\}$, which can still be useful if one is interested in particular observables. The choice of $g$ is up to the practitioner and application.

\subsection{Approximate point spectra} Since $\mathcal{K}$ is an isometry, any eigenvalue of $\mathcal{K}$ must lie in $\mathbb{T}$. The approximate point spectrum generalizes the notion of eigenvalues,
$$\setlength\abovedisplayskip{4pt}\setlength\belowdisplayskip{4pt}
\sigma_{\mathrm{ap}}(\mathcal{K})=\{\lambda\in\mathbb{C}:\exists\{g_n\}\subset L^2(\Omega,\omega)\text{ such that }\|g_n\|=1,\lim_{n\rightarrow\infty}\|(\mathcal{K}-\lambda)g_n\|=0\}.
$$
We can approximate $\sigma_{\mathrm{ap}}(\mathcal{K})$ using the eigenvalues computed by~\cref{alg:mp_EDMD}. If $\mathcal{K}$ is unitary, then $\sigma_{\mathrm{ap}}(\mathcal{K})=\sigma(\mathcal{K})\subset\mathbb{T}$. Otherwise, $\sigma_{\mathrm{ap}}(\mathcal{K})=\mathbb{T}$ and $\sigma(\mathcal{K})$ is the closed unit disc in $\mathbb{C}$. Any observable $g$ with $\|g\|=1$ and $\lambda\in\mathbb{C}$ such that $\|(\mathcal{K}-\lambda)g\|\leq\epsilon$ is known as ($\epsilon$-)approximate eigenfunction. Such observables are important for the dynamical system~\cref{eq:DynamicalSystem} since $\mathcal{K}^ng=\lambda^n g+\mathcal{O}(n\epsilon)$. In other words, $\lambda$ describes the coherent oscillation and decay/growth of the observable $g$ with time. We can verify and compute ($\epsilon$-)approximate eigenfunctions by approximating residuals in~\eqref{residual_bound}.\footnote{Readers familiar with the notion of pseudospectra 
will recognize the notion of approximate eigenfunctions. Pseudospectra are needed for generic non-normal Koopman operators since the transient behavior of the system can differ greatly from the behavior at large times. 
However, in our case, the Koopman operator is an isometry and hence $\sigma_\epsilon(\mathcal{K})=\{\lambda\in\mathbb{C}:\mathrm{dist}(\lambda,\sigma(\mathcal{K}))\leq \epsilon\}$ so that pseudospectra are not needed. In particular, the spectrum is stable to perturbations.}

The approximate eigenfunctions and $\sigma_{\mathrm{ap}}(\mathcal{K})$ encode lots of information about the underlying dynamical system~\eqref{eq:DynamicalSystem}~\cite{mezicAMS}. For example, the level sets of certain eigenfunctions determine the invariant manifolds~\cite{mezic2015applications} (e.g.,~\cref{fig:pendulum1}) and isostables~\cite{mauroy2013isostables}, and the global stability of equilibria and ergodic partitions can be characterized by approximate eigenfunctions and $\sigma_{\mathrm{ap}}(\mathcal{K})$~\cite{mauroy2016global,budivsic2012applied}.

\section{Extended Dynamic Mode Decomposition}
\label{sec:EDMD_recap}
Given a dictionary of functions $\{\psi_1,\ldots,\psi_{N}\}\subset L^2(\Omega,\omega)$, EDMD~\cite{williams2015data} constructs a matrix $\mathbb{K}_{\mathrm{EDMD}}\in\mathbb{C}^{N\times N}$ from the snapshot data \eqref{snapshot_data} that approximates the action of $\mathcal{K}$ on the finite-dimensional subspace ${V}_{{N}}=\mathrm{span}\{\psi_1,\ldots,\psi_{N}\}$. The choice of the dictionary is up to the user, with some common hand-crafted choices given in~\cite[Table 1]{williams2015data}. When the state-space dimension $d$ is large, it is beneficial to use a data-driven choice of dictionary \cite{kutz2016dynamic,williams2015kernel}, which can be verified aposteri to capture the relevant dynamics via residual techniques \cite{colbrook2021rigorous}. We define the vector-valued function or ``quasimatrix'' $\Psi$ via
$$\setlength\abovedisplayskip{4pt}\setlength\belowdisplayskip{4pt}
\Psi(\pmb{x})=\begin{bmatrix}\psi_1(\pmb{x}) & \cdots& \psi_{{N}}(\pmb{x}) \end{bmatrix}\in\mathbb{C}^{1\times {N}}.
$$
Any function $g\in V_{N}$ can then be written as $g(\pmb{x})=\sum_{j=1}^{N}\psi_j(\pmb{x})g_j=\Psi(\pmb{x})\,\pmb{g}$ for some vector of constant coefficients $\pmb{g}\in\mathbb{C}^{N}$. It follows from \eqref{eq:KoopmanOperator} that
$$\setlength\abovedisplayskip{4pt}\setlength\belowdisplayskip{4pt}
[\mathcal{K}g](\pmb{x})=\Psi(\pmb{x})(\mathbb{K}_{\mathrm{EDMD}}\,\pmb{g})+R(\pmb{g},\pmb{x}),\quad R(\pmb{g},\pmb{x}):={\Psi(F(\pmb{x}))\,\pmb{g}-\Psi(\pmb{x})(\mathbb{K}_{\mathrm{EDMD}}\,\pmb{g})}.
$$
Typically, the subspace $V_{N}$ generated by the dictionary is not an invariant subspace of $\mathcal{K}$, and hence there is no choice of $\mathbb{K}_{\mathrm{EDMD}}$ that makes the error $R(\pmb{g},\pmb{x})$ zero for all choices of $g\in V_N$ and $\pmb{x}\in\Omega$. Instead, it is natural to select $\mathbb{K}_{\mathrm{EDMD}}$ as a solution of
\begin{equation}\setlength\abovedisplayskip{4pt}\setlength\belowdisplayskip{4pt}
\underset{B\in\mathbb{C}^{N\times N}}{\mathrm{argmin}} \left\{\int_\Omega \max_{\|\pmb{g}\|_2=1}|R(\pmb{g},\pmb{x})|^2\,d\omega(\pmb{x})=\int_\Omega \left\|\Psi(F(\pmb{x})) - \Psi(\pmb{x})B\right\|_2^2\,d\omega(\pmb{x})\right\}.
\label{eq:ContinuousLeastSquaresProblem}
\end{equation} 
Here, $\|\cdot\|_2$ denotes the standard Euclidean norm of a vector. Given a finite amount of snapshot data, we cannot directly evaluate the integral in~\eqref{eq:ContinuousLeastSquaresProblem}. Instead, we approximate it via a quadrature rule by treating the data points $\{\pmb{x}^{(m)}\}_{m=1}^{M}$ as quadrature nodes with weights $\{w_m\}_{m=1}^{M}$. Note that in the original definition of EDMD, $\omega$ is a probability measure and the quadrature weights are $w_m=1/M$. General weights are an important consideration when we sample according to a measure different to $\omega$ or if we are free to chose $\{\pmb{x}^{(m)}\}_{m=1}^{M}$ according to a high-order quadrature rule. The discretized version of~\eqref{eq:ContinuousLeastSquaresProblem} is
\begin{equation}\setlength\abovedisplayskip{4pt}\setlength\belowdisplayskip{4pt}
\label{EDMD_opt_prob2}
\mathbb{K}_{\mathrm{EDMD}}\in\underset{B\in\mathbb{C}^{N\times N}}{\mathrm{argmin}}\sum_{m=1}^{M} w_m\left\|\Psi(\pmb{y}^{(m)})-\Psi(\pmb{x}^{(m)})B\right\|_2^2.
\end{equation}
For notational convenience, we define the following two matrices
\begin{equation}\setlength\abovedisplayskip{4pt}\setlength\belowdisplayskip{4pt}
\begin{split}
\Psi_X=\begin{pmatrix}
\Psi(\pmb{x}^{(1)})\\
\vdots\\
\Psi(\pmb{x}^{(M)})
\end{pmatrix}\in\mathbb{C}^{M\times N},\quad
\Psi_Y=\begin{pmatrix}
\Psi(\pmb{y}^{(1)})\\
\vdots \\
\Psi(\pmb{y}^{(M)})
\end{pmatrix}\in\mathbb{C}^{M\times N},
\label{psidef}
\end{split}
\end{equation}
and let $W=\mathrm{diag}(w_1,\ldots,w_M)$ be the diagonal weight matrix of the quadrature rule. We define the Gram-matrix $G=\Psi_X^*W\Psi_X$ and the matrix $A=\Psi_X^*W\Psi_Y$. Letting `$\dagger$' denote the pseudoinverse, a solution to~\eqref{EDMD_opt_prob2} is
$$\setlength\abovedisplayskip{4pt}\setlength\belowdisplayskip{4pt}
\mathbb{K}_{\mathrm{EDMD}}=G^{\dagger}A=(\Psi_X^*W\Psi_X)^{\dagger}(\Psi_X^*W\Psi_Y)=(\sqrt{W}\Psi_X)^\dagger\sqrt{W}\Psi_Y.
$$
In some applications, the matrix $G$ may be ill-conditioned and it is common to consider truncated singular value decompositions or other forms of regularization. For simplicity, we assume throughout the paper that $G$ is invertible.

If the quadrature approximation converges, then
\begin{equation}\setlength\abovedisplayskip{4pt}\setlength\belowdisplayskip{4pt}
\label{quad_convergence}
\lim_{M\rightarrow\infty}G_{jk} = \langle \psi_k,\psi_j \rangle\quad \text{ and }\quad \lim_{M\rightarrow\infty}A_{jk} = \langle \mathcal{K}\psi_k,\psi_j \rangle.
\end{equation}
Let $\mathcal{P}_{V_{N}}$ denote the orthogonal projection onto $V_{N}$. As $M\rightarrow \infty$, the convergence in~\eqref{quad_convergence} means that $\mathbb{K}_{\mathrm{EDMD}}$ approaches a matrix representation of $\mathcal{P}_{V_{N}}\mathcal{K}\mathcal{P}_{V_{N}}^*$. Thus, EDMD is a Galerkin method in the large data limit $M\rightarrow \infty$. As a special case, if $\psi_j(\pmb{x})=e_j^*\pmb{x}$ for $j=1,\ldots,d=N$ and $w_m=1/M$, then $\mathbb{K}_{\mathrm{EDMD}}=(\sqrt{W}\Psi_X)^\dagger\sqrt{W}\Psi_Y$. In this case, $\mathbb{K}_{\mathrm{EDMD}}$ is the transpose of the usual DMD matrix,
$$\setlength\abovedisplayskip{4pt}\setlength\belowdisplayskip{4pt}
\mathbb{K}_{\mathrm{DMD}}=\Psi_Y^\top\Psi_X^{\top\dagger}=\Psi_Y^\top\sqrt{W}(\Psi_X^\top\sqrt{W})^\dagger=((\sqrt{W}\Psi_X)^\dagger\sqrt{W}\Psi_Y)^\top=\mathbb{K}_{\mathrm{EDMD}}^\top.
$$
Thus, DMD can be interpreted as producing a Galerkin approximation of the Koopman operator using the set of linear monomials as basis functions. When $d$ is large, it is common to  form a low-rank approximation of
$\sqrt{W}\Psi_X$ via a truncated SVD \cite{kutz2016dynamic}.

There are typically three scenarios for which the convergence in \eqref{quad_convergence} holds:

\vspace{2mm}

\begin{itemize}
	\item[(i)] {\textbf{Random sampling:}} In the initial definition of EDMD, $\omega$ is a probability measure and $\{\pmb{x}^{(m)}\}_{m=1}^M$ are drawn independently according to $\omega$ with the quadrature weights $w_m=1/M$. The strong law of large numbers shows that \eqref{quad_convergence} holds with probability one~\cite[Section 3.4]{2158-2491_2016_1_51}, provided that $\omega$ is not supported on a zero level set that is a linear combination of the dictionary~\cite[Section 4]{korda2018convergence}. Convergence is typically at a Monte Carlo rate of $\mathcal{O}(M^{-1/2})$~\cite{caflisch1998monte}.\vspace{1mm}
	\item[(ii)] {\textbf{High-order quadrature:}} If the dictionary and $F$ are sufficiently regular and we are free to choose the $\{\pmb{x}^{(m)}\}_{m=1}^{M}$, then it is beneficial to select $\{\pmb{x}^{(m)}\}_{m=1}^{M}$ as an $M$-point quadrature rule with weights $\{w_m\}_{m=1}^{M}$. This can lead to much faster convergence rates in~\eqref{quad_convergence}~\cite{colbrook2021rigorous}, but can be difficult if $d$ is large.\vspace{1mm}
	\item[(iii)] {\textbf{Ergodic sampling:}} For a single fixed initial condition $\pmb{x}_0$ and $\pmb{x}^{(m)}=F^{m-1}(\pmb{x}_0)$ (i.e., data collected along one trajectory), if the dynamical system is ergodic, then one can use Birkhoff's Ergodic Theorem to show \eqref{quad_convergence}~\cite{korda2018convergence}. One chooses $w_m=1/M$ but the convergence rate is problem dependent~\cite{kachurovskii1996rate}.
\end{itemize}

\vspace{1mm}

If one is entirely free to select the initial conditions of the trajectory data, and $d$ is not too large, then we recommend picking them based on a high-order quadrature rule. Random and ergodic sampling have the advantage of being practical even when $d$ is large. Ergodic sampling is particularly useful when we have access to only one trajectory of the dynamical system. Ergodic sampling does not require knowledge of $\omega$ (e.g., if one wishes to study the dynamics near attractors).

\section{Measure-preserving EDMD}
\label{sec:mpEDMDD}
We now seek a matrix $\mathbb{K}\in\mathbb{C}^{N\times N}$ that approximates the action of $\mathcal{K}$ on the finite-dimensional subspace ${V}_{{N}}$, and, in addition, corresponds to a unitary operator on ${V}_{{N}}$. Given the Gram matrix $G=\Psi_X^*W\Psi_X$, we can approximate inner products via $\langle \Psi \pmb{g},\Psi \pmb{h} \rangle\approx \pmb{h}^*G\pmb{g}.$ If~\eqref{quad_convergence} holds, then this approximation converges to the inner product as $M\rightarrow\infty$. Similarly, $\|\Psi \pmb{g}\|^2\approx \pmb{g}^*G\pmb{g}=\|G^{1/2}\pmb{g}\|_2$ and $\|\Psi \mathbb{K}\,\pmb{g}\|^2\approx \pmb{g}^*\mathbb{K}^*G\mathbb{K}\pmb{g}.$ Since $\mathcal{K}$ is an isometry, it is natural to enforce
$$\setlength\abovedisplayskip{4pt}\setlength\belowdisplayskip{4pt}
\pmb{g}^*G\pmb{g}=\pmb{g}^*\mathbb{K}^*G\mathbb{K}\pmb{g},\quad \forall \pmb{g}\in\mathbb{C}^N.
$$
This holds if and only if $\mathbb{K}^*G\mathbb{K}=G$. Therefore, we replace \eqref{eq:ContinuousLeastSquaresProblem} by the problem
$$\setlength\abovedisplayskip{4pt}\setlength\belowdisplayskip{4pt}
\underset{\underset{B^*GB=G}{B\in\mathbb{C}^{N\times N}}}{\mathrm{argmin}} \left\{\int_\Omega \max_{\|G^{\frac{1}{2}}\pmb{g}\|_2=1}\!\!\!|R(\pmb{g},\pmb{x})|^2\,d\omega(\pmb{x})\!=\!\int_\Omega \left\|\Psi(F(\pmb{x}))G^{-\frac{1}{2}} - \Psi(\pmb{x})BG^{-\frac{1}{2}}\right\|^2_{2}\,d\omega(\pmb{x})\right\}.
$$
After applying the quadrature rule, the discretized version of this problem is
\begin{equation}\setlength\abovedisplayskip{4pt}\setlength\belowdisplayskip{4pt}
\label{EDMD_opt_prob3}
\underset{\underset{B^*GB=G}{B\in\mathbb{C}^{N\times N}}}{\mathrm{argmin}}\sum_{m=1}^{M} w_m\left\|\Psi(\pmb{y}^{(m)})G^{-\frac{1}{2}}-\Psi(\pmb{x}^{(m)})BG^{-\frac{1}{2}}\right\|_2^2.
\end{equation}
Letting $B=G^{-1/2}CG^{1/2}$ for some matrix $C$, the problem in \eqref{EDMD_opt_prob3} is equivalent to
\begin{equation}\setlength\abovedisplayskip{4pt}\setlength\belowdisplayskip{4pt}
\label{EDMD_opt_prob4}
\underset{\underset{C^*C=I}{C\in\mathbb{C}^{N\times N}}}{\mathrm{argmin}}\left\|W^{\frac{1}{2}}\Psi_XG^{-\frac{1}{2}}C-W^{\frac{1}{2}}\Psi_YG^{-\frac{1}{2}}\right\|^2_F,
\end{equation}
where $\|\cdot\|_F$ denotes the Frobenius norm. The problem~\eqref{EDMD_opt_prob4} is known as the \textit{orthogonal Procrustes problem} \cite{schonemann1966generalized,arun1992unitarily}. The predominant method for computing a solution is via the singular value decomposition (SVD). First, we compute an SVD of
$$\setlength\abovedisplayskip{4pt}\setlength\belowdisplayskip{4pt}
G^{-\frac{1}{2}}\Psi_Y^*W\Psi_XG^{-\frac{1}{2}}=G^{-\frac{1}{2}}A^*G^{-\frac{1}{2}}=U_1\Sigma U_2^*.
$$
A solution of \eqref{EDMD_opt_prob4} is then $C=U_2U_1^*$ and we take $\mathbb{K}=G^{-1/2}U_2U_1^*G^{1/2}$. If $\Sigma$ is degenerate, then $\mathbb{K}$ need not be unique.

Since $\mathbb{K}$ is similar to a unitary matrix, its eigenvalues lie along the unit circle. For stability purposes, the best way to compute the eigendecomposition of $\mathbb{K}$ is to do so for the unitary matrix $U_2U_1^*$. To numerically ensure an orthonormal basis of eigenvectors, we use Matlab's \texttt{schur} command in the examples of~\cref{sec:numerical_examples}. The computation of $\mathbb{K}$ and its eigendecomposition is summarized in Algorithm~\ref{alg:mp_EDMD}. The following proposition lists some useful properties of Algorithm~\ref{alg:mp_EDMD}.

\begin{algorithm}[t]\linespread{1.1}\selectfont{}
\textbf{Input:} Snapshot data $\{\pmb{x}^{(m)},\pmb{y}^{(m)}=F(\pmb{x}^{(m)})\}_{m=1}^{M}$, quadrature weights $\{w_m\}_{m=1}^{M}$, and a dictionary of functions $\{\psi_j\}_{j=1}^{N}$.\\
\vspace{-4mm}
\begin{algorithmic}[1]
\State Compute $G=\Psi_X^*W\Psi_X$ and $A=\Psi_X^*W\Psi_Y$, where $\Psi_X$, $\Psi_Y$ are given in~\eqref{psidef}.
\State Compute an SVD of $G^{-1/2}A^*G^{-1/2}=U_1\Sigma U_2^*$.
\State Compute the eigendecomposition $U_2U_1^*=\hat V\Lambda \hat V^*$.
\State Compute $\mathbb{K}=G^{-1/2}U_2U_1^* G^{1/2}$ and $V=G^{-1/2}\hat V$.
\end{algorithmic} \textbf{Output:} Koopman matrix $\mathbb{K}$, with eigenvectors $V$ and eigenvalues $\Lambda$.
\caption{: \texttt{mpEDMD} for approximating spectral properties of $\mathcal{K}$.}\label{alg:mp_EDMD}\linespread{1}\selectfont{}
\end{algorithm}

\begin{proposition}
\label{prop_basic}
The output of Algorithm~\ref{alg:mp_EDMD} has the following properties.
\begin{enumerate}
	\item[(i)] If~\eqref{quad_convergence} holds, then any limit point of the matrices $\mathbb{K}$ as $M\rightarrow\infty$ corresponds to an operator that is the unitary part of a polar decomposition of $\mathcal{P}_{V_{N}}\mathcal{K}\mathcal{P}_{V_{N}}^*$.
	\item[(ii)] If~\eqref{quad_convergence} holds and $g=\Psi \pmb{g}$ is such that $\mathcal{K}g\in V_{N}$, then $\lim_{M\rightarrow\infty}\mathbb{K}\pmb{g}$ exists and $\lim_{M\rightarrow\infty}\Psi \mathbb{K}\pmb{g}=\mathcal{K}g$.
	\item[(iii)] For any $\epsilon\geq 0$, $	\sigma_{\epsilon}(\mathbb{K})\subset \sigma_{\epsilon \kappa(G^{1/2})}(U_2U_1^*)\subset \{z:||z|-1|\leq \epsilon \kappa(G^{1/2})\}.$
	\item[(iv)] $\kappa(V)\leq \kappa(G^{1/2})$.
\end{enumerate}
\end{proposition}

\begin{proof}
Suppose that $B\in\mathbb{C}^{N\times N}$ is a limit point of the matrix $\mathbb{K}$ as $M\rightarrow\infty$. By taking subsequences if necessary (all matrices are bounded), we may assume that $B=\smash{\hat G^{-\frac{1}{2}}}\hat U_2\hat U_1^* \smash{\hat G^{\frac{1}{2}}}$, $\hat U_j=\lim_{M\rightarrow\infty}U_j$, $\hat G=\lim_{M\rightarrow\infty}G$, and $\hat \Sigma=\lim_{M\rightarrow\infty}\Sigma$. In the large data limit, the problem~\eqref{EDMD_opt_prob3} is independent of the choice of basis for $V_N$, and property (i) is also basis independent. Hence, we may assume without loss of generality that $\hat G$ is the identity matrix corresponding to an orthonormal basis. It follows that $\mathbb{K}_{\mathrm{EDMD}}=\hat U_2\hat\Sigma\hat U_1^*$. Part (i) now follows.

For part (ii), since $\mathcal{K}g\in V_{N}$, we have $\|g\|=\|\mathcal{K}g\|=\lim_{M\rightarrow\infty}\|U_2\Sigma U_1^*\smash{G^{\frac{1}{2}}}\pmb{g}\|_{2}=$ $\lim_{M\rightarrow\infty}\|\Sigma U_1^*\smash{G^{\frac{1}{2}}}\pmb{g}\|_{2},$ where the last equality holds because $U_2$ is unitary. Similarly,
$
\|g\|=\lim_{M\rightarrow\infty}\|\smash{G^{\frac{1}{2}}}\pmb{g}\|_{2}=\lim_{M\rightarrow\infty}\|U_1^*\smash{G^{\frac{1}{2}}}\pmb{g}\|_{2}.
$
Hence $\lim_{M\rightarrow\infty}\|U_1^*\smash{G^{\frac{1}{2}}}\pmb{g}\|_{2}=\lim_{M\rightarrow\infty}\|\Sigma U_1^*\smash{G^{\frac{1}{2}}}\pmb{g}\|_{2}$. $\Sigma$ is a diagonal matrix and all of its entries are in $[0,1]$. We claim that $\lim_{M\rightarrow\infty}[U_1^*\smash{G^{\frac{1}{2}}}-\Sigma U_1^*\smash{G^{\frac{1}{2}}}]\pmb{g}=0$. If not, then by taking a subsequence if necessary, we may assume that $\lim_{M\rightarrow\infty}\Sigma$ and $\lim_{M\rightarrow\infty} U_j$ exist with $\lim_{M\rightarrow\infty}[U_1^*\smash{G^{\frac{1}{2}}}-\Sigma U_1^*\smash{G^{\frac{1}{2}}}]\pmb{g}\neq 0$. But this contradicts $\lim_{M\rightarrow\infty}\|U_1^*\smash{G^{\frac{1}{2}}}\pmb{g}\|_{2}=\lim_{M\rightarrow\infty}\|\Sigma U_1^*\smash{G^{\frac{1}{2}}}\pmb{g}\|_{2}$. Since $\lim_{M\rightarrow\infty}\smash{G^{-\frac{1}{2}}}U_2\Sigma U_1^*\smash{G^{\frac{1}{2}}}\pmb{g}=\lim_{M\rightarrow\infty}\mathbb{K}_{\mathrm{EDMD}}\pmb{g}$ exists, $\lim_{M\rightarrow\infty}\mathbb{K}\pmb{g}$ exists and $\lim_{M\rightarrow\infty}\Psi \mathbb{K}\pmb{g}=\lim_{M\rightarrow\infty}\Psi \mathbb{K}_{\mathrm{EDMD}}\pmb{g}=g$.

For part (iii), for any $z\notin\sigma(\mathbb{K})$ we have
$\smash{
\|(\mathbb{K}-z)^{-1}\|=\|G^{-\frac{1}{2}}(U_2U_1^*-z)^{-1}G^{\frac{1}{2}}\|}\leq \smash{\kappa(G^{\frac{1}{2}})\|(U_2U_1^*\!-\!z)^{-1}\|.}
$
Hence, $\sigma_{\epsilon}(\mathbb{K})\subset \sigma_{\epsilon \kappa(G^{1/2})}(U_2U_1^*)$. $U_2U_1^*$ is unitary, and hence $\smash{\sigma_{\epsilon \kappa(G^{\frac{1}{2}})}(U_2U_1^*)\!\!\subset\!\! \{z\!:\!\!||z|\!-\!1|\!\leq\! \epsilon \kappa(\smash{G^{\frac{1}{2}}})\}}$. Finally, $\! V\!=\! \smash{G^{-\frac{1}{2}}}\!\hat V\!$ for unitary $\hat V$ so (iv) holds.
\end{proof}

Part (i) of~\cref{prop_basic} provides a geometric interpretation of Algorithm~\ref{alg:mp_EDMD}, that we use to prove convergence of spectral measures in~\cref{sec:conv_theory}. Part (ii) shows that Algorithm~\ref{alg:mp_EDMD} respects the invariance properties of $\mathcal{K}$. This is particularly useful for delay embedding (see~\cref{delay}). Parts (iii) and (iv) provide conditioning bounds on the eigendecomposition of $\mathbb{K}$. This is useful since we can only ever approximate the eigendecomposition using finite $M$. In contrast, conditioning bounds for $\mathbb{K}_{\mathrm{EDMD}}$ cannot hold in general. In fact, $\mathbb{K}_{\mathrm{EDMD}}$ need not even be diagonalizable (see \cref{sec:shift_example}). Further stability properties of \texttt{mpEDMD} are investigated in~\cref{num_example:pendulum}.

\section{Convergence theory}
\label{sec:conv_theory}

We now prove convergence results of Algorithm~\ref{alg:mp_EDMD} to the spectral information of $\mathcal{K}$. Throughout, $\{v_j\}_{j=1}^N$ denotes the eigenvectors of $\mathbb{K}$ (output of Algorithm~\ref{alg:mp_EDMD}) with corresponding eigenvalues $\{\lambda_j\}_{j=1}^N$.

\subsection{Approximation of projection-valued spectral measures}
To approximate the spectral measure $\mathcal{E}$, we consider the spectral measure, $\mathcal{E}_{N,M}$, of the matrix $\mathbb{K}$ on the Hilbert space $\mathbb{C}^N$ with the inner product induced by $G$,
$$\setlength\abovedisplayskip{4pt}\setlength\belowdisplayskip{4pt}
d\mathcal{E}_{N,M}(\lambda)=\sum_{j=1}^{N}v_jv_j^*G\delta(\theta-\lambda_j)\,d\lambda.
$$
We prove weak convergence\footnote{This is not to be confused with weak operator convergence of the operator-valued measures.} of $\Psi\mathcal{E}_{N,M}$ and begin with the following bound.

\begin{theorem}
\label{thm_weak_conv}
Suppose that $\phi:\mathbb{T}\rightarrow\mathbb{R}$ is Lipschitz continuous with Lipschitz constant bounded by $1$. Then for any $L\in\mathbb{N}$, $g\in L^2(\Omega,\omega)$ and $\pmb{g}\in\mathbb{C}^N$,
\begin{equation*}\setlength\abovedisplayskip{4pt}\setlength\belowdisplayskip{4pt}
\begin{split}
&\left\|\int_{\mathbb{T}} \!\phi(\lambda)[d\mathcal{E}(\lambda)g\!-\!\Psi d\mathcal{E}_{N,M}(\lambda) \pmb{g}]\right\|\leq C\Big(\frac{\log(L)}{L}\!\left[\!\|g\|\!+\!\|\Psi G^{-\frac{1}{2}}\!\|\|G^{\frac{1}{2}}\pmb{g}\|_2\!\right]\!+\!\|g-\Psi\pmb{g}\|\|\phi\|_\infty\!\!\\
&\!\!\quad\quad\quad\quad\quad\quad\quad\quad\quad\quad\quad\quad\quad\quad\quad\quad+\!\sum_{1\leq l \leq L}\!\frac{\left[\|\mathcal{K}^lg-\Psi\mathbb{K}^l\pmb{g}\|+\|(\mathcal{K}^*)^lg-\Psi\mathbb{K}^{-l}\pmb{g}\|\right]}{l}\Big),
\end{split}
\end{equation*}
where $C$ is a universal constant.
\end{theorem}

\begin{proof}
Consider the Laurent series of $\phi$, $\phi(\lambda)=\sum_{l=-\infty}^\infty c_l\lambda^l$, and let 
$$\setlength\abovedisplayskip{4pt}\setlength\belowdisplayskip{4pt}
S_L\phi(\lambda)= \sum_{|l|\leq L} c_l\lambda^l,\quad\text{where } c_l=\frac{1}{2\pi i}\int_{\mathbb{T}} \lambda^{-(l+1)}\phi(\lambda)\, d\lambda.
$$
For $|l|\geq 1$, $|c_l|\lesssim 1/|l|$. Arguing as in the proof of~\cref{spec_meas_indep},
$$\setlength\abovedisplayskip{4pt}\setlength\belowdisplayskip{4pt}
\int_{\mathbb{T}} \lambda^l \, d\mathcal{E}(\lambda)g=\begin{cases}\mathcal{K}^lg,&\quad \text{if }l\geq 0,\\
(\mathcal{K}^*)^{-l}g,&\quad \text{otherwise}
\end{cases}.
$$
Arguing directly, we see that $\Psi \int_{\mathbb{T}}\lambda^l\, d\mathcal{E}_{N,M}(\lambda) \pmb{g}=\Psi\mathbb{K}^l\pmb{g}.$ It follows that
\begin{equation}\label{W1a}\setlength\abovedisplayskip{4pt}\setlength\belowdisplayskip{4pt}
\begin{split}
&\left\|\int_{\mathbb{T}} S_L\phi(\lambda)\,d\mathcal{E}(\lambda)g-\Psi \int_{\mathbb{T}}S_L\phi(\lambda)\, d\mathcal{E}_{N,M}(\lambda) \pmb{g}\right\|\\
&\quad\quad\quad\lesssim \|g-\Psi\pmb{g}\|\|\phi\|_\infty+\sum_{1\leq l \leq L}\frac{1}{l}\left[\|\mathcal{K}^lg-\Psi\mathbb{K}^l\pmb{g}\|+\|(\mathcal{K}^*)^lg-\Psi\mathbb{K}^{-l}\pmb{g}\|\right].
\end{split}
\end{equation}
Let $\hat \phi=\phi-\phi(0)$, then since the Lipschitz constant of $\phi$ is bounded by $1$, $\|\hat \phi\|_\infty\leq \pi$. Since $\|\phi-S_L\phi\|_{\infty}=\|\hat\phi-S_L\hat\phi\|_{\infty}\lesssim {\log(L)}/{L}$ \cite[Chapter I.3]{jackson1930theory}, it follows that
\begin{equation}\label{W1b}\setlength\abovedisplayskip{4pt}\setlength\belowdisplayskip{4pt}
\left\|\int_{\mathbb{T}} (\phi-S_L\phi)(\lambda)\, d\mathcal{E}(\lambda)g\right\|\lesssim \frac{\log(L)}{L}\|g\|.
\end{equation}
Using the matrix functional calculus, we have
$$\setlength\abovedisplayskip{4pt}\setlength\belowdisplayskip{4pt}
\Psi \int_{\mathbb{T}}(\phi-S_L\phi)(\lambda)\, d\mathcal{E}_{N,M}(\lambda) \pmb{g}=\Psi G^{-1/2}(\phi-S_L\phi)(U_2U_1^*)G^{1/2}\pmb{g}.
$$
Since $U_2U_1^*$ is unitary, $\|(\phi-S_L\phi)(U_2U_1^*)\|\leq\|\phi-S_L\phi\|_\infty$. It follows that
\begin{equation}\label{W1c}\setlength\abovedisplayskip{4pt}\setlength\belowdisplayskip{4pt}
\left\|\Psi\int_{\mathbb{T}} (\phi-S_L\phi)(\lambda)\, d\mathcal{E}_{N,M}(\lambda) \pmb{g}\right\|\lesssim \frac{\log(L)}{L}\|\Psi G^{-1/2}\|\|G^{1/2}\pmb{g}\|_2.
\end{equation}
\cref{thm_weak_conv} follows by combining~\eqref{W1a},~\eqref{W1b} and~\eqref{W1c}.
\end{proof}

A key part of the above proof is that $\mathbb{K}$ represents a normal truncation of $\mathcal{K}$. Combined with the strong convergence of $\mathbb{K}$ to $\mathcal{K}$, this allows us to prove convergence of spectral measures. We consider a sequence of vectors spaces $\{V_N\}_{N=1}^\infty$ and the large data limit $M\rightarrow\infty$. The following lemma shows that the first summation term in~\cref{thm_weak_conv} converges to zero as $N\rightarrow\infty$ if the sequence of vector spaces is dense, and that the second summation term also converges to zero if, in addition, $\mathcal{K}$ is unitary. This result shows strong operator convergence of $\mathbb{K}^l$. The $\limsup$ as $M\rightarrow\infty$ is needed in the case that $\mathbb{K}$ is not unique, but is of no practical concern.

\begin{lemma}
\label{lemma_SOT_conv}
Suppose that $\lim_{N\rightarrow\infty}\!\mathrm{dist}(h,V_{N})=0$ for all $h\in L^2(\Omega,\omega)$ and~\eqref{quad_convergence} holds.  Then for any $g\in L^2(\Omega,\omega)$ and $\pmb{g}_N\in\mathbb{C}^N$ with $\lim_{N\rightarrow\infty}\|g-\Psi \pmb{g}_N\|=0$,
\begin{equation}\setlength\abovedisplayskip{4pt}\setlength\belowdisplayskip{4pt}
\label{limit1}
\lim_{N\rightarrow\infty}\limsup_{M\rightarrow\infty} \|\mathcal{K}^lg-\Psi\mathbb{K}^l\pmb{g}_N\|=0,\quad \forall l\in\mathbb{N}.
\end{equation}
If, in addition, $\mathcal{K}$ is unitary, then
\begin{equation}\setlength\abovedisplayskip{4pt}\setlength\belowdisplayskip{4pt}
\label{limit2}
\lim_{N\rightarrow\infty}\limsup_{M\rightarrow\infty} \|(\mathcal{K}^*)^lg-\Psi\mathbb{K}^{-l}\pmb{g}_N\|=0,\quad \forall l\in\mathbb{N}.
\end{equation}
\end{lemma}

\begin{proof}
Recall that $\mathcal{P}_{N}$ is the orthogonal projection onto $V_{N}$ so that $\mathcal{P}_{N}\mathcal{P}_{N}^*$ is the identity on $V_{N}$. For notational convenience, let $\mathcal{Q}_{N}=\mathcal{P}_{N}^*\mathcal{P}_{N}$. The assumption that $\lim_{N\rightarrow\infty}\mathrm{dist}(h,V_{N})=0$ for all $h\in L^2(\Omega,\omega)$ implies that $\mathcal{Q}_{N}$ converges strongly to the identity on $L^2(\Omega,\omega)$, $I$. It follows that $\mathcal{Q}_{N}\mathcal{K}\mathcal{Q}_{N}$ converges strongly to $\mathcal{K}$ and that $(\mathcal{Q}_{N}\mathcal{K}^*\mathcal{Q}_{N}\mathcal{K}\mathcal{Q}_{N})^{1/2}$ converges strongly to $(\mathcal{K}^*\mathcal{K})^{1/2}=I$.

To prove~\eqref{limit1}, we may assume without loss of generality, by taking subsequences if necessary, that the large data limit $\lim_{M\rightarrow\infty}\mathbb{K}$ exists for each fixed $N$. Let $\mathcal{K}_{N}$ denote the operator on $V_{N}$ represented by $\lim_{M\rightarrow\infty}\mathbb{K}$. \cref{prop_basic} (i) shows that
$$\setlength\abovedisplayskip{4pt}\setlength\belowdisplayskip{4pt}
\mathcal{P}_{N}^*\mathcal{K}_{N}\mathcal{P}_{N}(\mathcal{Q}_{N}\mathcal{K}^*\mathcal{Q}_{N}\mathcal{K}\mathcal{Q}_{N})^{1/2}=\mathcal{Q}_{N}\mathcal{K}\mathcal{Q}_{N}.
$$
Let $h\in L^2(\Omega,\omega)$ and $\epsilon>0$. Choose $N_0\in\mathbb{N}$ so that if $N\geq N_0$ then $\|\mathcal{Q}_{N}\mathcal{K}\mathcal{Q}_{N}h-\mathcal{K}h\|\leq \epsilon$ and $\|(\mathcal{Q}_{N}\mathcal{K}^*\mathcal{Q}_{N}\mathcal{K}\mathcal{Q}_{N})^{1/2}h-h\|\leq \epsilon$. It follows that if $N\geq N_0$, then
$$\setlength\abovedisplayskip{4pt}\setlength\belowdisplayskip{4pt}
\|\mathcal{P}_{N}^*\mathcal{K}_{N}\mathcal{P}_{N}h\!-\!\mathcal{K}h\|\!\leq\! \|\mathcal{Q}_{N}\mathcal{K}\mathcal{Q}_{N}h-\mathcal{K}h\|\!+\!\|\mathcal{P}_{N}^*\mathcal{K}_{N}\mathcal{P}_{N}[(\mathcal{Q}_{N}\mathcal{K}^*\mathcal{Q}_{N}\mathcal{K}\mathcal{Q}_{N})^{1/2}h-h]\|\!\leq\! 2\epsilon,
$$
where we have used the fact that $\|\mathcal{P}_{N}^*\mathcal{K}_{N}\mathcal{P}_{N}\|\leq 1$. Since $h\in L^2(\Omega,\omega)$ and $\epsilon>0$ were arbitrary, it follows that $\mathcal{P}_{N}^*\mathcal{K}_{N}\mathcal{P}_{N}$ converges strongly to $\mathcal{K}$ as $N\rightarrow\infty$, and hence $[\mathcal{P}_{N}^*\mathcal{K}_{N}\mathcal{P}_{N}]^l$ converges strongly to $\mathcal{K}^l$ for any $l\in\mathbb{N}$. Let $g_N=\Psi\pmb{g}_N$, then
$$\setlength\abovedisplayskip{4pt}\setlength\belowdisplayskip{4pt}
\lim_{M\rightarrow\infty}\Psi\mathbb{K}^l\pmb{g}_N=\mathcal{P}_{N}^*\mathcal{K}_{N}^l\mathcal{P}_{N}g_N=[\mathcal{P}_{N}^*\mathcal{K}_{N}\mathcal{P}_{N}]^l g_N,\quad \forall l\in\mathbb{N}
$$
Since $g_N$ converges to $g$, $[\mathcal{P}_{N}^*\mathcal{K}_{N}\mathcal{P}_{N}]^l$ converges strongly to $\mathcal{K}^l$, and all relevant operators are uniformly bounded, the limit in~\eqref{limit1} holds.

Now suppose that $\mathcal{K}$ is unitary so that $\mathcal{K}\mathcal{K}^*$ is the identity. Again, we may assume without loss of generality that $\lim_{M\rightarrow\infty}\mathbb{K}$ exists for each fixed $N$. Let $\mathcal{K}_{N}$ denote the operator on $V_{N}$ represented by $\lim_{M\rightarrow\infty}\mathbb{K}$. Since $\mathcal{P}_{N}^*\mathcal{K}_{N}\mathcal{P}_{N}$ converges strongly to $\mathcal{K}$ as $N\rightarrow\infty$, for all $h\in L^2(\Omega,\omega)$, we must have
$$\setlength\abovedisplayskip{4pt}\setlength\belowdisplayskip{4pt}
\limsup_{N\rightarrow\infty}\|[\mathcal{P}_{N}^*\mathcal{K}_{N}\mathcal{P}_{N}]^*\mathcal{P}_{N}^*\mathcal{K}_{N}\mathcal{P}_{N}\mathcal{K}^*h-[\mathcal{P}_{N}^*\mathcal{K}_{N}\mathcal{P}_{N}]^*h\|\leq\!\!\lim_{N\rightarrow\infty}\|\mathcal{P}_{N}^*\mathcal{K}_{N}\mathcal{P}_{N}\mathcal{K}^*h-h\|=0.
$$
Since $\mathcal{K}_{N}^*\mathcal{K}_{N}=\mathcal{P}_{N}\mathcal{P}_{N}^*$ are the identity on $V_{N}$,
$
[\mathcal{P}_{N}^*\mathcal{K}_{N}\mathcal{P}_{N}]^*\mathcal{P}_{N}^*\mathcal{K}_{N}\mathcal{P}_{N}\mathcal{K}^*h=\mathcal{Q}_N\mathcal{K}^*h
$
converges to $\mathcal{K}^*h$. It follows that $\mathcal{P}_{N}^*\mathcal{K}_{N}^*\mathcal{P}_{N}$ converges strongly to $\mathcal{K}^*$ as $N\rightarrow\infty$. Since $\lim_{M\rightarrow\infty}\Psi\mathbb{K}^{-l}\pmb{g}_N=\mathcal{P}_{N}^*(\mathcal{K}_{N}^*)^l\mathcal{P}_{N}g_N$, we argue as before to prove~\eqref{limit2} holds.
\end{proof}

\begin{remark}[Computing suitable $\pmb{g}_N$ and the Koopman mode decomposition]\label{KMD_rem}
Given $g\in L^2(\Omega,\omega)$, we can compute a suitable $\pmb{g}_N$ in~\cref{lemma_SOT_conv} via
$$\setlength\abovedisplayskip{4pt}\setlength\belowdisplayskip{4pt}
\pmb{g}_{N}=G^{-1}\Psi_X^*W\begin{pmatrix}g(\pmb{x}^{(1)})&\cdots&g(\pmb{x}^{(M)})\end{pmatrix}^\top\in\mathbb{C}^{N}.
$$
If the quadrature rule converges, $\Psi\pmb{g}_{N}$ converges to $\mathcal{P}_{V_{N}}g$ in the large data limit and $\lim_{N\rightarrow\infty}\|g-\mathcal{P}_{V_{N}}g\|=0$ under the first condition of the lemma. We obtain
$$\setlength\abovedisplayskip{4pt}\setlength\belowdisplayskip{4pt}
\label{gen_kp_m_decm}
[\mathcal{P}_{V_N}g](\pmb{x})\approx \Psi(\pmb{x})V\left[V^{-1}(\sqrt{W}\Psi_X)^\dagger \sqrt{W}\begin{pmatrix}
g(\pmb{x}^{(1)})&\cdots&
g(\pmb{x}^{(M)})
\end{pmatrix}^\top\right].
$$
Hence, we have the approximate factorization
\begin{equation}\setlength\abovedisplayskip{4pt}\setlength\belowdisplayskip{4pt}
\begin{split}
g(\pmb{x}_n)&\approx \Psi(\pmb{x}_0)\mathbb{K}^nV \left[V^{-1}(\sqrt{W}\Psi_X)^\dagger \sqrt{W}\begin{pmatrix}g(\pmb{x}^{(1)})&\cdots&
g(\pmb{x}^{(M)})\end{pmatrix}^\top\right]\\
&=\left[\Psi(\pmb{x}_0)V\right]\Lambda^n \left[V^{-1}(\sqrt{W}\Psi_X)^\dagger \sqrt{W}\begin{pmatrix}g(\pmb{x}^{(1)})&\cdots&
g(\pmb{x}^{(M)})\end{pmatrix}^\top\right].
\end{split}
\label{LIP_pred}
\end{equation}
The factor $\Psi V$ is a quasimatrix of approximate Koopman eigenfunctions. The columns of the final factor in square brackets are known as Koopman modes~\cite{mezic2005spectral}. The first part of~\cref{lemma_SOT_conv} shows the convergence of this approximation. $\hfill\blacksquare$
\end{remark}
Using~\cref{lemma_SOT_conv}, we now show that $\Psi\mathcal{E}_{N,M}$ converges weakly to $\mathcal{E}$ if $\mathcal{K}$ is unitary. For example, if $F$ is invertible and measure-preserving, $\mathcal{K}$ is unitary.

\begin{theorem}
\label{thm_weak_conv2}
Suppose that $\lim_{N\rightarrow\infty}\mathrm{dist}(h,V_{N})=0$ for all $h\in L^2(\Omega,\omega)$,~\eqref{quad_convergence} holds, $\mathcal{K}$ is unitary and that $\phi:\mathbb{T}\rightarrow\mathbb{R}$ is Lipschitz continuous. Then for any $g\in L^2(\Omega,\omega)$ and $\pmb{g}_N\in\mathbb{C}^N$ with $\lim_{N\rightarrow\infty}\|g-\Psi \pmb{g}_N\|=0$,
$$\setlength\abovedisplayskip{4pt}\setlength\belowdisplayskip{4pt}
\lim_{N\rightarrow\infty}\limsup_{M\rightarrow\infty}\left\|\int_{\mathbb{T}} \phi(\lambda)\,d\mathcal{E}(\lambda)g-\Psi \int_{\mathbb{T}}\phi(\lambda)\, d\mathcal{E}_{N,M}(\lambda) \pmb{g}_N\right\|=0.
$$
\end{theorem}

\begin{proof}
By rescaling, we may assume without loss of generality that the Lipschitz constant of $\phi$ is bounded by $1$. We use the bound in~\cref{thm_weak_conv}, replacing $\pmb{g}$ by $\pmb{g}_N$. We have $\lim_{M\rightarrow\infty}\|\Psi G^{-1/2}\|=1$ and $\lim_{M\rightarrow\infty}\|G^{1/2}\pmb{g}_{N}\|_2=\|\Psi \pmb{g}_N\|$. Since $\lim_{N\rightarrow\infty}\|g-\Psi \pmb{g}_N\|=0$, it follows that for any $L\in\mathbb{N}$,
\begin{align*}\setlength\abovedisplayskip{4pt}\setlength\belowdisplayskip{4pt}
&\limsup_{N\rightarrow\infty}\limsup_{M\rightarrow\infty}\left\|\int_{\mathbb{T}} \phi(\lambda)\,d\mathcal{E}(\lambda)g-\Psi \int_{\mathbb{T}}\phi(\lambda)\, d\mathcal{E}_{N,M}(\lambda) \pmb{g}_N\right\|\\
&\quad\quad\lesssim \frac{\log(L)}{L}\|g\|+\limsup_{N\rightarrow\infty}\limsup_{M\rightarrow\infty}\sum_{1\leq l \leq L}\!\frac{1}{l}\!\left[\|\mathcal{K}^lg-\Psi\mathbb{K}^l\pmb{g}_N\|+\|(\mathcal{K}^*)^lg-\Psi\mathbb{K}^{-l}\pmb{g}_N\|\right].
\end{align*}
Since $L\in\mathbb{N}$ is arbitrary, to prove the theorem it is enough to show that
$$\setlength\abovedisplayskip{4pt}\setlength\belowdisplayskip{4pt}
\limsup_{N\rightarrow\infty}\limsup_{M\rightarrow\infty} \|\mathcal{K}^lg-\Psi\mathbb{K}^l\pmb{g}_N\|+\|(\mathcal{K}^*)^lg-\Psi\mathbb{K}^{-l}\pmb{g}_N\|=0,\quad \forall l\in\mathbb{N}.
$$
This follows from~\cref{lemma_SOT_conv}.
\end{proof}

\subsection{Warning example}\label{sec:shift_example}

In general, we cannot drop the condition that $\mathcal{K}$ is unitary from~\cref{thm_weak_conv2}. For example, consider $\Omega=\mathbb{N}$, the counting measure $\omega$,
$$\setlength\abovedisplayskip{4pt}\setlength\belowdisplayskip{4pt}
\text{and}\qquad x_{n+1} = F(x_{n})\qquad \text{with} \qquad F(x) =\begin{cases}
x - 1,&\quad \text{if $x>1$}\\
0,&\quad \text{otherwise.}
\end{cases}
$$
Let $V_N=\mathrm{span}\{e_1,\ldots, e_N\}$, where $e_k(j)=\delta_{k,j}$. We can choose a quadrature rule with nodes $\{1,\ldots,M\}$ (with $M\geq N$) and weights $w_m=1$, so that 
$$\setlength\abovedisplayskip{4pt}\setlength\belowdisplayskip{4pt}
\mathbb{K}_{\mathrm{EDMD}}=\lim_{M\rightarrow\infty}\mathbb{K}_{\mathrm{EDMD}}=\begin{pmatrix}
0 &&&\\
1 & \ddots && \\
& \ddots &\ddots& \\
& & 1 & 0
\end{pmatrix},\quad\mathbb{K}=\lim_{M\rightarrow\infty}\mathbb{K}=\begin{pmatrix}
0 &&&1\\
1 & \ddots && \\
& \ddots &\ddots& \\
& & 1 & 0
\end{pmatrix}.
$$
Let $\phi(\lambda)=1/\lambda$, then from the proof of~\cref{thm_weak_conv},
$$\setlength\abovedisplayskip{4pt}\setlength\belowdisplayskip{4pt}
\int_{\mathbb{T}} \phi(\lambda)\,d\mathcal{E}(\lambda)e_1=\mathcal{K}^* e_1=0,\qquad \lim_{M\rightarrow\infty}\Psi \int_{\mathbb{T}}\phi(\lambda)\, d\mathcal{E}_{N,M}(\lambda) e_1=e_N,
$$
where we also use $e_1$ to denote the first canonical basis vector of $\mathbb{C}^N$. Clearly $e_N$ does not converge to $e_1$ in $L^2(\Omega,\omega)$. Note also that for this example, $\mathbb{K}_{\mathrm{EDMD}}$ is not even diagonalizable.\footnote{This issue, as well as spectral pollution and the absence of spectral inclusion, also hold for the natural extension of this example (the bi-lateral shift) on $\Omega=\mathbb{Z}$, for which $\mathcal{K}$ is unitary.} This kind of behavior is by no means rare. In fact, this example is connected to many dynamical systems, such as Bernoulli shifts, with so-called Lebesgue spectrum~\cite[Chapter 2]{arnold1968ergodic}. However, $e_N$ does converge weakly to $0$ in $L^2(\Omega,\omega)$. Motivated by this, we remove the need for $\mathcal{K}$ to be unitary when considering scalar-valued spectral measures in the next subsection.

\subsection{Approximation of scalar-valued spectral measures}

Let $g\in L^2(\Omega,\omega)$ with $\|g\|=1$. We approximate $\mu_g$ (see~\cref{sec:spec_meas_def_jkhkhjk}) by $\smash{\mu_{\pmb{g}}^{(N,M)}}$, where
\begin{equation}\setlength\abovedisplayskip{4pt}\setlength\belowdisplayskip{4pt}
\label{muNM_def}
\smash{\mu_{\pmb{g}}^{(N,M)}(U)}=\pmb{g}^*G\mathcal{E}_{N,M}(U)\pmb{g}=\sum_{\lambda_j\in U}|v_j^*G\pmb{g}|^2
\end{equation}
and $\pmb{g}$ is normalized so that $\pmb{g}^*G\pmb{g}=1$. Since $\{G^{1/2}v_j\}_{j=1}^N$ is an orthonormal basis for $\mathbb{C}^N$, $\smash{\mu_{\pmb{g}}^{(N,M)}}$ is a probability measure. To measure the distance between probability measures, we use the Wasserstein metric. For two Borel probability measures $\mu$ and $\nu$ on $\mathbb{T}$, the $W_1$ distance is defined as
$$\setlength\abovedisplayskip{4pt}\setlength\belowdisplayskip{4pt}
W_1(\mu,\nu)=\sup\left\{\int_{\mathbb{T}}\phi(\lambda)\,d(\mu-\nu)(\lambda):\phi:\mathbb{T}\rightarrow\mathbb{R}\text{ Lip. cts., Lip. constant $\leq1$}\right\}.
$$
Convergence in this metric is equivalent to the usual weak convergence of measures. The following theorem provides an explicit bound on $W_1(\mu_g,\mu_{\pmb{g}}^{(N,M)})$.

\begin{theorem}
\label{thm_weak_conv_scalar}
For any $L\in\mathbb{N}$, $g\in L^2(\Omega,\omega)$ and $\pmb{g}\in\mathbb{C}^N$ with $\pmb{g}^*G\pmb{g}=1$,
$$\setlength\abovedisplayskip{4pt}\setlength\belowdisplayskip{4pt}
W_1\!\left(\mu_g,\smash{\mu_{\pmb{g}}^{(N,M)}}\right)\!\leq \! C\Big(\frac{\log(L)}{L}+\!\!\!\sum_{1\leq l \leq L}\!\!\!\frac{| \langle\mathcal{K}^{|l|}g,g\rangle\!-\!\pmb{g}^*G\mathbb{K}^{|l|}\pmb{g}|}{l}\Big),\text{ for a universal constant $C$.}
$$
\end{theorem}

\begin{proof}
The proof is almost identical to that of~\cref{thm_weak_conv}. Let $\phi:\mathbb{T}\rightarrow\mathbb{R}$ be Lipschitz continuous with Lipschitz constant bounded by $1$. Since $\mu_g$ and $\mu_{\pmb{g}}^{(N,M)}$ are probability measures, we may assume from the start that $\phi(0)=0$.
Moreover,
$$\setlength\abovedisplayskip{4pt}\setlength\belowdisplayskip{4pt}
\int_{\mathbb{T}} \lambda^l \, d\mu_g(\lambda)=\begin{cases}\langle\mathcal{K}^lg,g\rangle,&\quad \text{if }l\geq 0,\\
\langle(\mathcal{K}^{-l})^*g,g\rangle=\langle g,\mathcal{K}^{-l}g\rangle,&\quad \text{otherwise}
\end{cases}
$$
and $
\pmb{g}^*G \int_{\mathbb{T}}\lambda^l\, d\smash{\mu_{\pmb{g}}^{(N,M)}}(\lambda) \pmb{g}=\pmb{g}^*G\mathbb{K}^l\pmb{g}=\pmb{g}^*G^{1/2}[U_2U_1^*]^l G^{1/2}\pmb{g}.$ In particular, if $l<0$, then $
\pmb{g}^*G \int_{\mathbb{T}}\lambda^l\, d\smash{\mu_{\pmb{g}}^{(N,M)}}(\lambda) \pmb{g}=\overline{\pmb{g}^*G\mathbb{K}^{|l|}\pmb{g}}.$ It follows that
$$\setlength\abovedisplayskip{4pt}\setlength\belowdisplayskip{4pt}
\left|\int_{\mathbb{T}} \lambda^l \, d(\mu_g- \smash{\mu_{\pmb{g}}^{(N,M)}})(\lambda)\right|=\left| \langle\mathcal{K}^{|l|}g,g\rangle-\pmb{g}^*G\mathbb{K}^{|l|}\pmb{g}\right|,\quad \forall l\in\mathbb{Z}.
$$
Arguing as in the proof of~\cref{thm_weak_conv}, it follows that
\begin{equation}\label{W2a}\setlength\abovedisplayskip{4pt}\setlength\belowdisplayskip{4pt}
\left|\int_{\mathbb{T}} S_L\phi(\lambda) \, d(\mu_g- \smash{\mu_{\pmb{g}}^{(N,M)}})(\lambda)\right|\lesssim \sum_{1\leq l \leq L}\frac{1}{l}\left| \langle\mathcal{K}^{|l|}g,g\rangle-\pmb{g}^*G\mathbb{K}^{|l|}\pmb{g}\right|.
\end{equation}
Since $\mu_g$ and $\smash{\mu_{\pmb{g}}^{(N,M)}}$ are probability measures and $\|\phi-S_L\phi\|_{\infty}\lesssim \log(L)/L$,
\begin{equation}\label{W2b}\setlength\abovedisplayskip{4pt}\setlength\belowdisplayskip{4pt}
\left|\int_{\mathbb{T}} (\phi-S_L\phi)(\lambda) \, d(\mu_g- \smash{\mu_{\pmb{g}}^{(N,M)}})(\lambda)\right|\lesssim \frac{\log(L)}{L}.
\end{equation}
The result follows by combining~\eqref{W2a} and~\eqref{W2b} and taking suprema over such $\phi$.
\end{proof}

\cref{thm_weak_conv_scalar} and the first part of~\cref{lemma_SOT_conv}, show the following corollary.

\begin{corollary}\label{mug_conv}
Suppose that $\lim_{N\rightarrow\infty}\mathrm{dist}(h,V_{N})=0$ for all $h\in L^2(\Omega,\omega)$, and~\cref{quad_convergence} holds. Then for any $g\in L^2(\Omega,\omega)$ and $\pmb{g}_N\in\mathbb{C}^N$ with $\lim_{N\rightarrow\infty}\!\!\|g-\Psi \pmb{g}_N\|=0$,
\begin{equation}\label{weak_weak_convergence2}\setlength\abovedisplayskip{4pt}\setlength\belowdisplayskip{4pt}
\lim_{N\rightarrow\infty}\limsup_{M\rightarrow\infty}W_1\left(\mu_g,\smash{\mu_{\pmb{g}}^{(N,M)}}\right)=0.
\end{equation}
\end{corollary}

A popular choice of dictionary is a Krylov subspace, i.e., $\mathrm{span}\{g,\mathcal{K}g,\ldots,\mathcal{K}^{N-1}g\}$. This corresponds to time-delay embedding, which is a popular method for DMD-type algorithms~\cite{arbabi2017ergodic,kamb2020time,pan2020structure}. 
Part (ii) of~\cref{prop_basic} shows that if $g,\mathcal{K}g,\ldots,\mathcal{K}^{l}g\in V_N$ and $g=\Psi\pmb{g}$, then $\lim_{M\rightarrow\infty}| \langle\mathcal{K}^{|l|}g,g\rangle-\pmb{g}^*G\mathbb{K}^{|l|}\pmb{g}|=0.$ Combining this with \cref{thm_weak_conv_scalar} shows the following corollary, which provides an explicit rate of convergence.

\begin{corollary}
\label{delay}
If $\{g,\mathcal{K}g,\ldots,\mathcal{K}^{L}g\}\subset V_N$, $g=\Psi\pmb{g}$, and~\cref{quad_convergence} holds, then
\begin{equation}\label{weak_weak_convergence3}\setlength\abovedisplayskip{4pt}\setlength\belowdisplayskip{4pt}
\lim_{M\rightarrow\infty}W_1\left(\mu_g,\smash{\mu_{\pmb{g}}^{(N,M)}}\right)\leq C{\log(L)}/{L},\quad \text{for a universal constant $C$.}
\end{equation}
\end{corollary}

\subsection{Approximation of spectra}

We end this section with the convergence to the approximate point spectrum of $\mathcal{K}$. The following theorem shows that the eigenvalues computed by~\cref{alg:mp_EDMD} approximate the whole of $\sigma_{\mathrm{ap}}(\mathcal{K})$ as $N\rightarrow\infty$ and the subspace $V_N$ becomes richer.

\begin{theorem}
\label{lemma_spectra_conv}
If $\lim_{N\rightarrow\infty}\mathrm{dist}(h,V_{N})=0$ $\forall h\in L^2(\Omega,\omega)$ and~\eqref{quad_convergence} holds, then \begin{equation}
\label{spectra_conv_bound}\setlength\abovedisplayskip{4pt}\setlength\belowdisplayskip{4pt}
\lim_{N\rightarrow\infty}\limsup_{M\rightarrow\infty} \sup_{\lambda\in \sigma_{\mathrm{ap}}(\mathcal{K})}\mathrm{dist}(\lambda,\sigma(\mathbb{K}))=0.
\end{equation}
\end{theorem}

\begin{proof}
To prove~\eqref{spectra_conv_bound}, we may assume without loss of generality, by taking subsequences if necessary, that the large data limit $\lim_{M\rightarrow\infty}\mathbb{K}$ exists for each fixed $N$. Let $\mathcal{K}_{N}$ denote the operator on $V_{N}$ represented by $\lim_{M\rightarrow\infty}\mathbb{K}$.

Let $\delta>0$ and let $\{z_1,\ldots,z_k\}\subset \sigma_{\mathrm{ap}}(\mathcal{K})$ be such that $\mathrm{dist}(\lambda,\{z_1,\ldots,z_k\})\leq \delta$ for any $\lambda\in \sigma_{\mathrm{ap}}(\mathcal{K})$. Such a $\delta$-net exists since $\sigma_{\mathrm{ap}}(\mathcal{K})$ is compact. For $j=1,\ldots,k$ there exists $g_j\in L^2(\Omega,\omega)$ of norm $1$ such that $\|(\mathcal{K}-z_j)g_j\|\leq \delta$. Since $\lim_{N\rightarrow\infty}\mathrm{dist}(h,V_{N})=0$ for any $h\in L^2(\Omega,\omega)$, we may choose $g_{j,N}=\Psi \pmb{g}_{j,N}\in V_N$, each of norm $1$, such that $\lim_{N\rightarrow\infty}\|g_j-g_{j,N}\|=0$ for $j=1,\ldots,k$. Using the first part of~\cref{lemma_SOT_conv},
$$\setlength\abovedisplayskip{4pt}\setlength\belowdisplayskip{4pt}
\limsup_{N\rightarrow\infty} \|(\mathcal{K}_N-z_j)g_{j,N}\|=\limsup_{N\rightarrow\infty} \lim_{M\rightarrow\infty}\|\Psi(\mathbb{K}-z_j)\pmb{g}_{j,N}\|= \|(\mathcal{K}-z_j)g_{j}\|\leq\delta.
$$
Since $\mathcal{K}_N$ is unitary, $\limsup_{N\rightarrow\infty}\mathrm{dist}(z_j,\sigma(\mathcal{K}_N))\leq\delta$ and hence
$$\setlength\abovedisplayskip{4pt}\setlength\belowdisplayskip{4pt}
\limsup_{N\rightarrow\infty}\limsup_{M\rightarrow\infty}\mathrm{dist}(z_j,\sigma(\mathbb{K}))=\limsup_{N\rightarrow\infty}\mathrm{dist}(z_j,\sigma(\mathcal{K}_N))\leq\delta.
$$
Since $\sup_{\lambda\in \sigma_{\mathrm{ap}}(\mathcal{K})}\mathrm{dist}(\lambda,\sigma(\mathbb{K}))\leq \sup_{j=1,\ldots,k}\mathrm{dist}(z_j,\sigma(\mathbb{K}))+\delta$, we have
$$\setlength\abovedisplayskip{4pt}\setlength\belowdisplayskip{4pt}
\limsup_{N\rightarrow\infty}\limsup_{M\rightarrow\infty} \sup_{\lambda\in \sigma_{\mathrm{ap}}(\mathcal{K})}\mathrm{dist}(\lambda,\sigma(\mathbb{K}))\leq 2\delta.
$$
Since $\delta>0$ was arbitrary, the theorem follows.
\end{proof}

Despite this result, $\sigma(\mathbb{K})$ can suffer from spectral pollution. That is, eigenvalues of $\mathbb{K}$ may approximate points that are not in the spectrum of $\mathcal{K}$. We can avoid spectral pollution by computing residuals and discarding eigenpairs with a large residual. Suppose that~\eqref{quad_convergence} holds, $v\in\mathbb{C}^N$, and $\lambda\in\mathbb{C}$. Since $\mathcal{K}^*\mathcal{K}$ is the identity,
\begin{equation}\setlength\abovedisplayskip{4pt}\setlength\belowdisplayskip{4pt}
\label{residual_bound}
\|(\mathcal{K}\!-\!\lambda)\Psi v\|\!=\!\sqrt{\langle (\mathcal{K}\!-\!\lambda)\Psi v,(\mathcal{K}\!-\!\lambda)\Psi v \rangle}\!=\!\lim_{M\rightarrow\infty}\!\! \sqrt{v^*[(1+|\lambda|^2)G-\overline{\lambda}A-\lambda A^*]v}.\!\!\!\!
\end{equation}
Since $\mathcal{K}$ is an isometry, this residual provides a good error estimate. In particular, if $v$ is normalized so that $\lim_{M\rightarrow\infty}\|G^{1/2}v\|_2=1$, then $\|(\mathcal{K}-\lambda)\Psi v\|\geq\mathrm{dist}(\lambda,\sigma(\mathcal{K}))$.

\section{Numerical examples}
\label{sec:numerical_examples}

We consider three numerical examples, two with data from numerical simulations, and one with experimentally collected data. Each example demonstrates different aspects and advantages of \texttt{mpEDMD}.

\subsection{Lorenz system and convergence of spectral measures}
The Lorenz system~\cite{lorenz1963deterministic} is the following system of three coupled ordinary differential equations:
$$\setlength\abovedisplayskip{4pt}\setlength\belowdisplayskip{4pt}
\dot{X}=\sigma\left(Y-X\right),\quad\dot{Y}=X\left(\rho-Z\right)-Y,\quad \dot{Z}=XY-\beta Z.
$$
 The system describes a truncated model of Rayleigh--B\'{e}nard convection, where the parameters $\sigma$, $\rho$, and $\beta$ are proportional to the Prandtl number, Rayleigh number, and the physical proportions of the fluid, respectively. We take $\sigma=10$, $\beta=8/3$, and $\rho=28$, corresponding to the original system studied by Lorenz, and consider the dynamics of $\pmb{x}=(X,Y,Z)$ on the Lorenz attractor. The system is chaotic and strongly mixing~\cite{luzzatto2005lorenz} (and hence ergodic), so that there are no non-trivial eigenvalues of $\mathcal{K}$. We consider the corresponding discrete-time dynamical system by sampling with a time-step $\Delta_t=0.1$. We use the \texttt{ode45} command in MATLAB to collect data along a single trajectory with $M$ snapshots, from an initial point on the attractor. The quadrature rule in~\cref{sec:EDMD_recap} therefore corresponds to ergodic sampling.\footnote{Though we cannot accurately numerically integrate for long time periods since the system is chaotic, this does not effect the convergence of the quadrature rule. This effect is known as shadowing.}

\begin{figure}[!tbp]
 \centering
\begin{minipage}[b]{0.32\textwidth}
  \begin{overpic}[width=\textwidth,trim={0mm 0mm 0mm -2mm},clip]{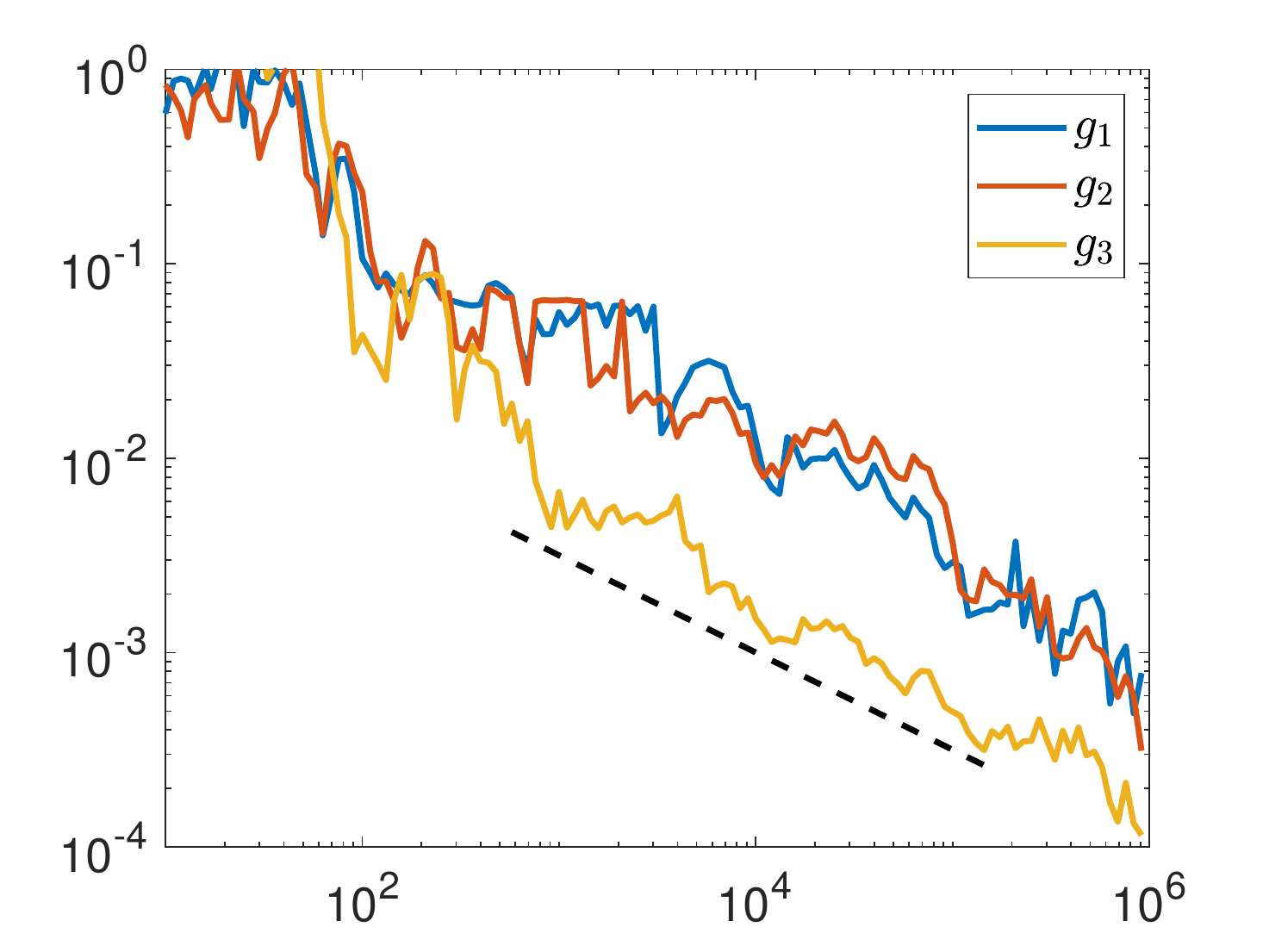}
		\put (6,74) {$W_1\left(\smash{\mu_{\pmb{g}_j}^{(50,2\times 10^6)}},\smash{\mu_{\pmb{g}_j}^{(50,M)}}\right)$}
   \put (49,-5) {$M$}
	\put (34,28) {\small\rotatebox{-27}{$\mathcal{O}(M^{-1/2})$}}
   \end{overpic}
 \end{minipage}
\begin{minipage}[b]{0.32\textwidth}
  \begin{overpic}[width=\textwidth,trim={0mm 0mm 0mm -2mm},clip]{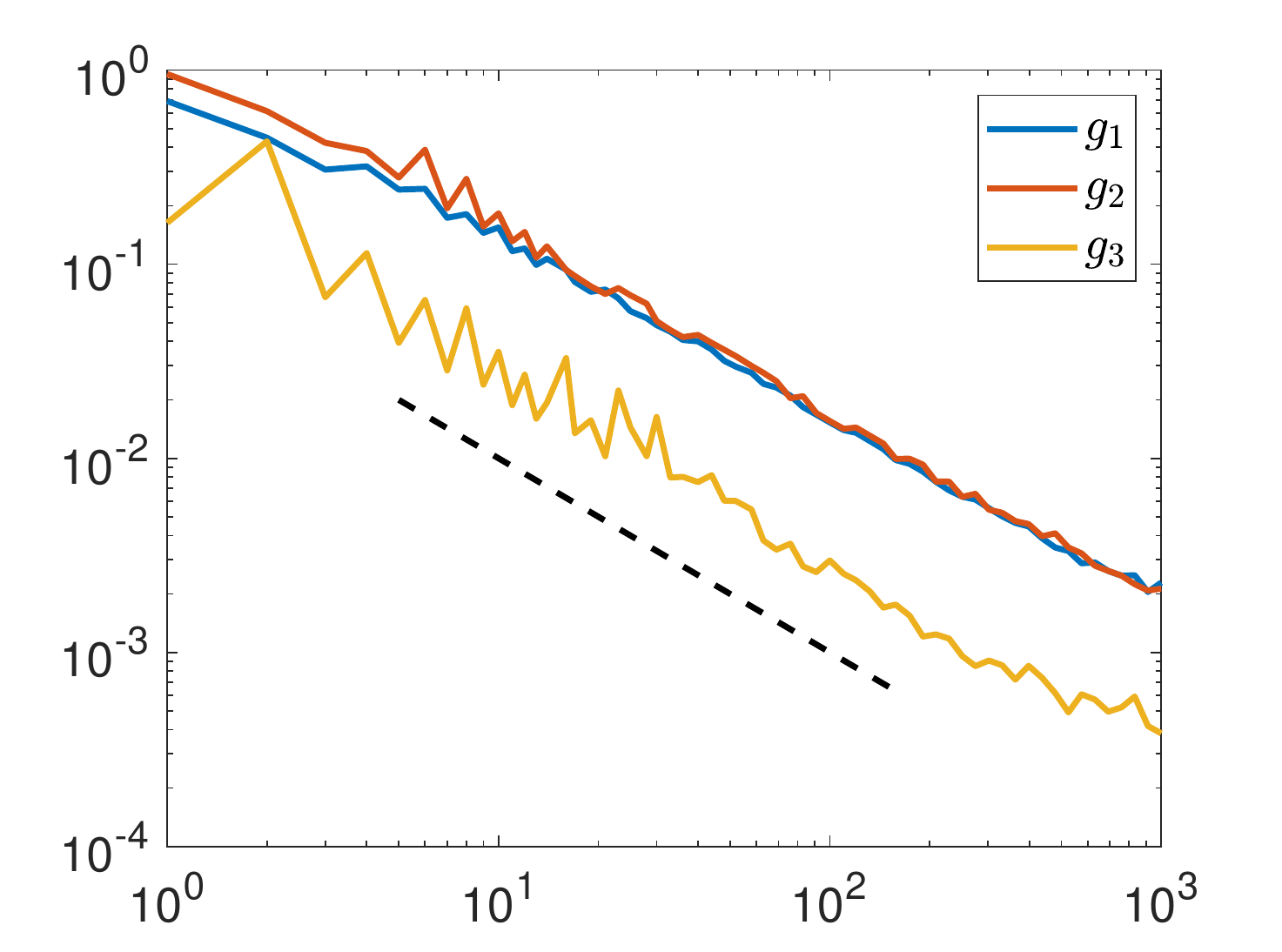}
		\put (18,74) {$W_1\left(\mu_{g_j},\smash{\mu_{\pmb{g}_j}^{(N,10^6)}}\right)$}
   \put (49,-5) {$N$}
	\put (30,34) {\small\rotatebox{-30}{$\mathcal{O}(N^{-1})$}}
   \end{overpic}
 \end{minipage}
\begin{minipage}[b]{0.32\textwidth}
  \begin{overpic}[width=\textwidth,trim={0mm 0mm 0mm -2mm},clip]{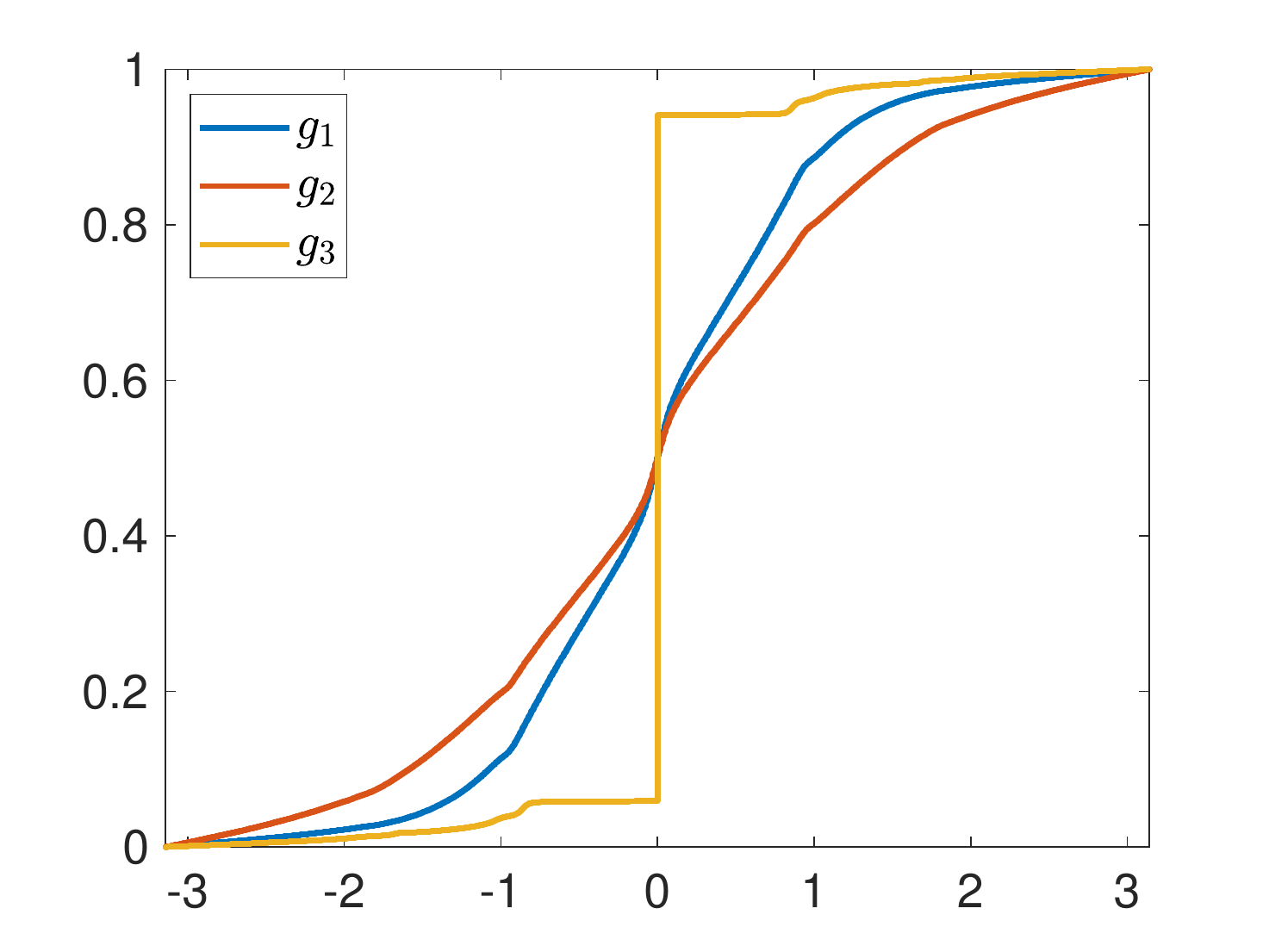}
		\put (20,74) {Cdf of $\smash{\mu_{\pmb{g}_j}^{(10^3,10^6)}}$}
   \put (49,-5) {$\theta$}
   \end{overpic}
 \end{minipage}
\vspace{-1mm}
	  \caption{Left: Convergence of $\smash{\mu_{\pmb{g}_j}^{(50,M)}}$ as $M\rightarrow\infty$. Middle: Convergence to the scalar-valued measure as $N\rightarrow\infty$. The $W_1$ distance is computed by comparing to an approximation with larger $N$. Right: Cdf of ${\mu_{\pmb{g}_j}^{(10^3,10^6)}}$ plotted against the phase, $\theta$, of the spectral parameter $\lambda=e^{i\theta}$. In all cases, the $W_1$ distance is computed using the $L^1$ distance between the cdfs.}\vspace{-3mm}
\label{fig:lorenz1}
\end{figure}
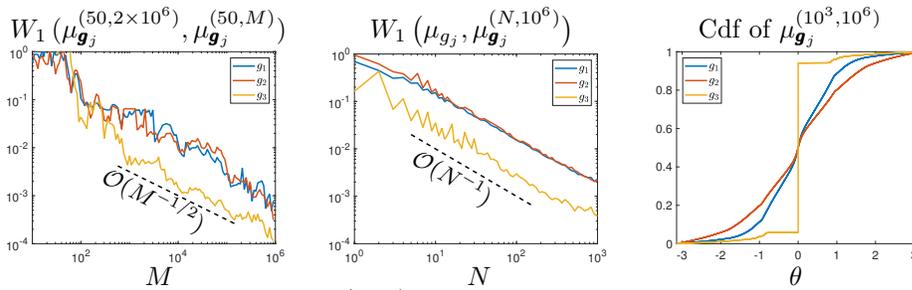

We consider the scalar-valued spectral measures $\mu_{\pmb{g}_j}^{(N,M)}$, where $g_j(\pmb{x})=c_j[\pmb{x}]_j$ is the $j$th coordinate suitably normalized to have norm $1$ with respect to the ergodic measure $\omega$. In each case, we use $\{g_j,\mathcal{K}g_j,\ldots, \mathcal{K}^{N-1}g_j\}$ as the dictionary. This choice corresponds to time-delay embedding.~\cref{fig:lorenz1} (left) shows the convergence as $M\rightarrow\infty$ (large data limit) for a fixed $N=50$. The convergence is at a Monte--Carlo rate of $\mathcal{O}(M^{-1/2})$.~\cref{fig:lorenz1} (middle) shows the convergence as $N\rightarrow\infty$, where $M=10^6$ is selected large enough to have negligible effect on the shown errors. The plot demonstrates the rate $\mathcal{O}(N^{-1})$ from~\cref{delay}.~\cref{fig:lorenz1} (right) plots the cumulative distribution functions (cdfs) of $\smash{\mu_{\pmb{g}_j}^{(N,M)}}$ for $N=10^3$ and $M=10^6$. 
The cdf for $\smash{\mu_{\pmb{g}_3}^{(N,M)}}$ suggests an atom at $\lambda =1$ with small absolutely continuous spectrum in the vicinity of $\lambda=1$. In contrast, $\smash{\mu_{\pmb{g}_1}^{(N,M)}}$ and $\smash{\mu_{\pmb{g}_2}^{(N,M)}}$ are more uniform.

Next, we approximate the projection-valued spectral measures and demonstrate \cref{thm_weak_conv2}. We use $\{g_1,g_2,g_3,\mathcal{K}g_1,\mathcal{K}g_2,\mathcal{K}g_3,\ldots, \mathcal{K}^{q-1}g_1,\mathcal{K}^{q-1}g_2,\mathcal{K}^{q-1}g_3\}$ as the dictionary. We take $\phi(\lambda)=\exp((\lambda-\overline{\lambda})/(2i))$  and compute $\int_{\mathbb{T}}\phi(\lambda)\, d \mathcal{E}_{N,M}(\lambda) \pmb{g}_j$. \cref{fig:lorenz2} (left) shows the convergence as $M\rightarrow\infty$ for a fixed $N=30$. Again, we see the Monte--Carlo rate of convergence $\mathcal{O}(M^{-1/2})$. \cref{fig:lorenz2} (right) shows the convergence as $N\rightarrow\infty$, where $M=10^6$ is selected large enough to have negligible effect on the shown errors.~\cref{thm_weak_conv2} does not provide a rate of convergence and~\cref{fig:lorenz1} suggests a convergence rate of approximately $\mathcal{O}(N^{-0.75})$. In general, we found the rate to be dependent on the function $g$. Finally,~\cref{fig:lorenz3} shows the outputs as functions on the Lorenz attractor using $N=3q=999$ basis functions and $M=10^6$.


\begin{figure}[!tbp]
 \centering\vspace{-1mm}
\begin{minipage}[b]{0.4\textwidth}
  \begin{overpic}[width=\textwidth,trim={0mm 0mm 0mm -7mm},clip]{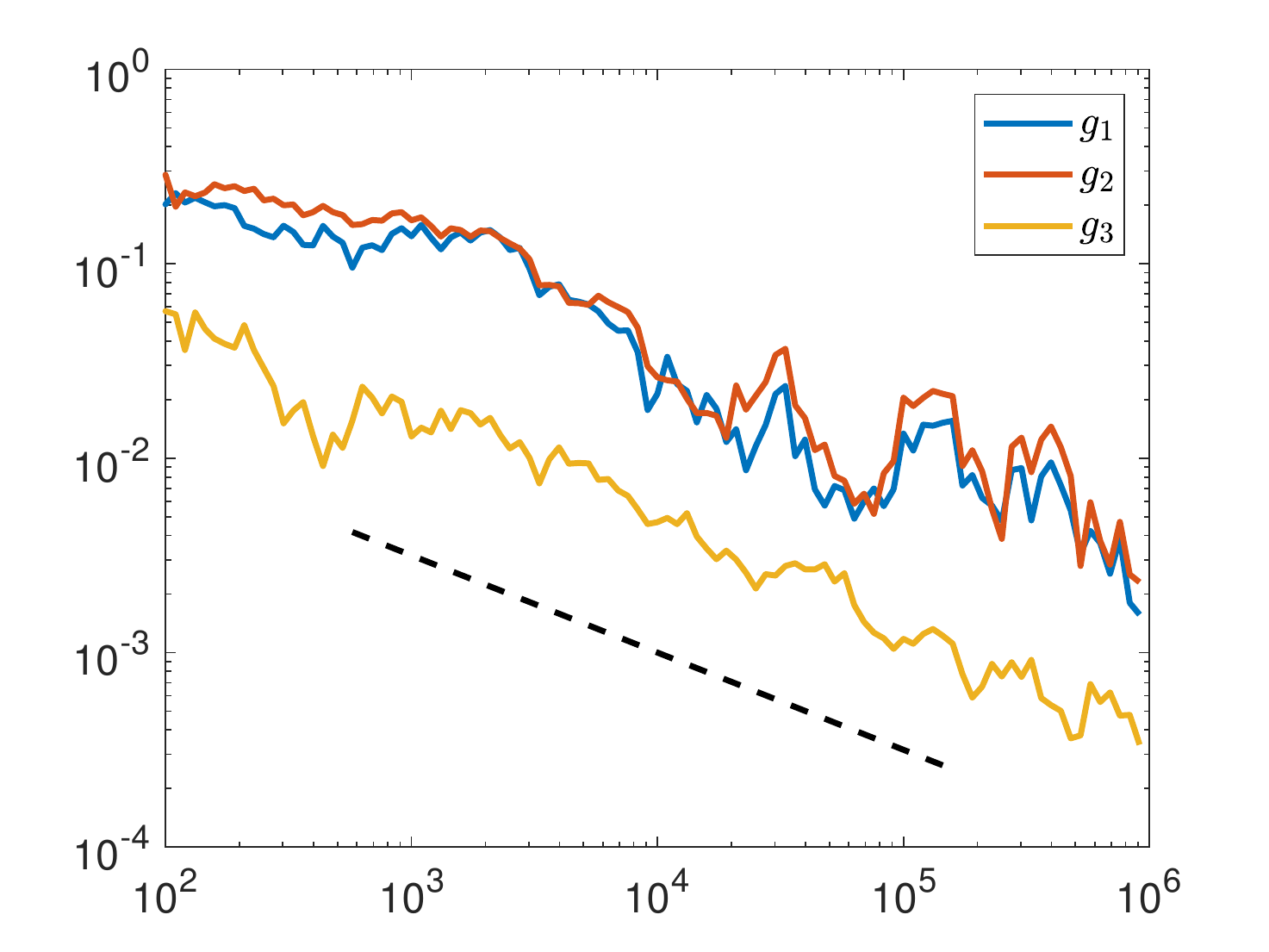}
		\put (3,74) {\tiny{\parbox{1\linewidth}{\begin{equation*}\frac{\|\Psi\int_{\mathbb{T}}\phi(\lambda) d[ \mathcal{E}_{30,2\times 10^6}(\lambda)- \mathcal{E}_{30,M}(\lambda)] \pmb{g}_j\|}{\|\Psi\int_{\mathbb{T}}\phi(\lambda) d \mathcal{E}_{30,2\times 10^6}(\lambda) \pmb{g}_j\|}\end{equation*}}}}
   \put (49,-3) {$M$}
	\put (35,23) {\rotatebox{-22}{$\mathcal{O}(M^{-1/2})$}}
   \end{overpic}
 \end{minipage}
\begin{minipage}[b]{0.4\textwidth}
  \begin{overpic}[width=\textwidth,trim={0mm 0mm 0mm -7mm},clip]{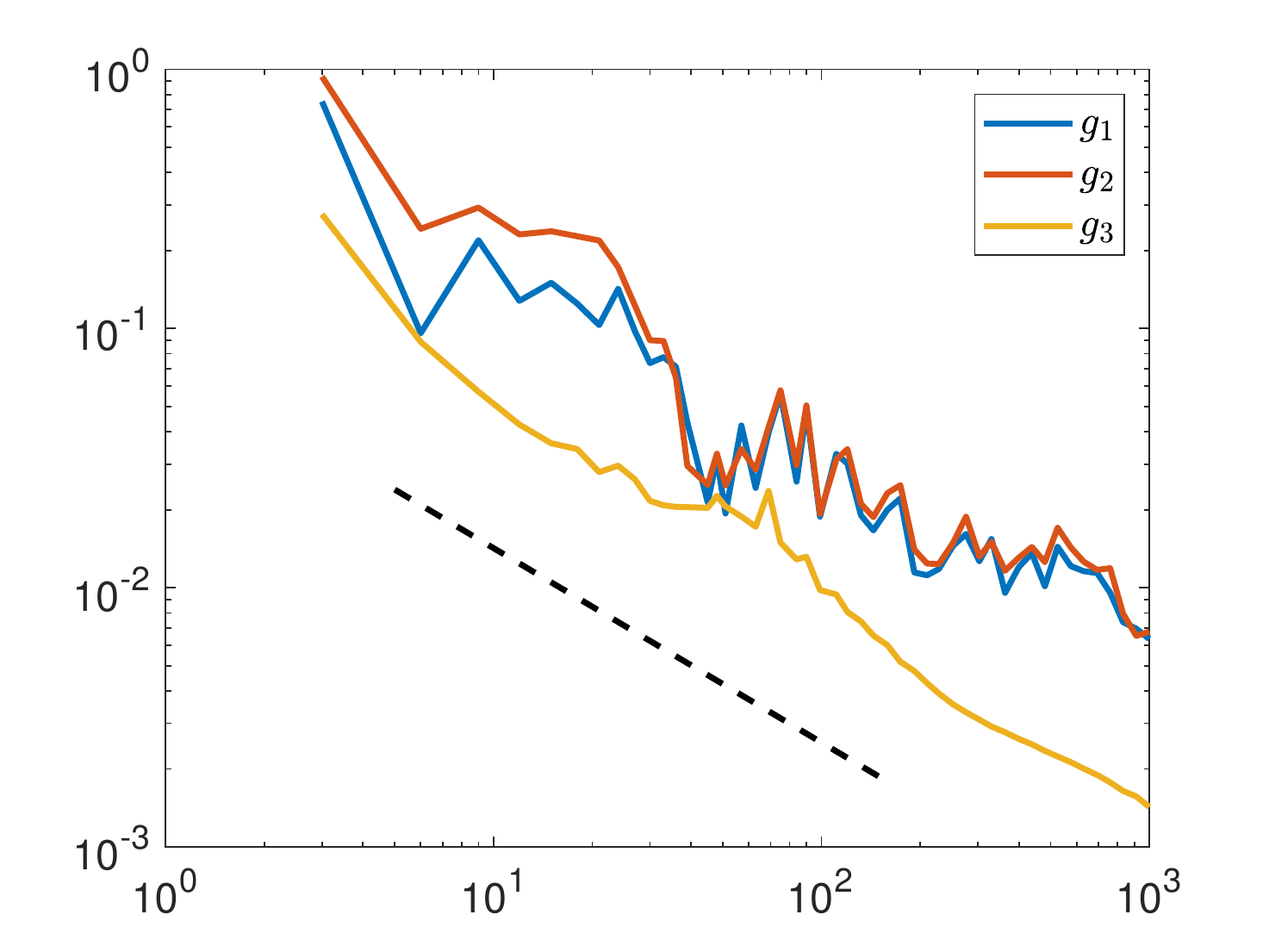}
		\put (3,74) {\tiny{\parbox{1\linewidth}{\begin{equation*}\frac{\|\int_{\mathbb{T}}\phi(\lambda) (d \mathcal{E}(\lambda) g_j-\Psi d \mathcal{E}_{999,10^6}(\lambda) \pmb{g}_j)\|}{\|\int_{\mathbb{T}}\phi(\lambda) d \mathcal{E}(\lambda) g_j\|}\end{equation*}}}}
   \put (49,-3) {$N$}
	\put (33,27) {\rotatebox{-32}{$\mathcal{O}(N^{-0.75})$}}
   \end{overpic}
 \end{minipage}\vspace{-2mm}
	  \caption{Left: Convergence of integrals as $M\rightarrow\infty$. Right: Convergence of the integrals as $N\rightarrow\infty$. The relative error is computed by comparing to an approximation with larger $N$.}\vspace{-2mm}
\label{fig:lorenz2}
\end{figure}

\begin{figure}[!tbp]
 \centering
 \begin{minipage}[b]{0.32\textwidth}
  \begin{overpic}[width=\textwidth,trim={0mm 0mm 0mm 0mm},clip]{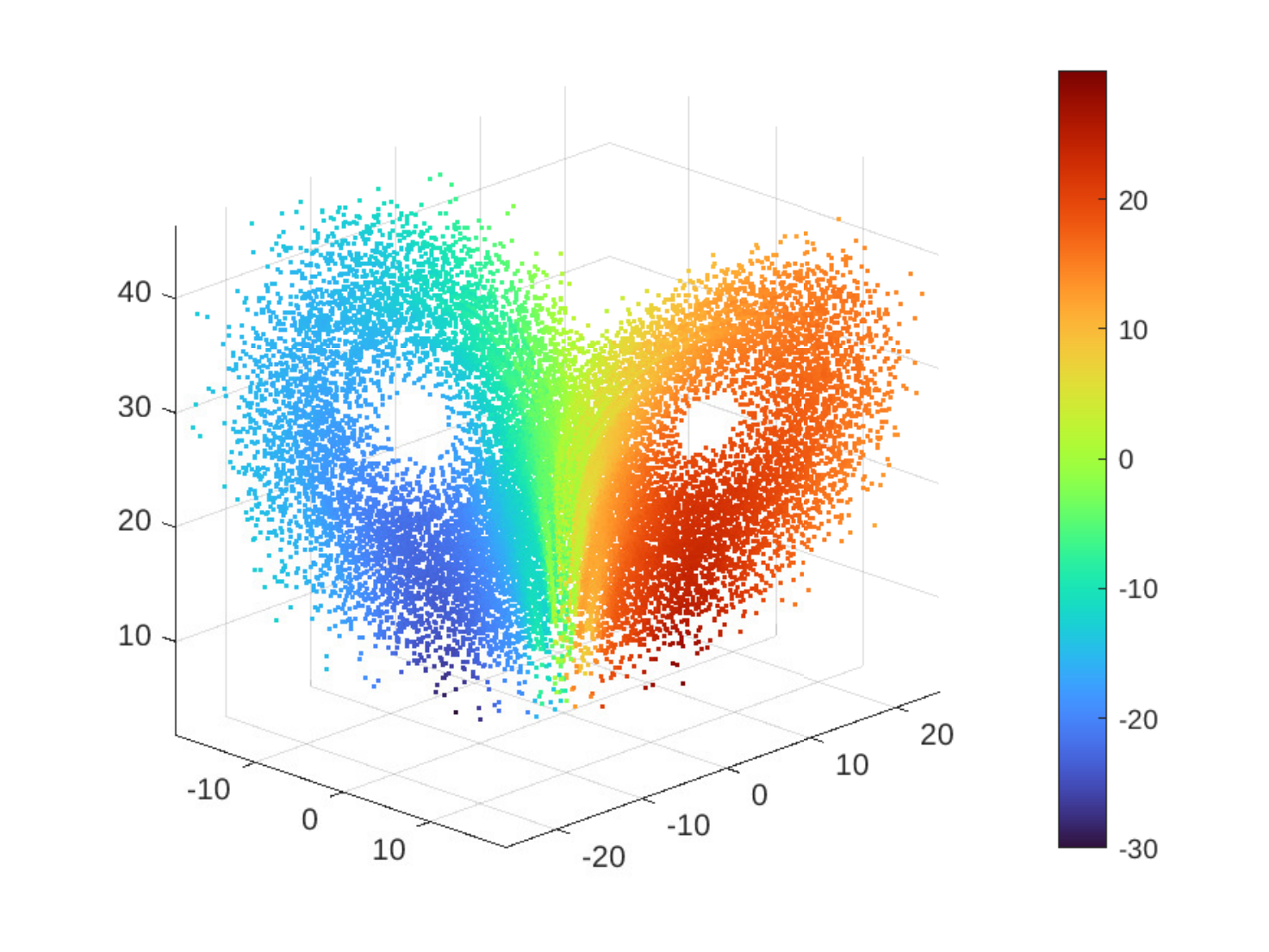}
		\put (6,68) {\small$\Psi\int_{\mathbb{T}}\phi(\lambda)\, d \mathcal{E}_{N,M}(\lambda) \pmb{g}_1$}
   \put (15,4) {$X$}
	\put (62,6) {$Y$}
		\put (2,33) {$Z$}
   \end{overpic}
 \end{minipage}
	\begin{minipage}[b]{0.32\textwidth}
  \begin{overpic}[width=\textwidth,trim={0mm 0mm 0mm 0mm},clip]{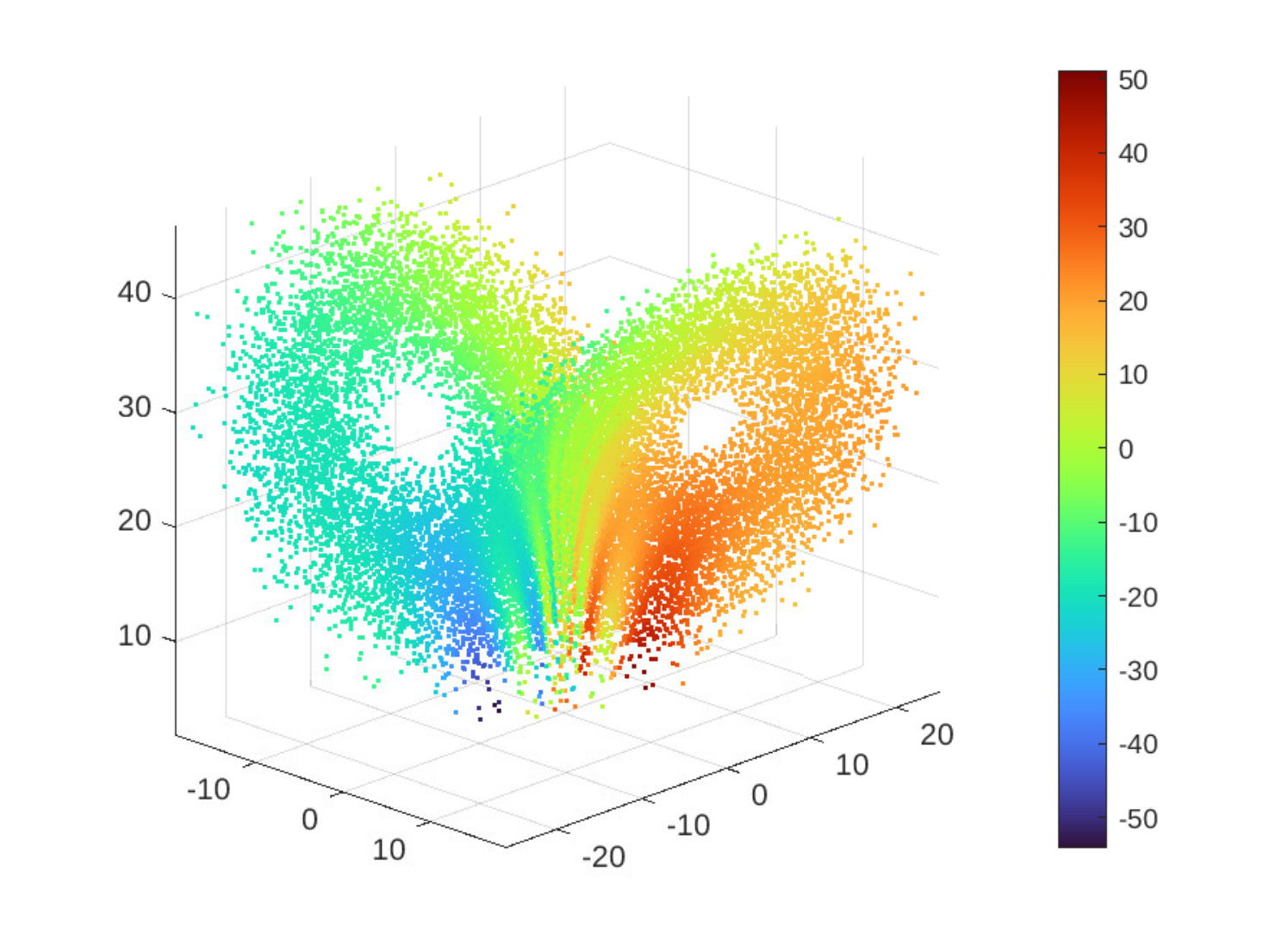}
		\put (6,68) {\small$\Psi\int_{\mathbb{T}}\phi(\lambda)\, d \mathcal{E}_{N,M}(\lambda) \pmb{g}_2$}
   \put (15,4) {$X$}
	\put (62,6) {$Y$}
		\put (2,33) {$Z$}
   \end{overpic}
 \end{minipage}
\begin{minipage}[b]{0.32\textwidth}
  \begin{overpic}[width=\textwidth,trim={0mm 0mm 0mm 0mm},clip]{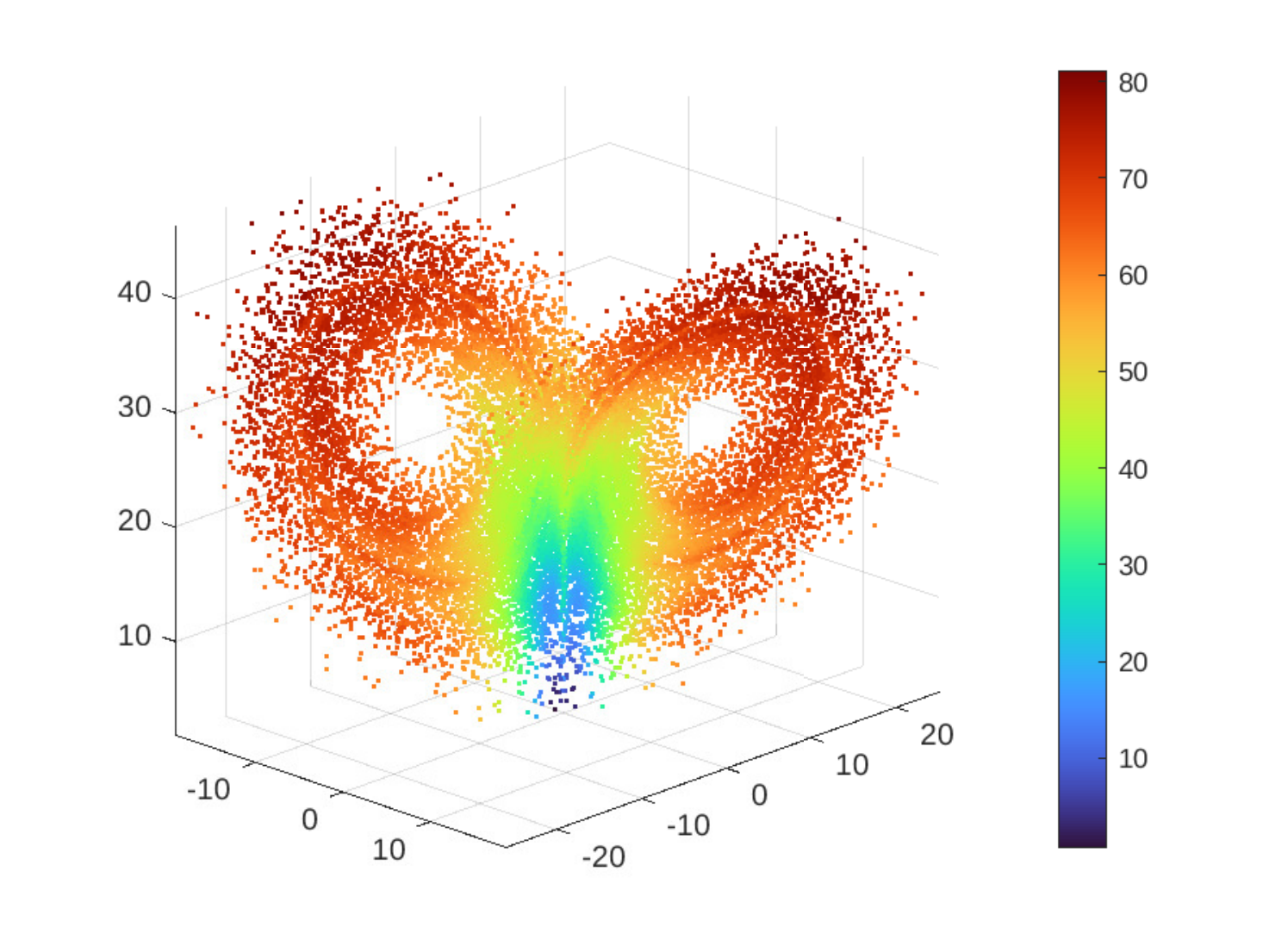}
		\put (6,68) {\small$\Psi\int_{\mathbb{T}}\phi(\lambda)\, d \mathcal{E}_{N,M}(\lambda) \pmb{g}_3$}
   \put (15,4) {$X$}
	\put (62,6) {$Y$}
		\put (2,33) {$Z$}
   \end{overpic}
 \end{minipage}\vspace{-4mm}
	  \caption{$\int_{\mathbb{T}}\phi(\lambda)\, d \mathcal{E}_{N,M}(\lambda) \pmb{g}_j$ computed using $N=3q=999$, $M=10^6$ and plotted at $2\times 10^4$ points on the attractor for visualization.}\vspace{-4mm}
\label{fig:lorenz3}
\end{figure}
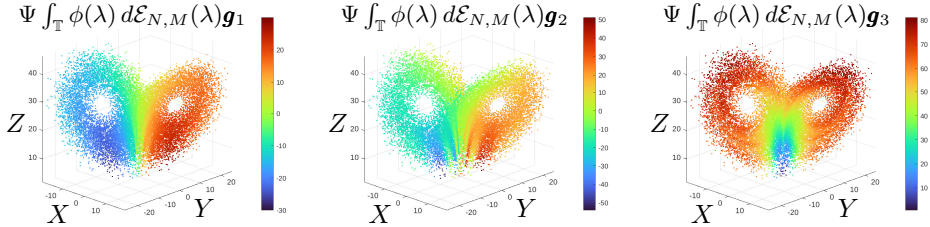

\subsection{Nonlinear pendulum, approximate eigenfunctions, and robustness to noise}\label{num_example:pendulum}

We now consider the dynamical system of the nonlinear pendulum. Let $\pmb{x}=(x_1,x_2)$ be the state variables governed by the following equations of motion: 
$$\setlength\abovedisplayskip{4pt}\setlength\belowdisplayskip{4pt}
\dot{x_1}=x_2,\quad \dot{x_2}=-\sin(x_1),\quad\text{ with}\quad \Omega=[-\pi,\pi]_{\mathrm{per}}\times \mathbb{R},
$$
where $\omega$ is the standard Lebesgue measure on $\Omega$. We consider the corresponding discrete-time dynamical system by sampling with a time-step $\Delta_t=0.5$. The system is non-chaotic and Hamiltonian, with challenging Koopman operator theory~\cite{lusch2018deep}.

We use the dictionary $\{g,\mathcal{K}g,\ldots, \mathcal{K}^{99}g\}$, with $g(x_1,x_2)=\exp(ix_1)x_2\exp(-x_2^2/2)$. We collect data points on an equispaced tensor product grid corresponding to the periodic trapezoidal quadrature rule with $M_1$ points in the $x_1$ direction and a truncated trapezoidal quadrature rule with $M_2=M_1$ points in the $x_2$ direction. For our problem, these quadrature rules have exponential~\cite{trefethen2014exponentially} and $\smash{\mathcal{O}(\exp(-CM_2^{2/3}))}$~\cite{trefethen2022exactness} convergence, respectively. To simulate the collection of trajectory data, we compute trajectories starting at each initial condition using the \texttt{ode45} command in MATLAB.

\cref{fig:pendulum1} shows approximate eigenfunctions on a log-scale, computed using $M_1=200$. The Koopman operator $\mathcal{K}$ has no normalizable eigenfunctions, but has generalized eigenfunctions supported along unions of contour lines of the action variable \cite{mezic2020spectrum}. The eigenfunctions produced by \texttt{mpEDMD} are much more localized along these contour lines and better approximate the generalized eigenfunctions than EDMD, whose approximate eigenfunctions are blurred. \cref{fig:pendulum2} (left) shows the eigenvalues of $\mathbb{K}$ and $\mathbb{K}_{\mathrm{EDMD}}$. The eigenvalues of $\mathbb{K}_{\mathrm{EDMD}}$ lie strictly inside the unit disc, corresponding to spectral pollution. Note that this spectral pollution has nothing to do with any stability issues, but instead is due to the discretization of the infinite-dimensional operator $\mathcal{K}$ by a finite matrix. In contrast, \texttt{mpEDMD} does not suffer from spectral pollution.

\begin{figure}[!tbp]
 \centering\vspace{-2mm}
\begin{minipage}[b]{0.24\textwidth}
  \begin{overpic}[width=\textwidth,trim={28mm 0mm 25mm 0mm},clip]{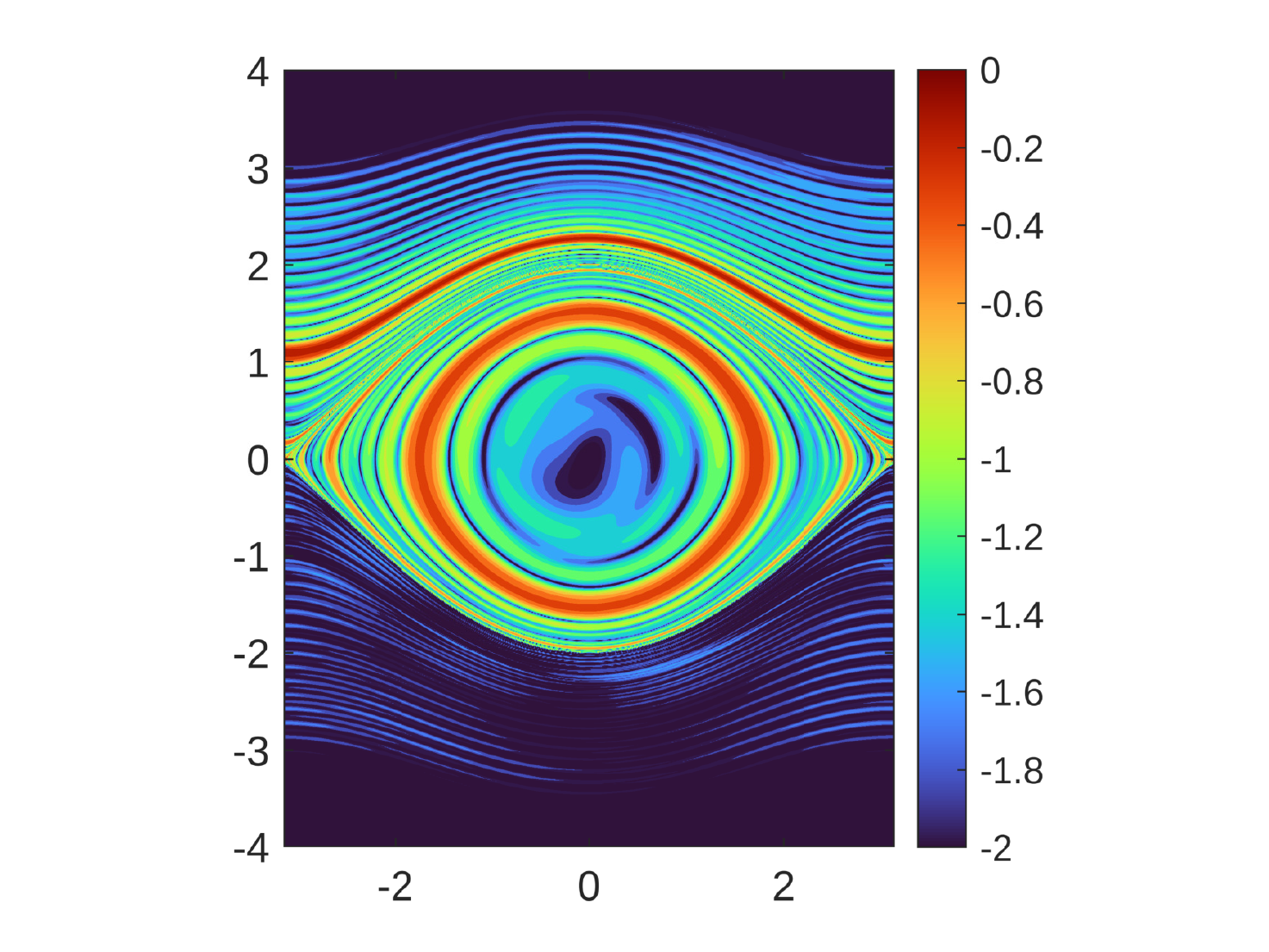}
	\put (5,96) {\texttt{mpEDMD}, $\lambda\approx e^{i\pi/4}$}
		\put (34,-1) {$x_1$}
		\put (-10,50) {$x_2$}
   \end{overpic}
 \end{minipage}
\begin{minipage}[b]{0.24\textwidth}
  \begin{overpic}[width=\textwidth,trim={28mm 0mm 25mm 0mm},clip]{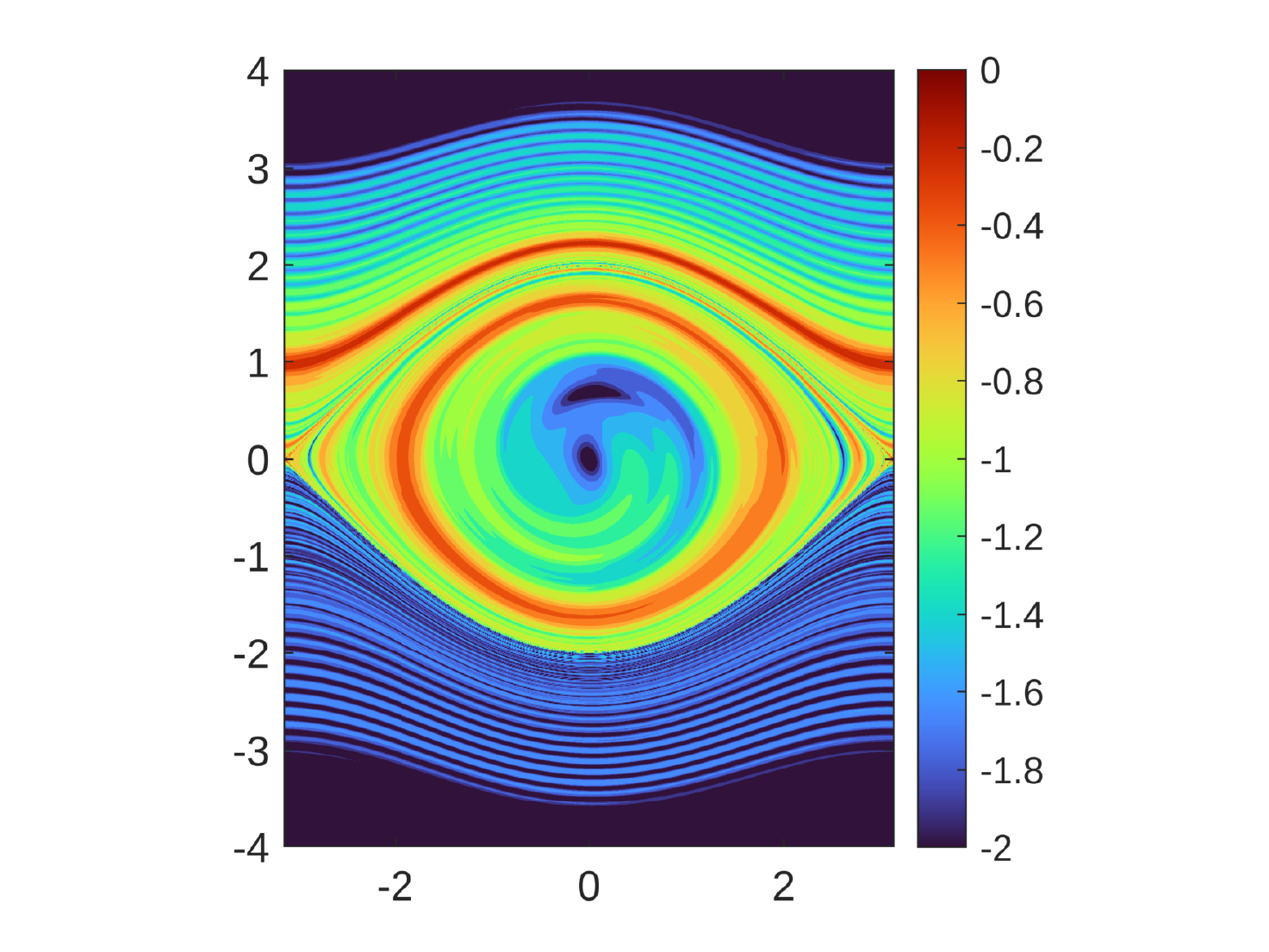}
	\put (5,96) {EDMD, $\lambda\approx e^{i\pi/4}$}
		\put (34,-1) {$x_1$}
   \end{overpic}
 \end{minipage}
\begin{minipage}[b]{0.24\textwidth}
  \begin{overpic}[width=\textwidth,trim={28mm 0mm 25mm 0mm},clip]{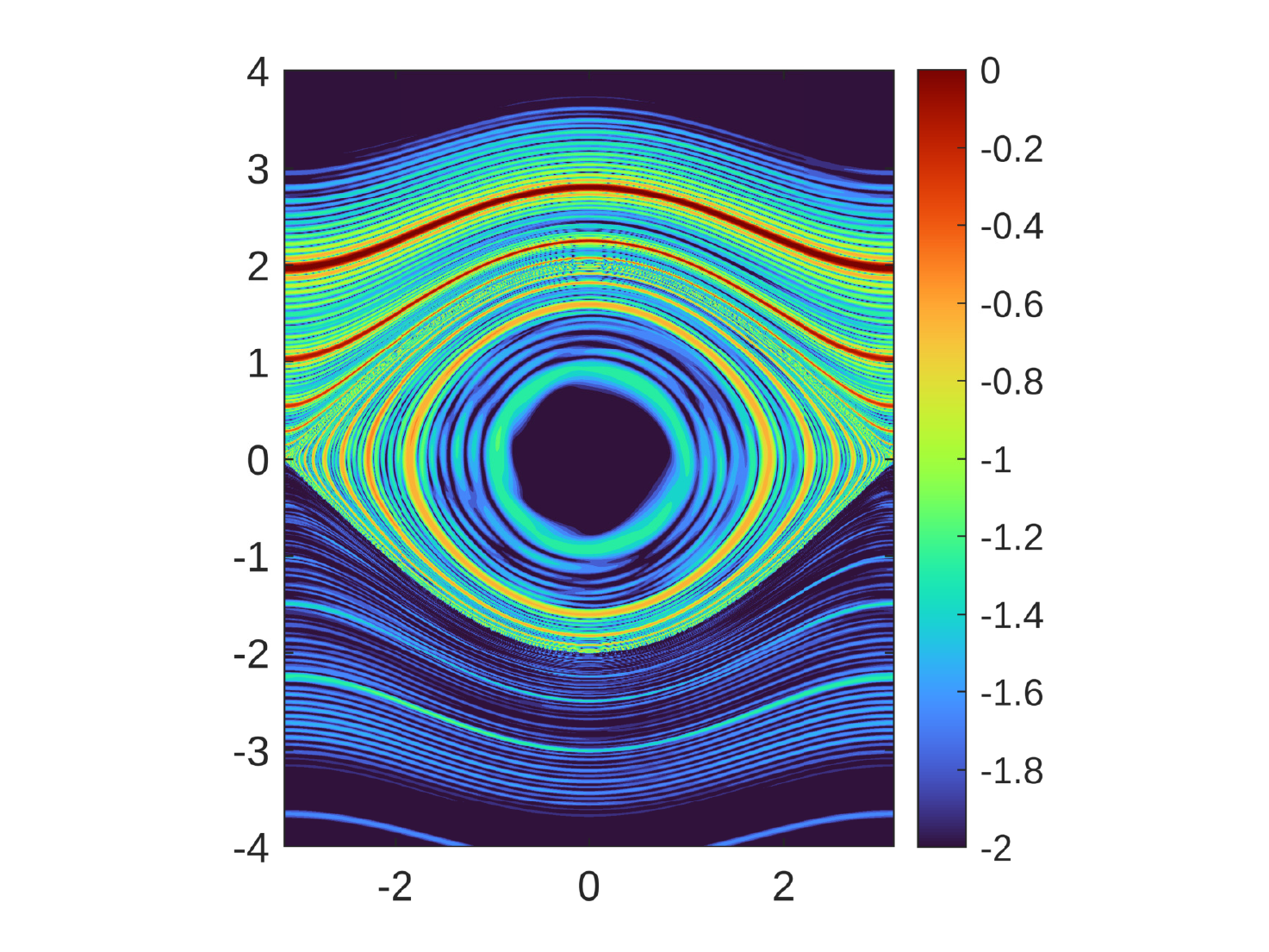}
	\put (5,96) {\texttt{mpEDMD}, $\lambda\approx e^{i3\pi/4}$}
		\put (34,-1) {$x_1$}
   \end{overpic}
 \end{minipage}
\begin{minipage}[b]{0.24\textwidth}
  \begin{overpic}[width=\textwidth,trim={28mm 0mm 25mm 0mm},clip]{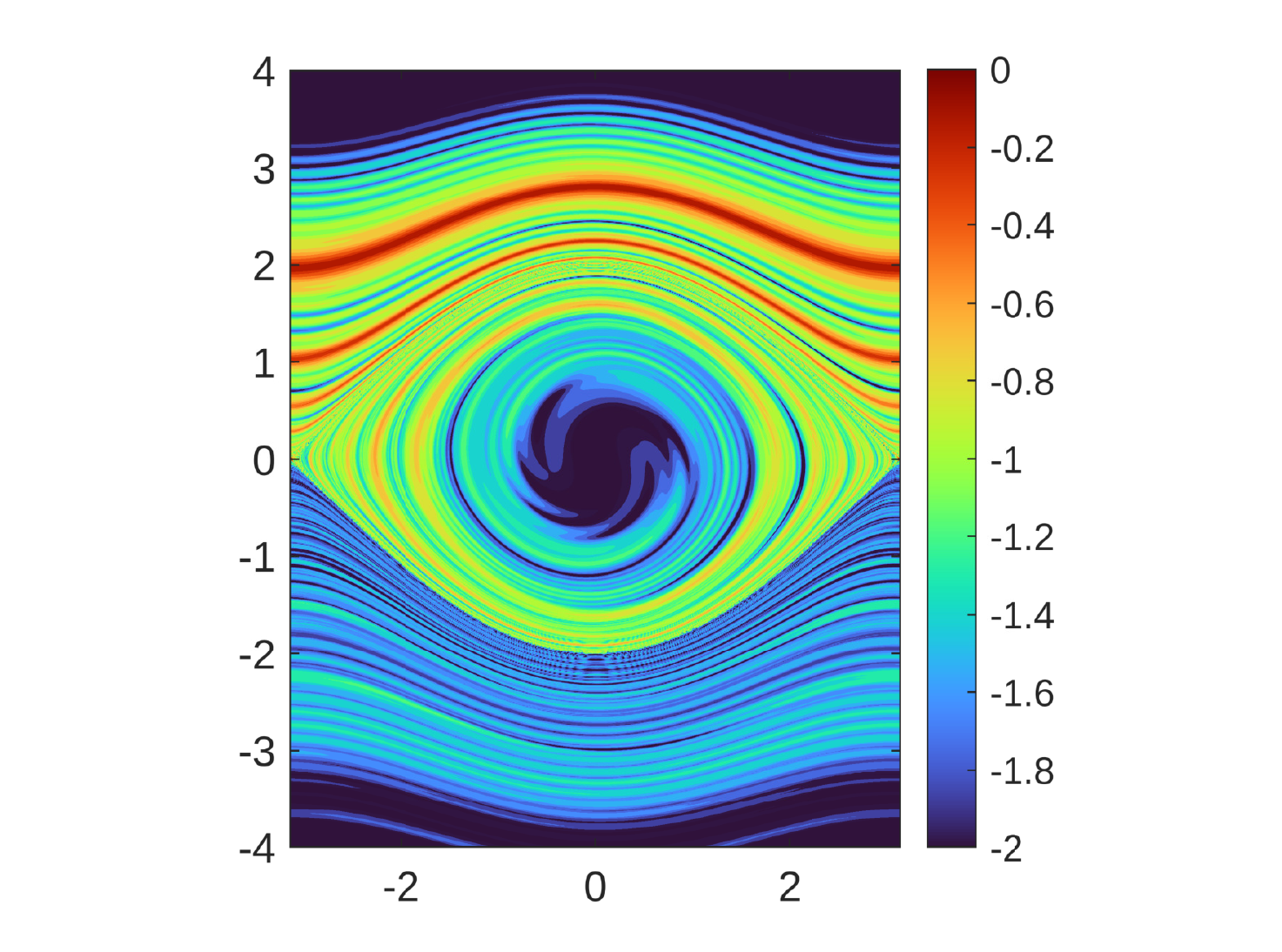}
	\put (5,96) {EDMD, $\lambda\approx e^{i3\pi/4}$}
		\put (34,-1) {$x_1$}
   \end{overpic}
 \end{minipage}
\vspace{-2mm}
	  \caption{Eigenfunctions $\log_{10}(|v|)$, where $v=\Psi \pmb{v}$ is normalized and $\pmb{v}$ is the eigenvector of $\mathbb{K}$ (\texttt{mpEDMD}) or $\mathbb{K}_{\mathrm{EDMD}}$ (EDMD). In each case we plot the eigenfunction with eigenvalue nearest to the shown value of $\lambda$. Taking $\lambda\rightarrow\overline{\lambda}$ yields the corresponding eigenfunctions reflected in $x_2=0$.}\vspace{-3mm}
\label{fig:pendulum1}
\end{figure}

\begin{figure}[!tbp]
 \centering\vspace{-1mm}
\begin{minipage}[b]{0.4\textwidth}
  \begin{overpic}[width=\textwidth,trim={0mm -3mm 0mm 2mm},clip]{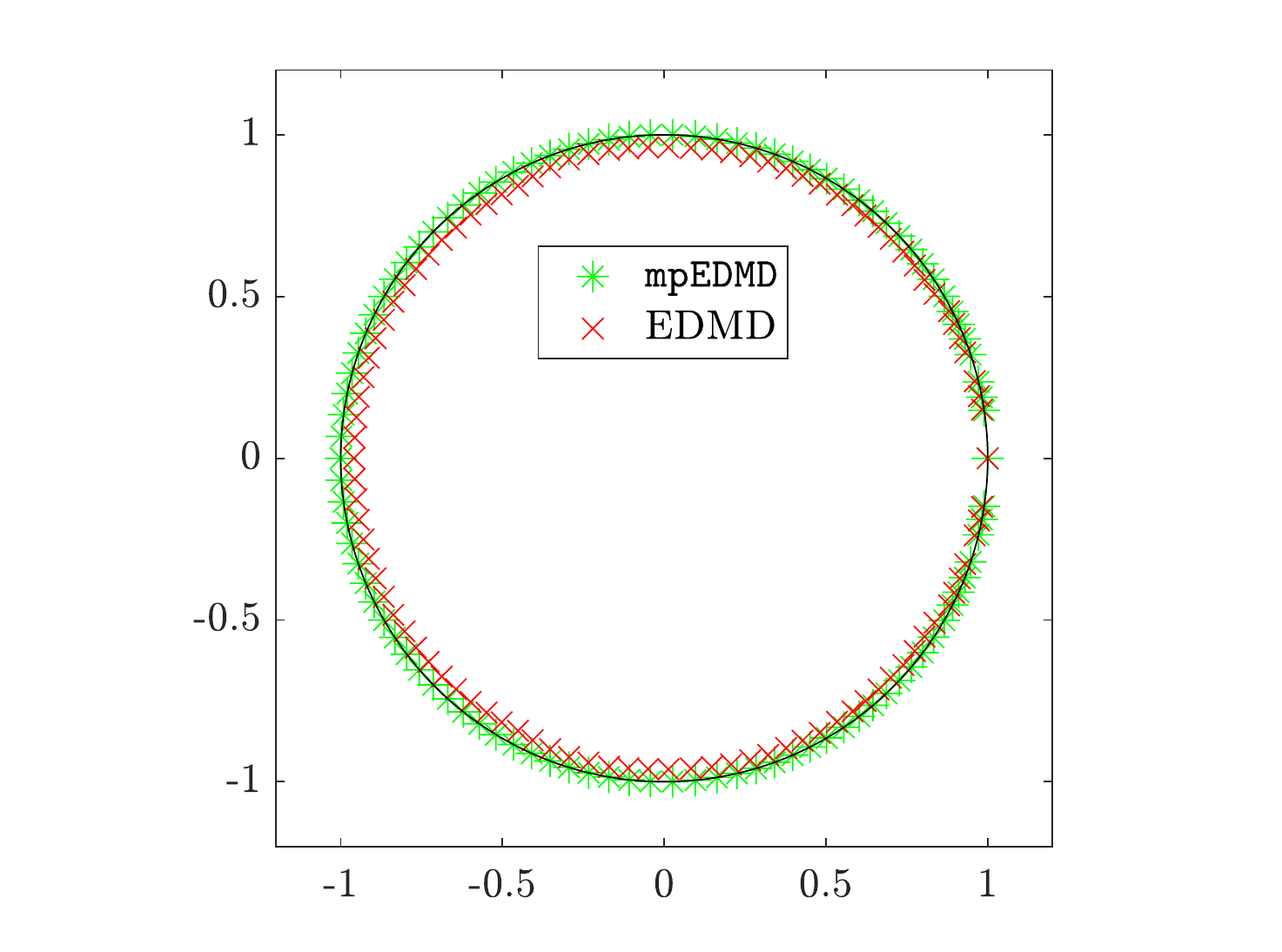}
		\put (45,0) {\small$\mathrm{Re}(\lambda)$}
		\put (10,35) {\small\rotatebox{90}{$\mathrm{Im}(\lambda)$}}
		\put(39,37){\vector(-3,-1){10}}
		\put(41,37){\small{}spectral}
		\put(41,31){\small{}pollution}
   \end{overpic}
 \end{minipage}
 \begin{minipage}[b]{0.4\textwidth}
  \begin{overpic}[width=\textwidth,trim={0mm -3mm 0mm 2mm},clip]{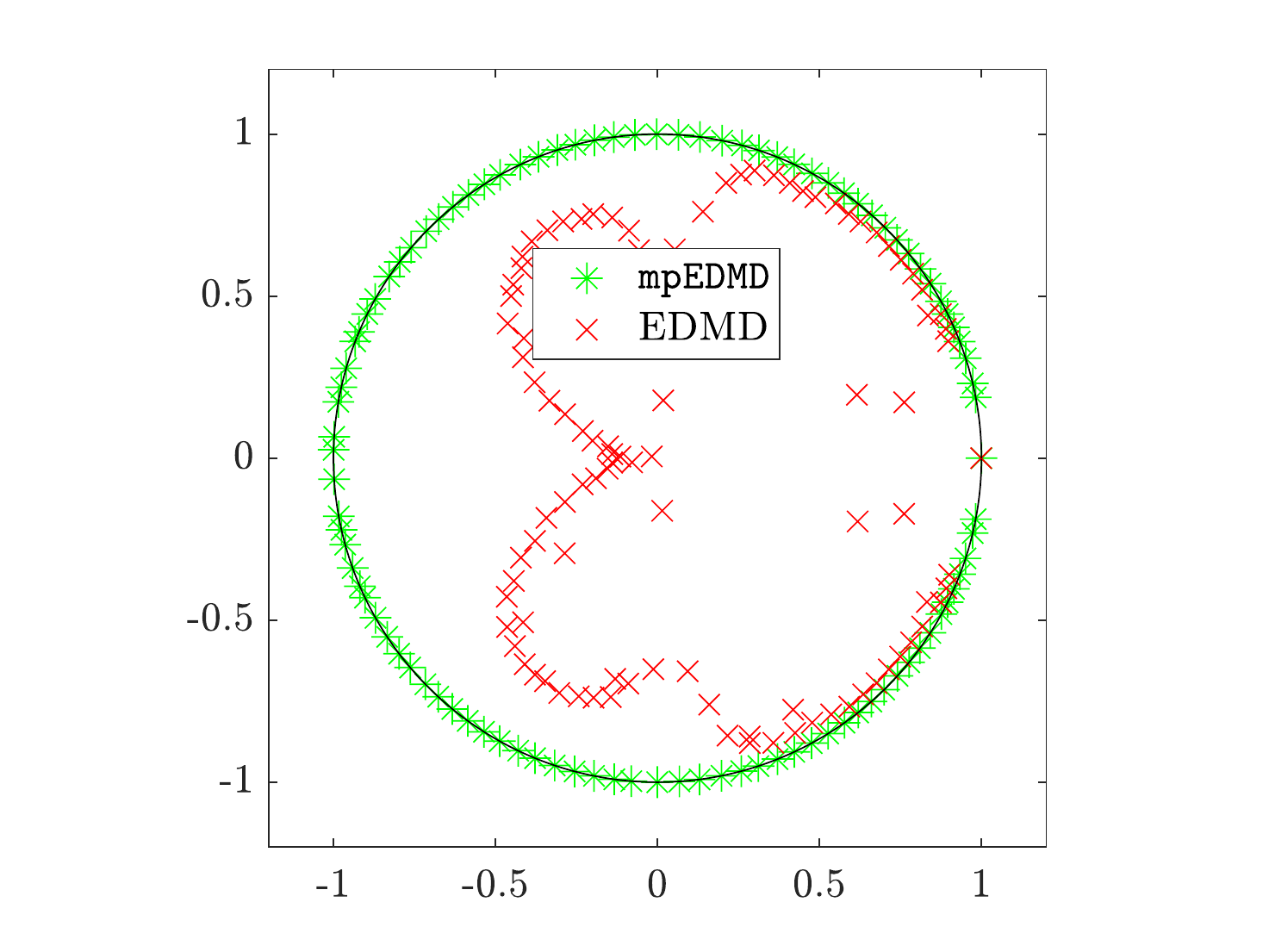}
		\put (45,0) {\small$\mathrm{Re}(\lambda)$}
		\put (10,35) {\small\rotatebox{90}{$\mathrm{Im}(\lambda)$}}
   \end{overpic}
 \end{minipage}
\vspace{-3mm}
	  \caption{Eigenvalues of $\mathbb{K}_{\mathrm{EDMD}}$ (EDMD) and $\mathbb{K}$ (\texttt{mpEDMD}). Left: Noise-free case. Right: $10\%$ Gaussian random noise added to $\Psi_X$ and $\Psi_Y$.}\vspace{-1mm}
\label{fig:pendulum2}
\end{figure}

\begin{figure}[!tbp]
 \centering
\begin{minipage}[b]{0.4\textwidth}
  \begin{overpic}[width=\textwidth,trim={0mm 0mm 0mm 0mm},clip]{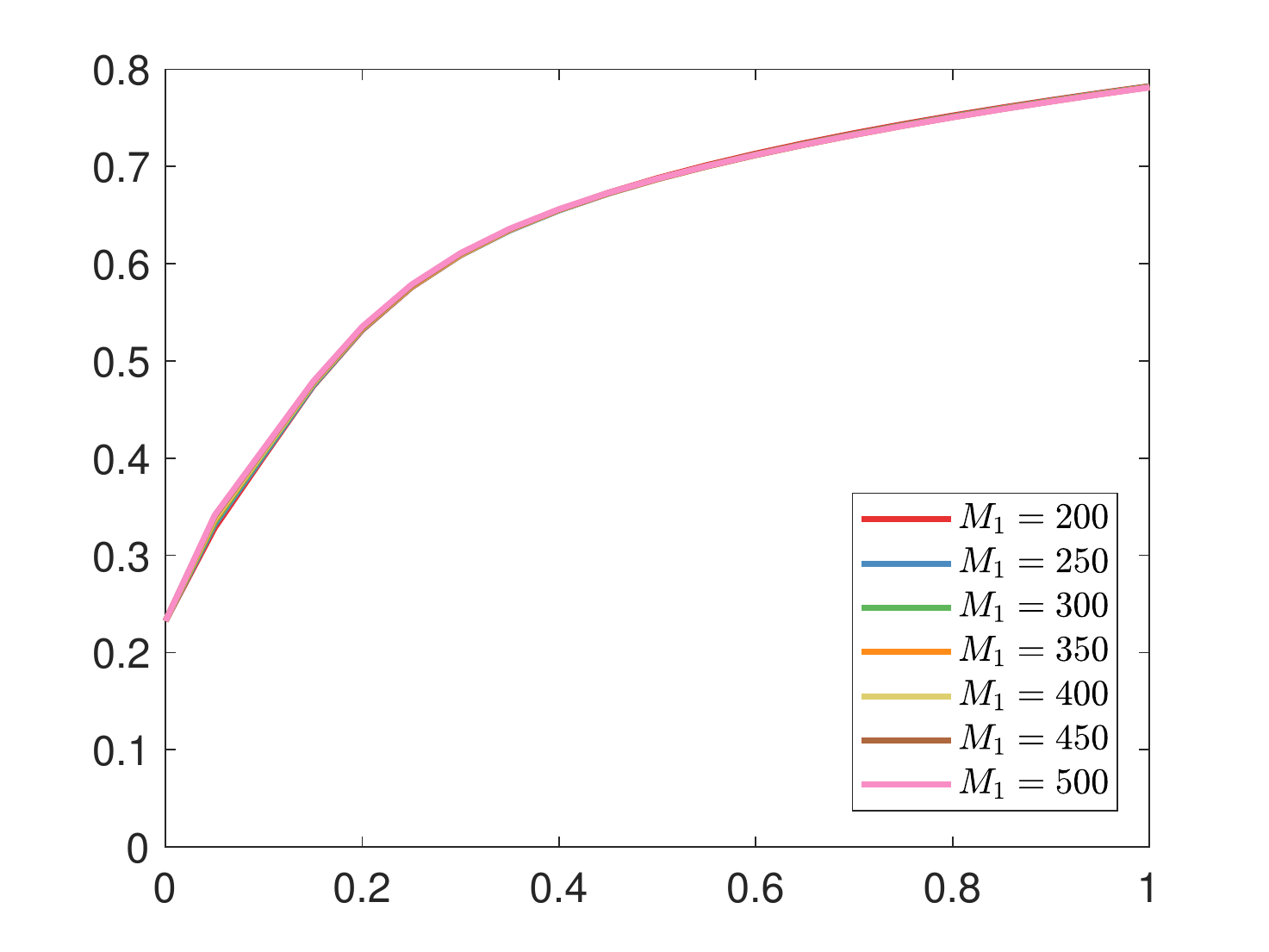}
	\put (18,72) {Mean residual (EDMD)}
   \put (30,-2) {$\tau$ (noise level)}
   \end{overpic}
 \end{minipage}
\begin{minipage}[b]{0.4\textwidth}
  \begin{overpic}[width=\textwidth,trim={0mm 0mm 0mm 0mm},clip]{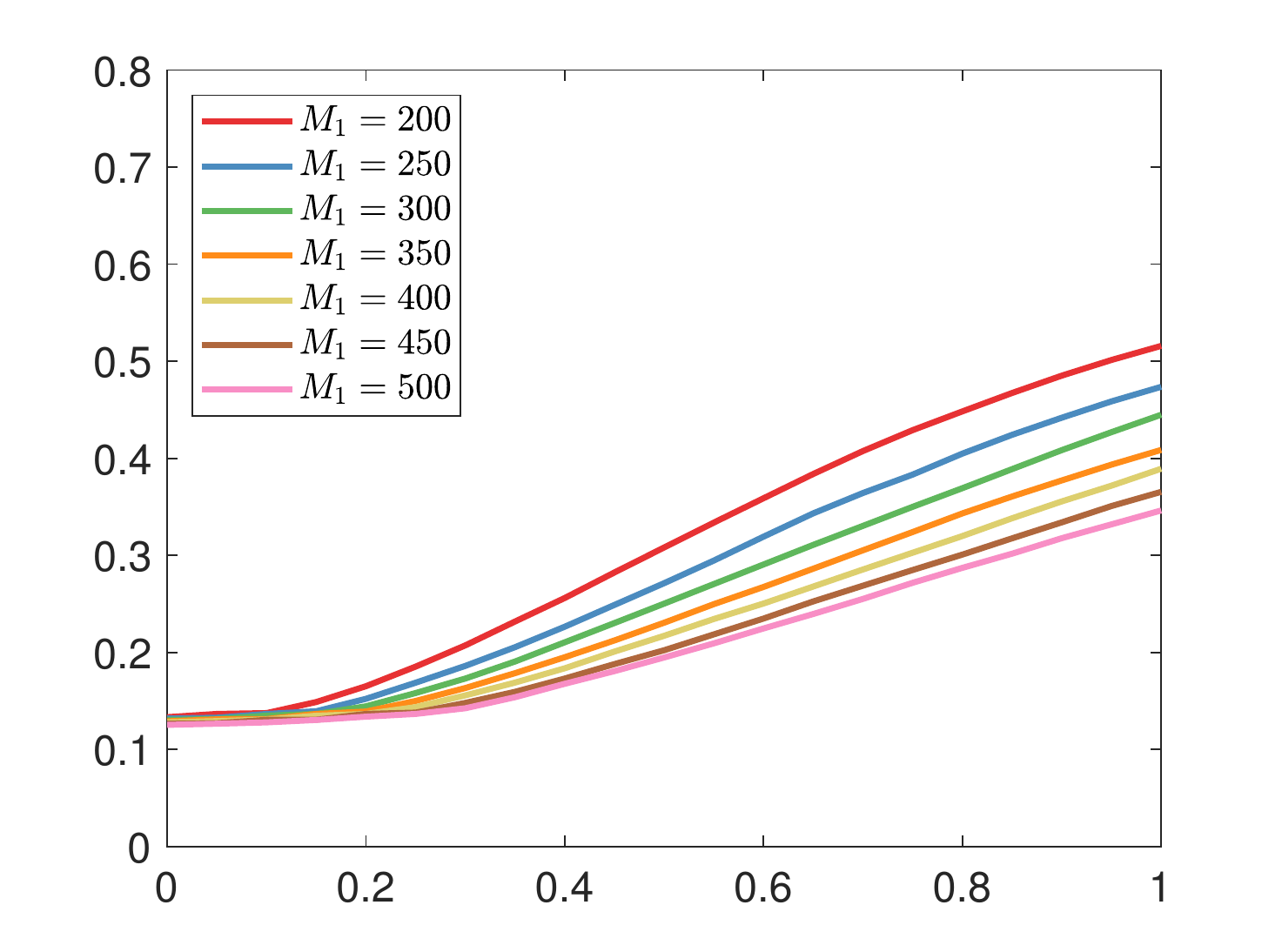}
		\put (18,72) {Mean residual (\texttt{mpEDMD})}
   \put (30,-2) {$\tau$ (noise level)}
   \end{overpic}
 \end{minipage}\vspace{-2mm}
	  \caption{Mean residual over all $N$ eigenpairs and $10$ independent realizations per noise level $\tau$. Residuals are computed using~\eqref{residual_bound} with matrices $G$ and $A$ computed using a larger $M_1$ and $\tau=0$.}\vspace{-2mm}
\label{fig:pendulum3}
\end{figure}

Noise is a substantial problem for most DMD methods, and a common remedy is to consider a total least squares (TLS) problem~\cite{dawson2016characterizing}. The solution to the orthogonal Procrustes problem~\eqref{EDMD_opt_prob4} is also the solution to the corresponding constrained TLS problem~\cite{arun1992unitarily}. Hence, \texttt{mpEDMD} is optimally robust when noise is present in both data matrices in~\eqref{EDMD_opt_prob4}~\cite{van1991total}. We test the robustness to noise by adding $\tau$ Gaussian random noise to the measurement matrices $\Psi_X$ and $\Psi_Y$ in~\eqref{psidef}. \cref{fig:pendulum2} (right) shows the effect of noise on the eigenvalues of $\mathbb{K}$ and $\mathbb{K}_{\mathrm{EDMD}}$ for $\tau=0.1$ ($10\%$ noise). The deterioration of the spectrum of $\mathbb{K}_{\mathrm{EDMD}}$ is clear. To further investigate robustness, we compute the (relative) residual of eigenpairs using~\eqref{residual_bound} with noise-free matrices $G,A$ computed using large $M_1$.~\cref{fig:pendulum3} plots the mean residual over all $N=200$ eigenpairs and $10$ independent noise realizations against the noise level $\tau$. We see that \texttt{mpEDMD} is much more robust to noise than EDMD. Moreover, for a given noise level $\tau$, the accuracy of \texttt{mpEDMD} increases as $M_1$ increases. A full statistical analysis of this phenomenon is beyond the scope of this paper, but we note that this type of behavior, known as strongly consistent estimation, is typical of TLS~\cite[Chapter 8]{van1991total}. This phenomenon does not happen with EDMD in~\cref{fig:pendulum3}.

\subsection{Conservation of energy and statistics for turbulent boundary layer flow}
\label{sec:turbulent_flow}
As our final example, we consider the boundary layer generated by a thin jet of height $12.7$mm injecting air onto a smooth flat wall. Experiments are performed at the Wall Jet Wind Tunnel of Virginia Tech~\cite{szoke2021flow}. A two-component time-resolved particle image velocimetry system is used to capture $1000$ snapshots of the two-dimensional velocity field of the wall-jet flow over a spatial grid and a time period of 1s. The streamwise origin of the field-of-view is 1282.7mm downstream of the wall-jet nozzle. We use a jet velocity of 50m/s, corresponding to a jet Reynolds number of $6.4 \times 10^4$. The length and height of the field-of-view is approximately 75mm $\times$ 40mm, and the spatial resolution of the measurements is $\approx$0.24mm. This corresponds to dimension $d=102300$ in \eqref{eq:DynamicalSystem}. We use a full SVD of the data matrix to form a dictionary, as outlined in~\cref{sec:EDMD_recap}. The flow consists of two main regions. Within the region bounded by the wall and the peak in the velocity profile, the flow exhibits the properties of a zero pressure gradient turbulent boundary layer. Above this fluid portion, the flow is dominated by a two-dimensional shear layer consisting of large, energetic flow structures. This example is a considerable challenge for regular DMD approaches due to multiple turbulent scales expected within the boundary layer.

We investigate the conservation of energy and statistics of the flow when using the KMD in~\cref{KMD_rem}. We consider the velocity profiles predicted by \texttt{mpEDMD}, EDMD, and piDMD over a time period of 5s (five times the window of observations), and averaged over $100$ random initializations $\pmb{x}_0$. \cref{fig:turbulence1} (left, middle) shows the turbulent kinetic energy (TKE) of the predictions, averaged in the (homogenous) horizontal direction, at vertical heights in the boundary layer (left panel) and in the shear layer (middle panel). The instability of the KMD for EDMD is clear. Whilst piDMD is stable and approximately conservative, it does not preserve the correct values of TKE. In contrast, \texttt{mpEDMD} conserves the correct TKE. \cref{fig:turbulence1} (right) highlights this by showing the time averaged TKE prediction of \texttt{mpEDMD} and piDMD as a function of the vertical height. The relative error of \texttt{mpEDMD} is bounded by $0.001$. These results underline the importance, even in the case of linear dictionary functions, of the non-trivial matrix $G$ in~\cref{alg:mp_EDMD}.

\begin{figure}[!tbp]
 \centering\vspace{-2mm}
 \begin{minipage}[b]{0.32\textwidth}
  \begin{overpic}[width=\textwidth,trim={0mm 0mm 0mm 0mm},clip]{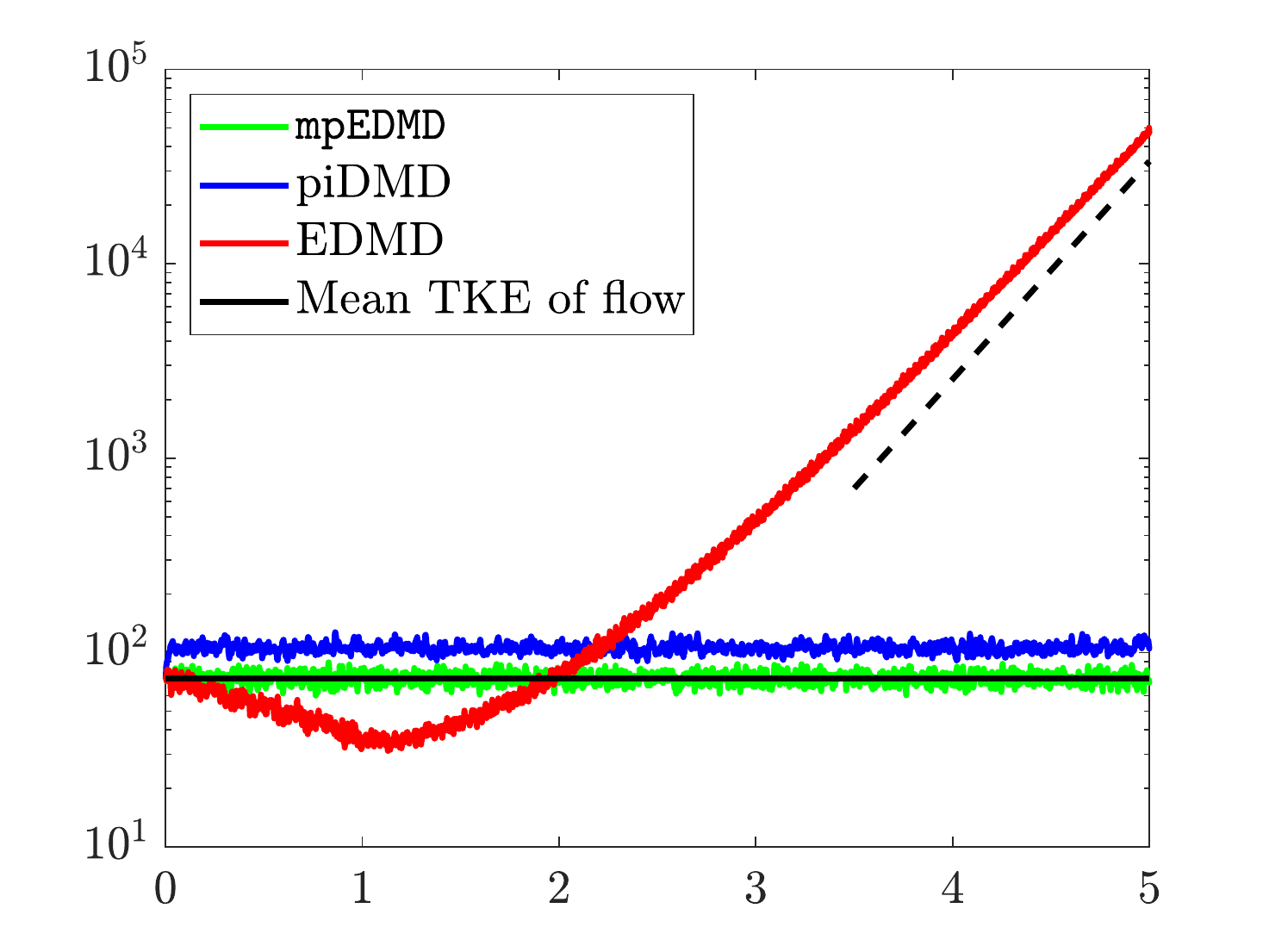}
	\put (26,72) {\small{TKE $y\approx 5$mm}}
   \put (37,-3) {\small{Time (s)}}
   \end{overpic}
 \end{minipage}
	\begin{minipage}[b]{0.32\textwidth}
  \begin{overpic}[width=\textwidth,trim={0mm 0mm 0mm 0mm},clip]{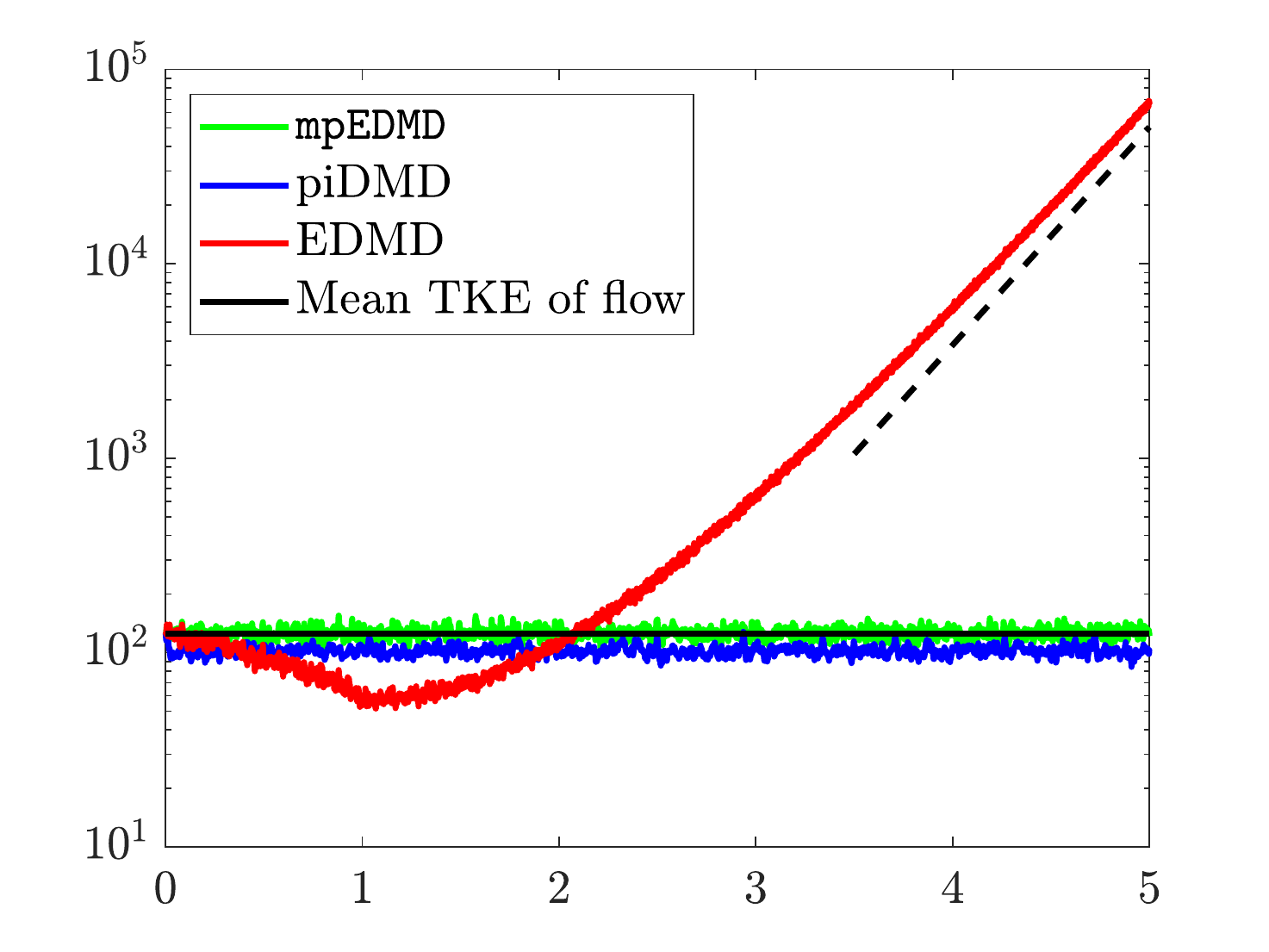}
\put (23,72) {\small{TKE, $y\approx 35$mm}}
   \put (37,-3) {\small{Time (s)}}
   \end{overpic}
 \end{minipage}
\begin{minipage}[b]{0.32\textwidth}
  \begin{overpic}[width=\textwidth,trim={0mm 0mm 0mm 0mm},clip]{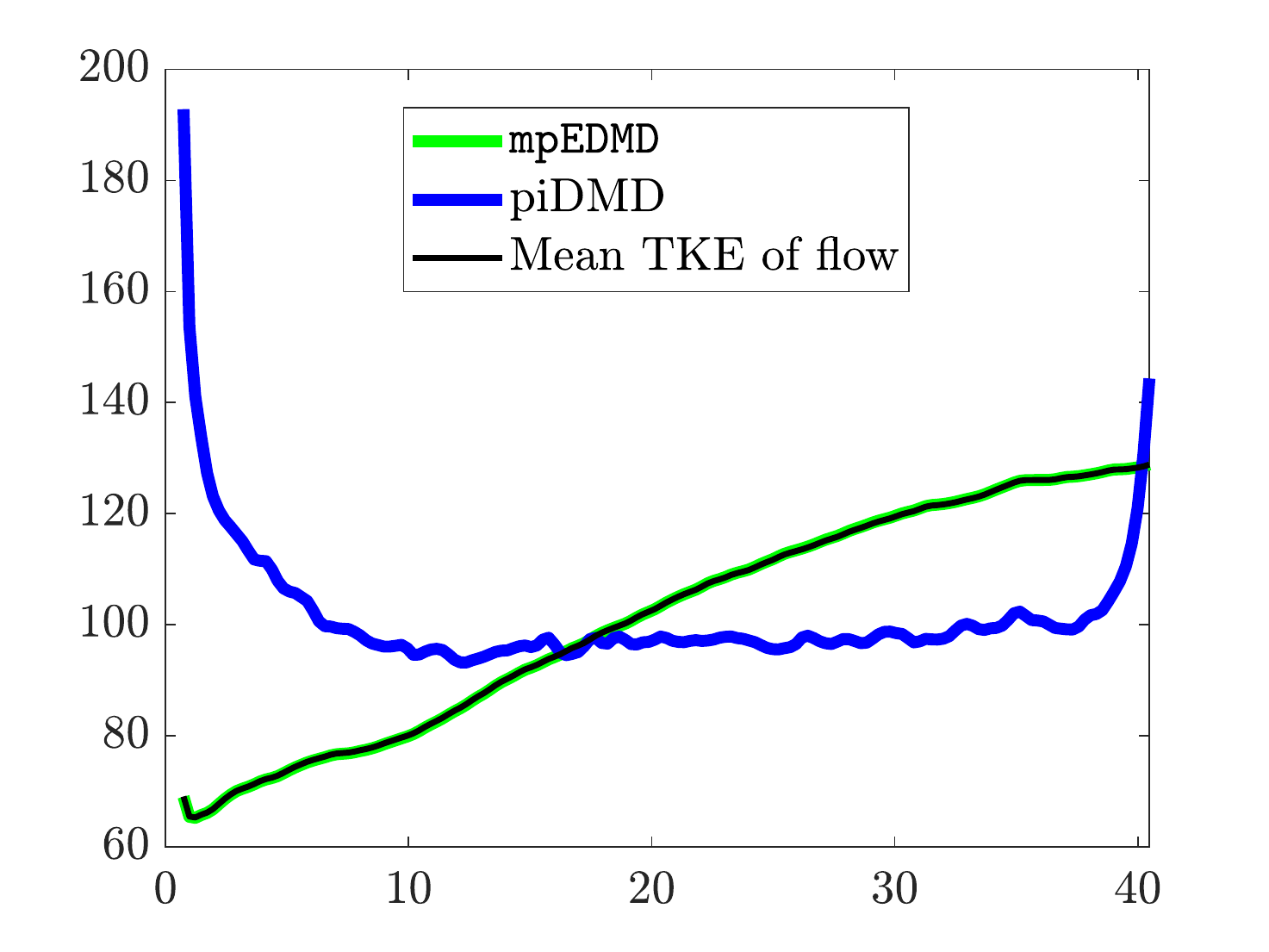}
\put (15,72) {\small{Time-averaged TKE}}
   \put (27,-3) {\small{Height $y$ (mm)}}
   \end{overpic}
 \end{minipage}\vspace{-2mm}
	  \caption{Left and middle: TKE as a function of time, averaged over (homogeneous) horizontal direction. The dashed line shows the expected growth rate of EDMD from the eigenvalues of $\mathbb{K}_{\mathrm{EDMD}}$. Right: TKE as a function of vertical height, averaged over time and horizontal direction. 
}\vspace{-4mm}
\label{fig:turbulence1}
\end{figure}
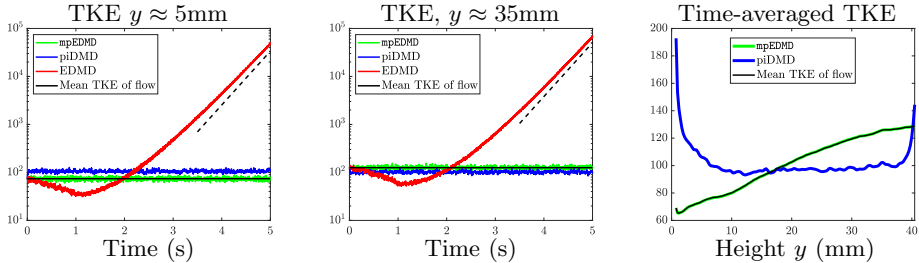

\cref{fig:turbulence1} (top row) shows characteristic predictions of the horizontal component of the velocity field at prediction time 4s. Qualitatively, \texttt{mpEDMD} captures the larger-scale structures above the boundary layer, whereas piDMD does not, and EDMD overpredicts the velocity magnitude. To investigate the statistics of the predictions, \cref{fig:turbulence1} (bottom row) shows the wavenumber spectrum, computed by applying the Fourier transform to spatial autocorrelations of the predictions in the horizontal direction \cite[Chapter 8]{glegg2017aeroacoustics}. The wavenumber spectrum of \texttt{mpEDMD} shows excellent agreement with the flow. In contrast, EDMD and piDMD do not capture the correct turbulent statistics. Whilst we can only ever capture the statistics to the resolution of the collected data, this example provides very promising results for the use of \texttt{mpEDMD} in real-world applications.

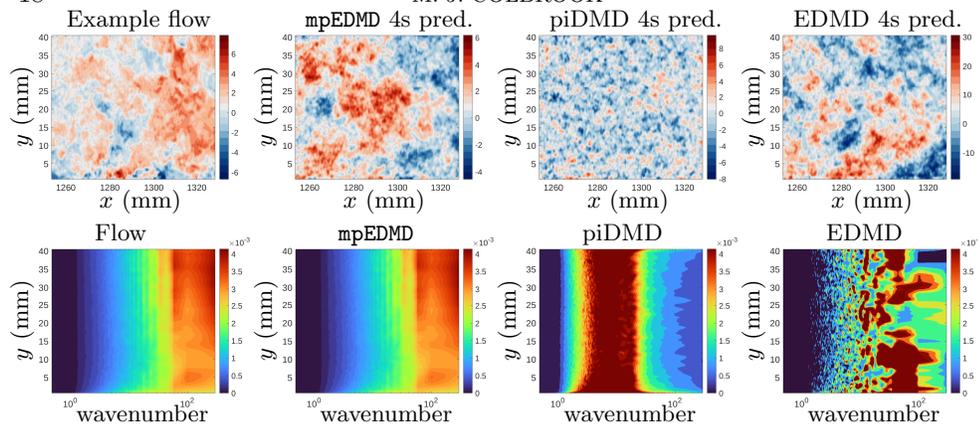
\begin{figure}[!tbp]
 \centering\vspace{-2mm}
\begin{minipage}[b]{0.24\textwidth}
  \begin{overpic}[width=\textwidth,trim={0mm 0mm 0mm 0mm},clip]{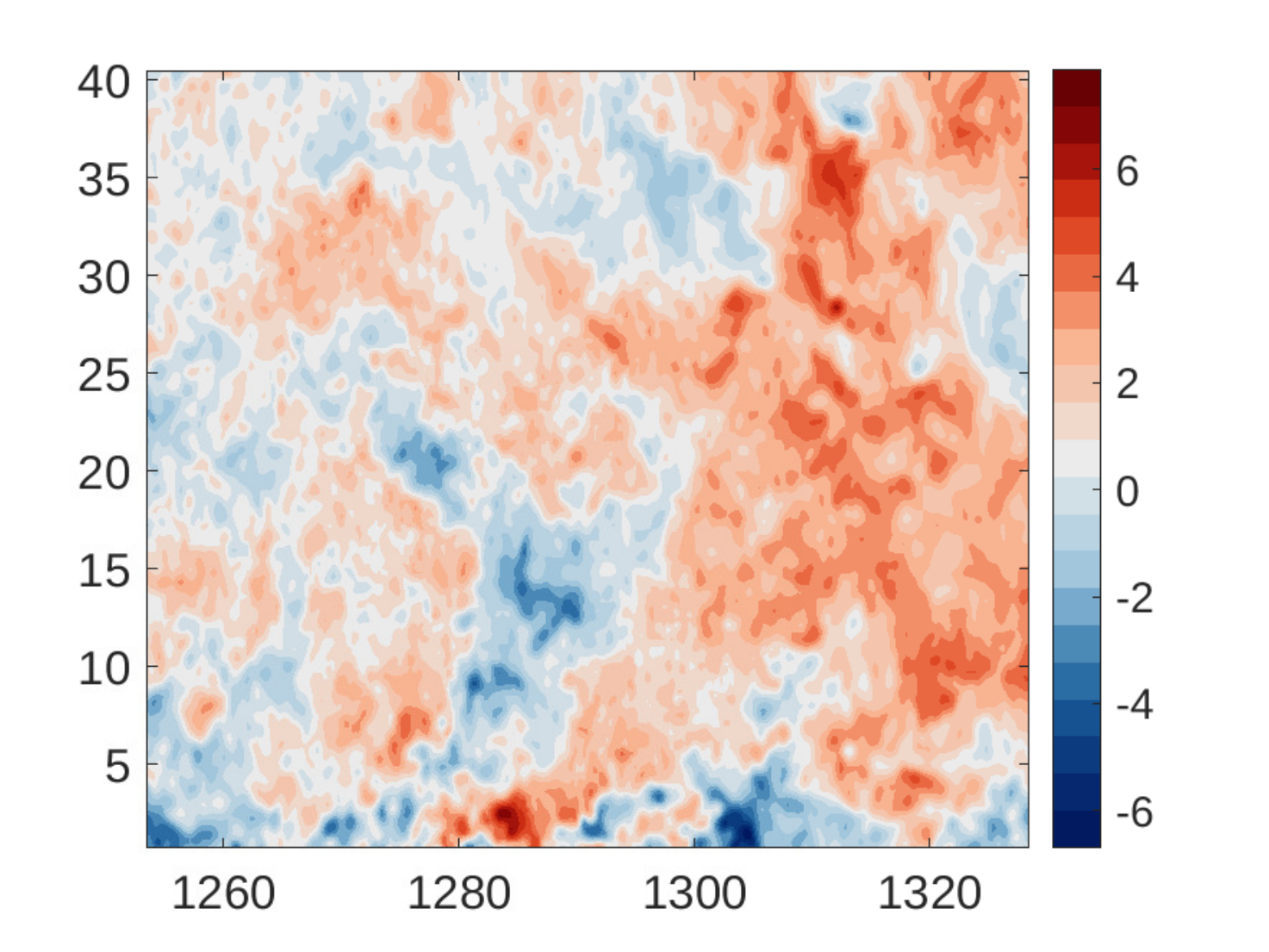}
		\put (-6,22) {\small\rotatebox{90}{$y$ (mm)}}
		\put (32,-4) {\small{$x$ (mm)}}
		\put(18,73) {\small{Example flow}}
   \end{overpic}\vspace{5mm}
 \end{minipage}
 \begin{minipage}[b]{0.24\textwidth}
  \begin{overpic}[width=\textwidth,trim={0mm 0mm 0mm 0mm},clip]{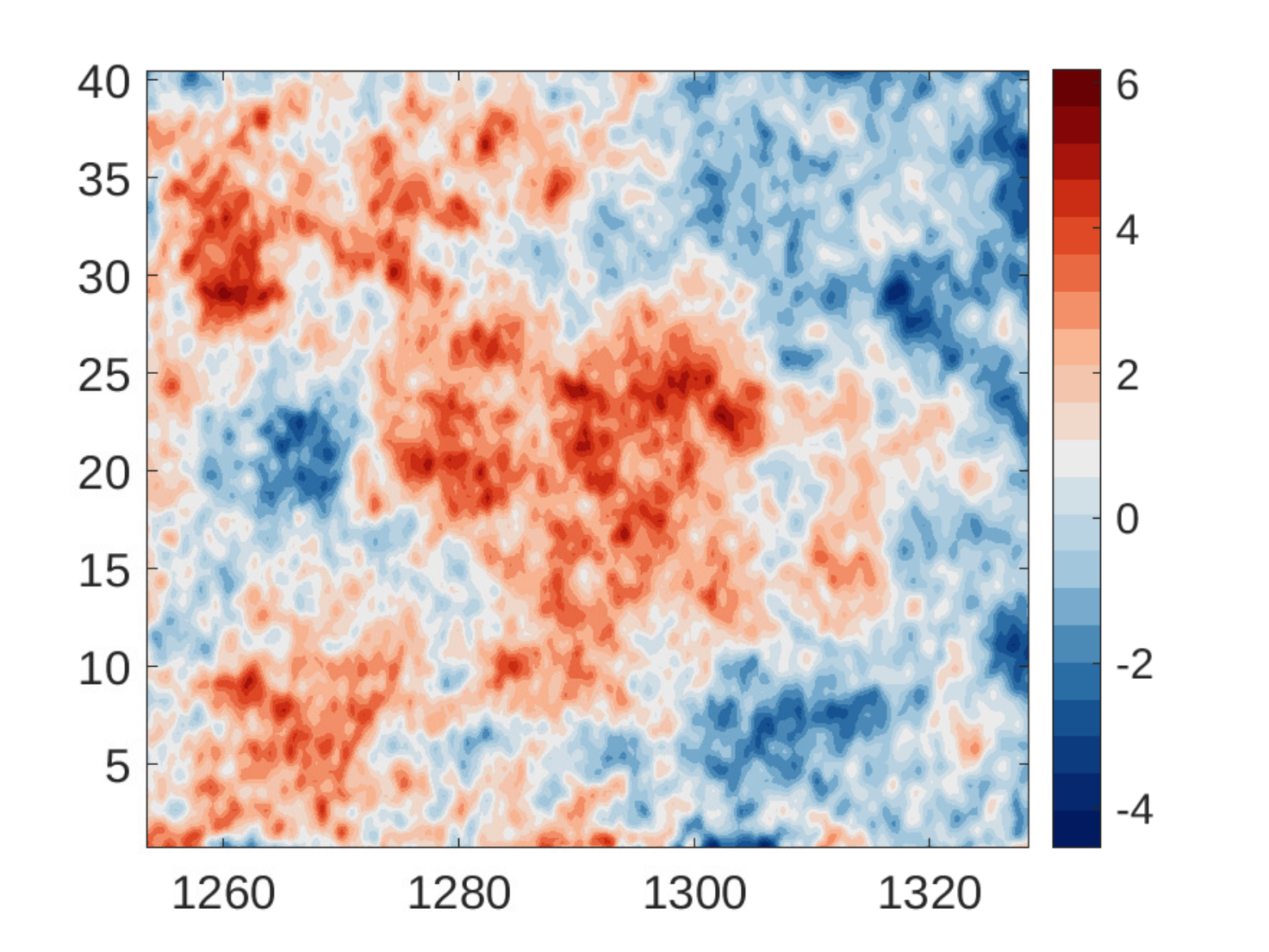}
	\put (-6,22) {\small\rotatebox{90}{$y$ (mm)}}
		\put (32,-4) {\small{$x$ (mm)}}
		\put(16,73) {\small{\texttt{mpEDMD} 4s pred.}}
   \end{overpic}\vspace{5mm}
 \end{minipage}
	\begin{minipage}[b]{0.24\textwidth}
  \begin{overpic}[width=\textwidth,trim={0mm 0mm 0mm 0mm},clip]{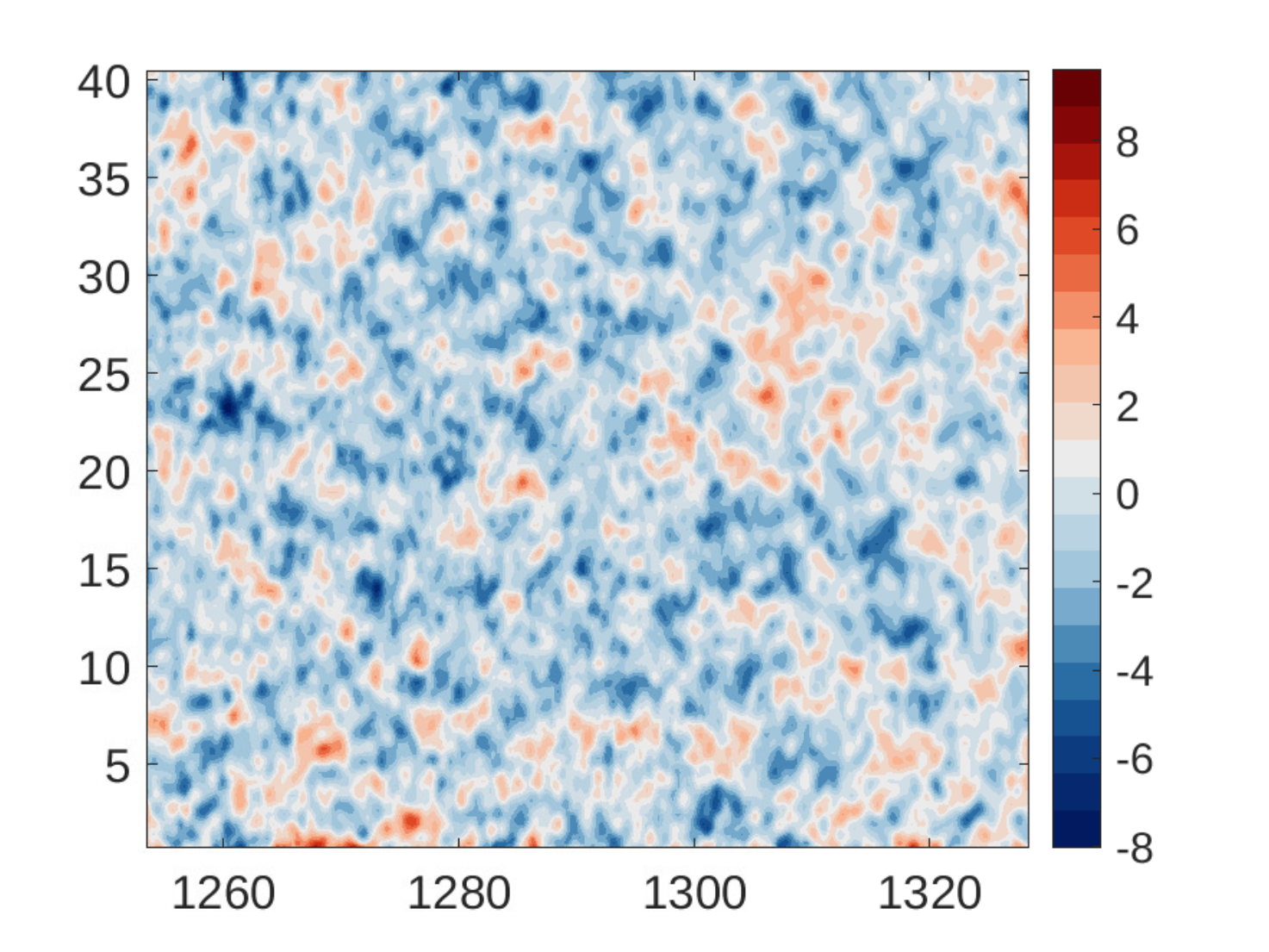}
		\put (-6,22) {\small\rotatebox{90}{$y$ (mm)}}
		\put (32,-4) {\small{$x$ (mm)}}
		\put(16,73) {\small{piDMD 4s pred.}}
   \end{overpic}\vspace{5mm}
 \end{minipage}
\begin{minipage}[b]{0.24\textwidth}
  \begin{overpic}[width=\textwidth,trim={0mm 0mm 0mm 0mm},clip]{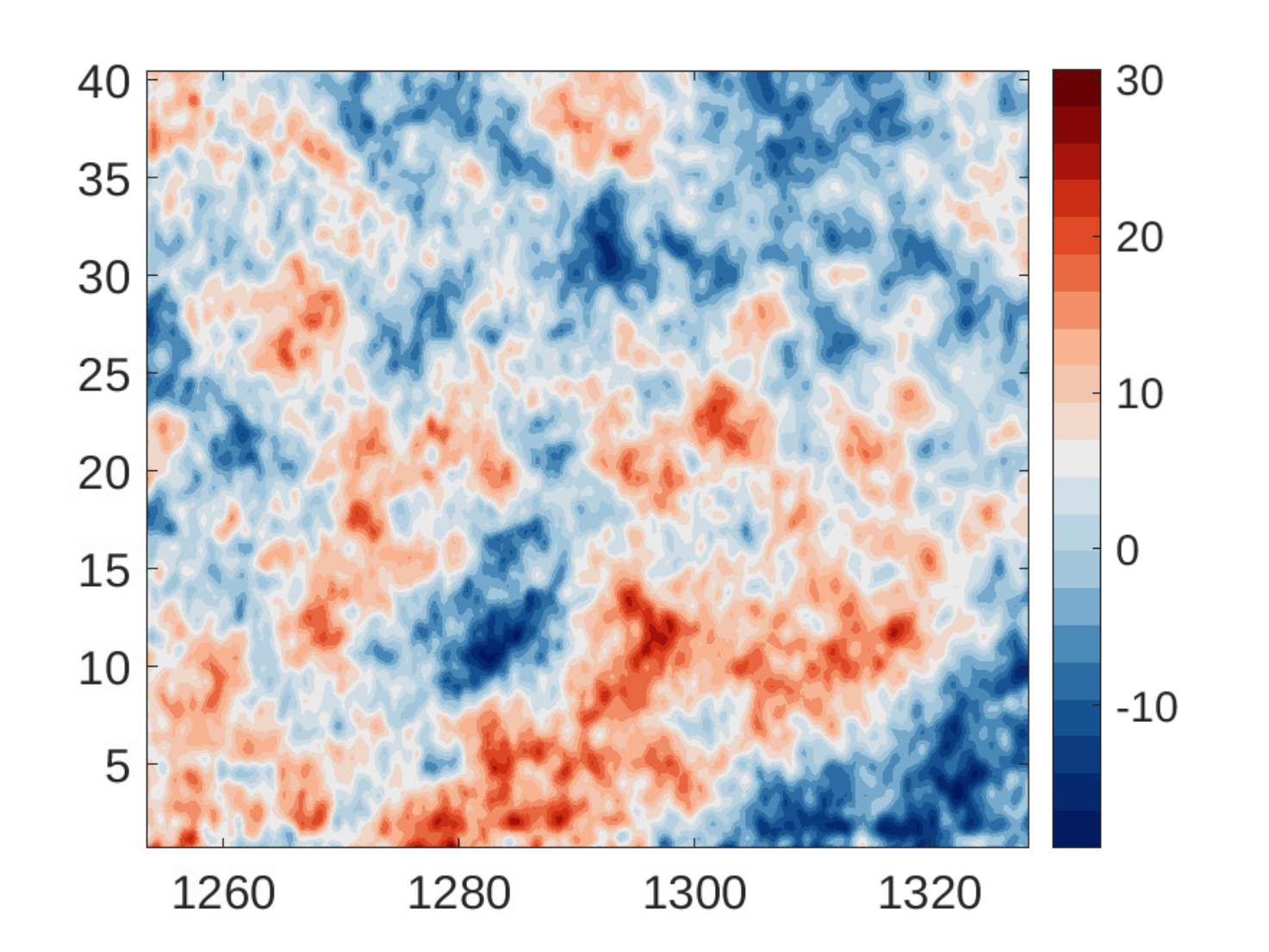}
	\put (-6,22) {\small\rotatebox{90}{$y$ (mm)}}
		\put (32,-4) {\small{$x$ (mm)}}
		\put(16,73) {\small{EDMD 4s pred.}}
   \end{overpic}\vspace{5mm}
 \end{minipage}

\begin{minipage}[b]{0.24\textwidth}
  \begin{overpic}[width=\textwidth,trim={0mm 0mm 0mm 0mm},clip]{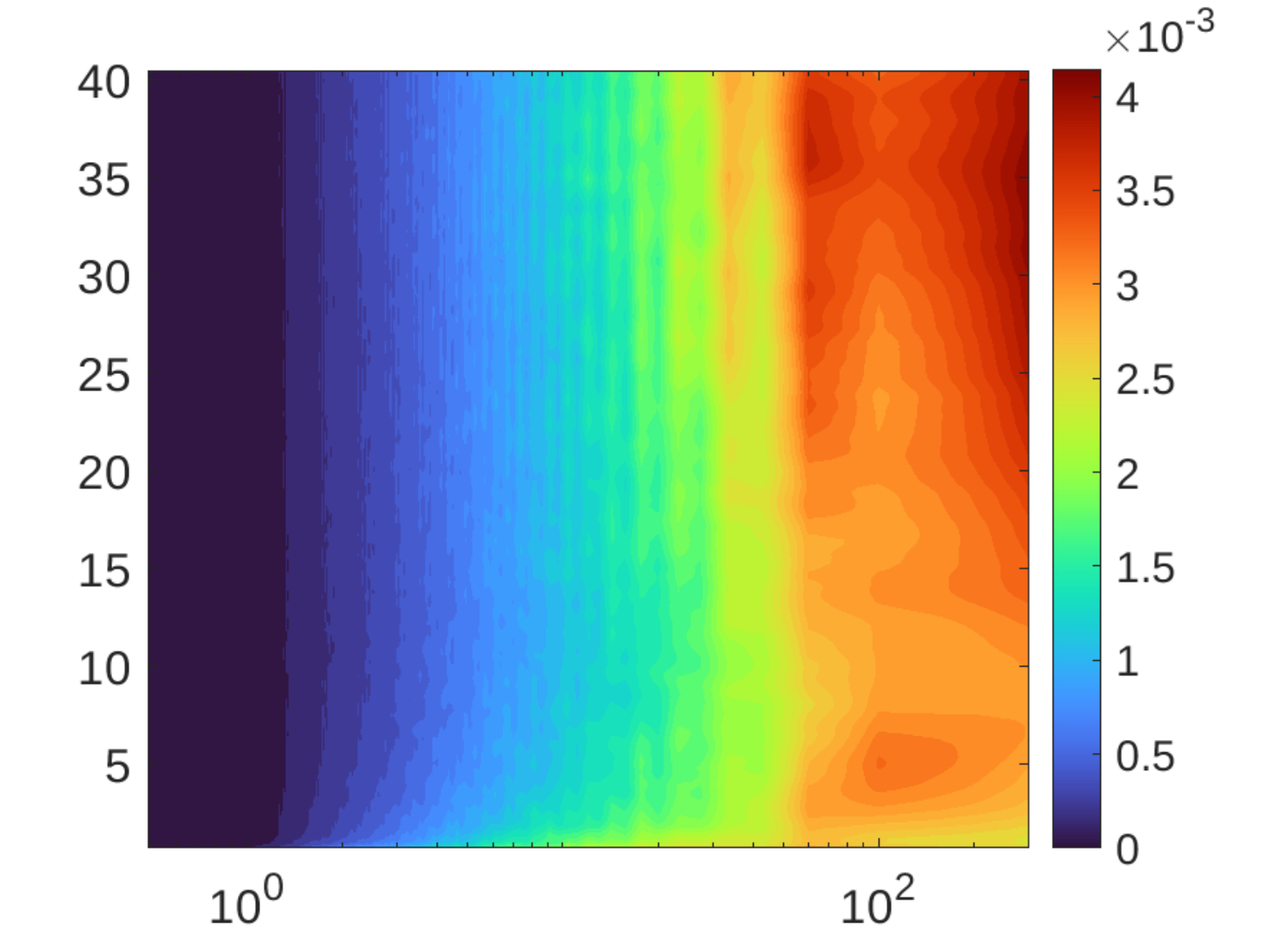}
		\put (-6,22) {\small\rotatebox{90}{$y$ (mm)}}
		\put (22,-4) {\small{wavenumber}}
		\put(30,73) {\small{Flow}}
   \end{overpic}
 \end{minipage}
 \begin{minipage}[b]{0.24\textwidth}
  \begin{overpic}[width=\textwidth,trim={0mm 0mm 0mm 0mm},clip]{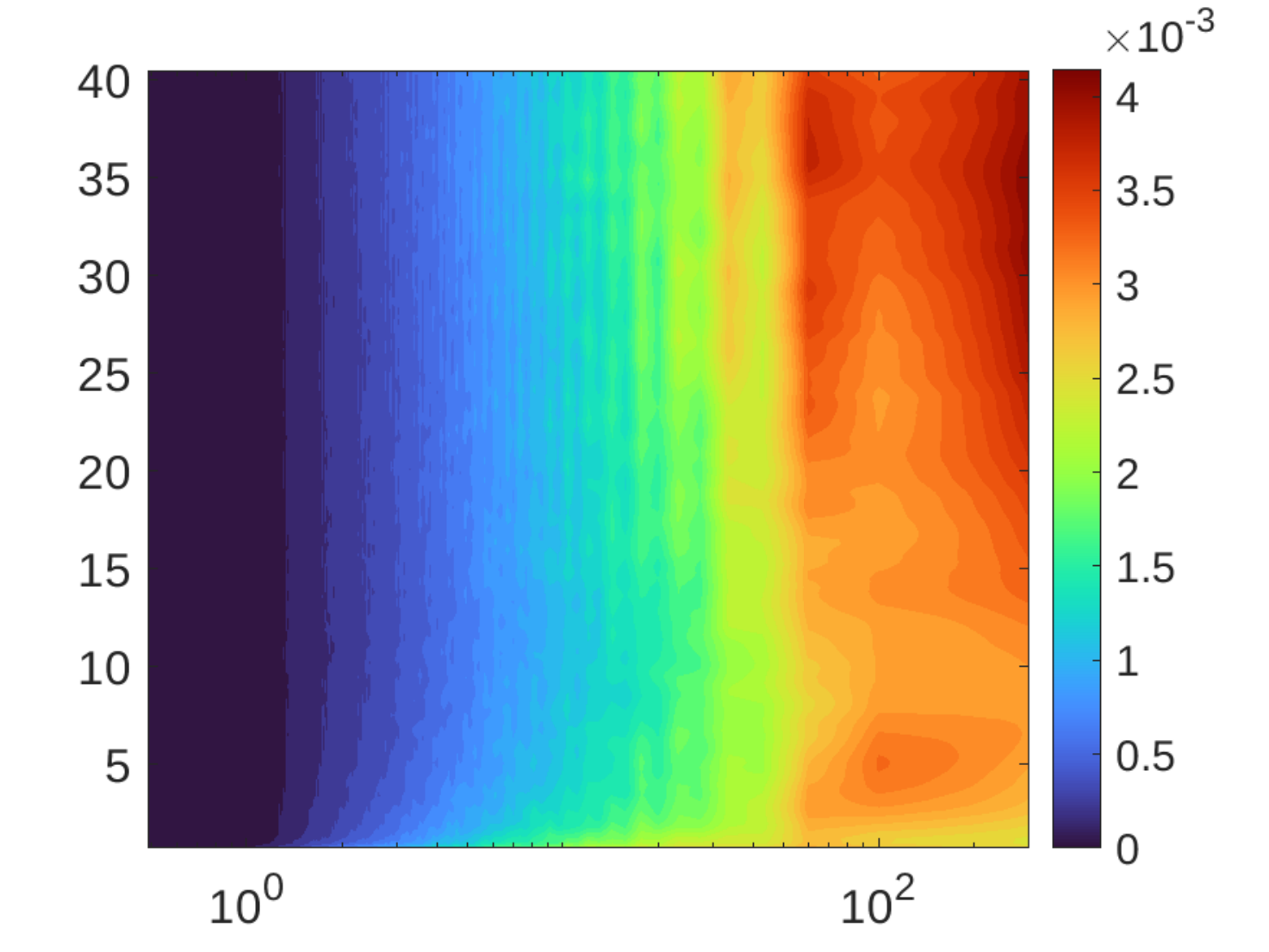}
		\put (-6,22) {\small\rotatebox{90}{$y$ (mm)}}
		\put (22,-4) {\small{wavenumber}}
		\put(30,73) {\small{\texttt{mpEDMD}}}
   \end{overpic}
 \end{minipage}
	\begin{minipage}[b]{0.24\textwidth}
  \begin{overpic}[width=\textwidth,trim={0mm 0mm 0mm 0mm},clip]{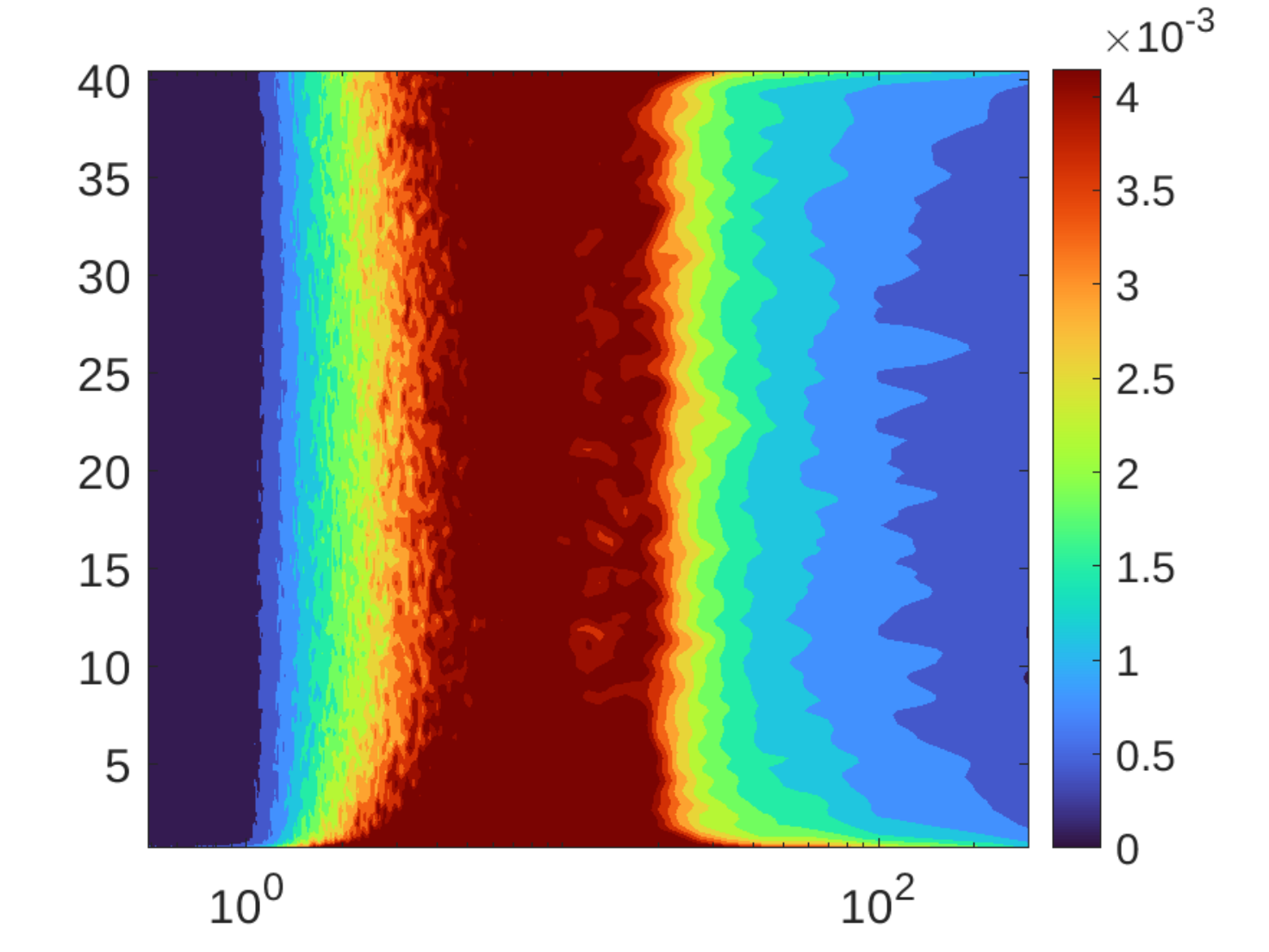}
\put (-6,22) {\small\rotatebox{90}{$y$ (mm)}}
		\put (22,-4) {\small{wavenumber}}
		\put(30,73) {\small{piDMD}}
   \end{overpic}
 \end{minipage}
\begin{minipage}[b]{0.24\textwidth}
  \begin{overpic}[width=\textwidth,trim={0mm 0mm 0mm 0mm},clip]{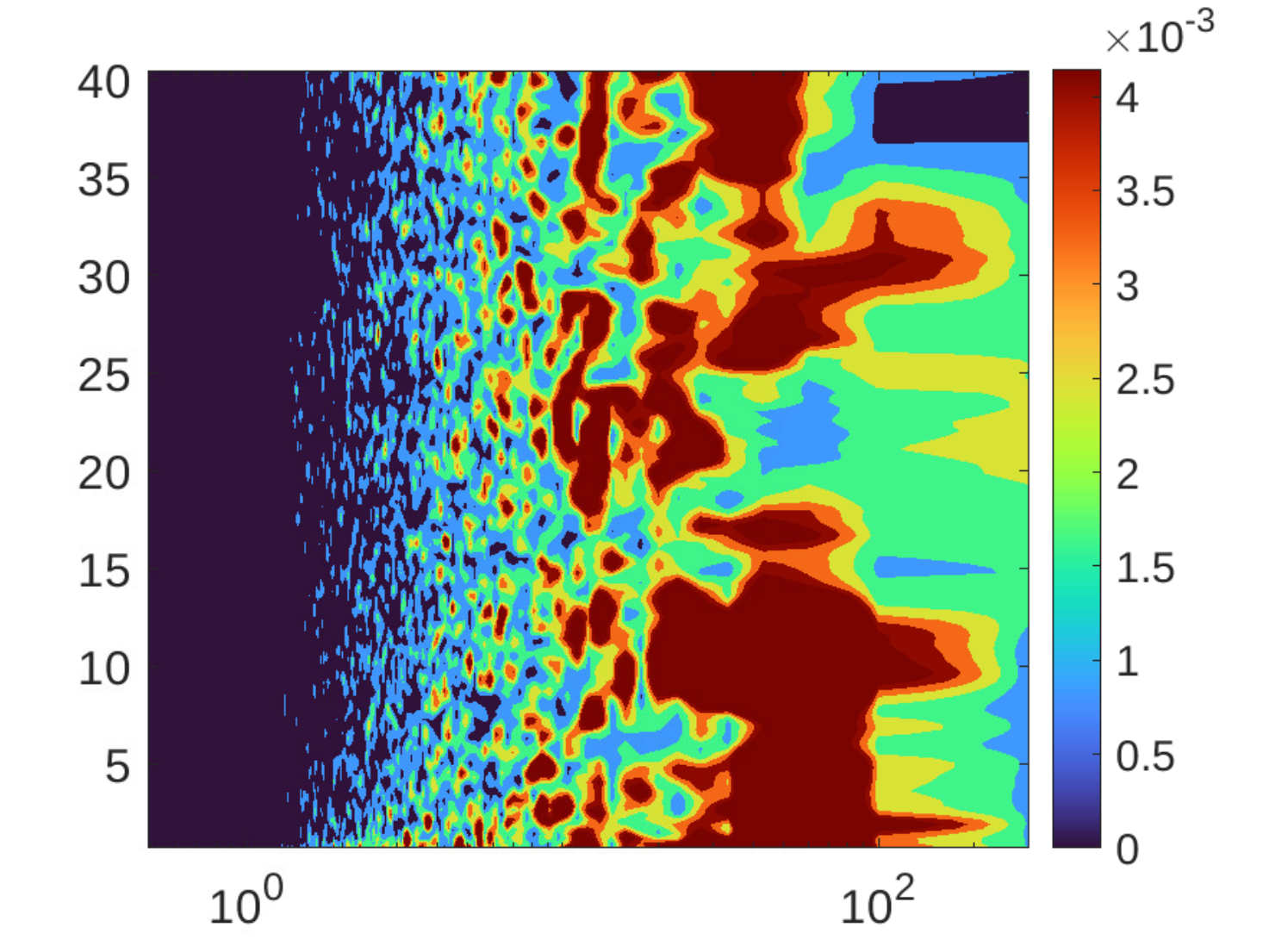}
\put (-6,22) {\small\rotatebox{90}{$y$ (mm)}}
		\put (22,-4) {\small{wavenumber}}
		\put(30,73) {\small{EDMD}}
   \end{overpic}
 \end{minipage}

\vspace{-2mm}
	  \caption{Top row: Horizontal velocity profiles predicted at 4s. Bottom row: Wavenumber spectra.}\vspace{-3mm}
\label{fig:turbulence2}
\end{figure}


\section{Conclusion}
\label{sec:conclusion}

We formulated a structure-preserving data-driven approximation of Koopman operators for measure-preserving dynamical systems, \texttt{mpEDMD}, summarized in~\cref{alg:mp_EDMD}. We proved the convergence of \texttt{mpEDMD} to various infinite-dimensional spectral quantities of interest, summarized in~\cref{approx_table}. In particular, \texttt{mpEDMD} is the first truncation method whose eigendecomposition converges to these spectral quantities for general measure-preserving dynamical systems. We also proved the first results on convergence rates of the approximation in the size of the dictionary. As well as the convergence theory, our numerical examples show the increased robustness of \texttt{mpEDMD} to noise compared with other DMD-type methods, and the ability to capture energy conservation and statistics of a real-world turbulent boundary layer flow. These results open the door to future extensions to more general structure-preserving methods for Koopman operators and data-driven dynamical systems.\vspace{-2mm}

\renewcommand{\baselinestretch}{0.95}
\bibliographystyle{siamplain}
\bibliography{koopman}
\end{document}